\documentclass[12pt,oneside]{amsart}

\title{a topological origin of quantum symmetric pairs}
\author{T.A.N. Weelinck} \address{School of Mathematics\\University of Edinburgh\\ Edinburgh, UK}\email{T.A.N.Weelinck@ed.ac.uk}
\usepackage{amsmath,amssymb,amsthm, adjustbox}
\usepackage{verbatim}
\usepackage{enumitem} 
\usepackage{tikz,tikz-cd}
\usepackage{stmaryrd} 
\usepackage{fullpage} 
\usepackage{relsize}
\usepackage{scalerel}
\usepackage{braids} 
\usetikzlibrary{knots} 
\usetikzlibrary{arrows} 
\usepackage{caption} 
\usepackage{color} 
\usepackage{enumitem} 
\usepackage{bbm} 
\usepackage{caption} 

\usepackage{hyperref}
\hypersetup{linktocpage}
\usepackage{xcolor}
\hypersetup{
    colorlinks,
    linkcolor={blue!80!black},
    citecolor={blue!50!black},
    urlcolor={blue!80!black}
}

\let\oldtocsection=\tocsection
 
\let\oldtocsubsection=\tocsubsection
  
\renewcommand{\tocsection}[2]{\hspace{0em}\oldtocsection{#1}{#2}}
\renewcommand{\tocsubsection}[2]{\hspace{2em}\oldtocsubsection{#1}{#2}}

\theoremstyle{plain}
\newtheorem{thm}{Theorem}[section]
\numberwithin{equation}{thm}
\newtheorem{prop}[thm]{Proposition}
\newtheorem{lem}[thm]{Lemma}
\newtheorem{cor}[thm]{Corollary}
\newtheorem{defprop}[thm]{Definition-Proposition}

\theoremstyle{definition}
\newtheorem{defn}[thm]{Definition}
\newtheorem{ex}[thm]{Example}

\newtheorem{rmk}[thm]{Remark}

\theoremstyle{remark}

\newtheorem*{conv}{Convention}
\newtheorem*{nota}{Notation}

\newcommand{\R}{\mathbb{R}}

\newcommand{\Z}{\mathbb{Z}}

\newcommand{\kk}{\mathbbm{k}}

\newcommand{\vect}{\mathbf{Vect}}
\newcommand{\fdvect}{\vect_{\mathrm{f.d.}}}

\newcommand{\catC}{\mathcal{C}}
\newcommand{\catA}{\mathcal{A}}
\newcommand{\catB}{\mathcal{B}}

\newcommand{\catD}{\mathcal{D}}
\newcommand{\DD}{\mathbb{D}}
\newcommand{\catM}{\mathcal{M}}
\newcommand{\catO}{\mathcal{O}}
\newcommand{\Rex}{\mathbf{Rex}}

\newcommand{\act}{\mathrm{act}}

\newcommand{\ztd}{\Z_2\mathcal{D}}

\DeclareSymbolFont{bbold}{U}{bbold}{m}{n}
\DeclareSymbolFontAlphabet{\mathbbold}{bbold}
\newcommand{\1}{\mathbbold{1}} 

\newcommand{\bcyl}{B^{\mathrm{cyl}}}
\newcommand{\g}{\mathfrak{g}}

\newcommand{\gl}{\mathfrak{g}\mathfrak{l}}

\newcommand{\U}{\mathcal{U}}

\newcommand{\B}{\mathcal{B}}

\newcommand{\rhob}{\overline{R}_2}
\newcommand{\lambdab}{\overline{L}_2}
\newcommand{\bL}{\overline{L}}
\newcommand{\bR}{\overline{R}}
\newcommand{\bep}{\overline{\ep}}
\newcommand{\etab}{\overline{\eta}}

\newcommand{\uqg}{\U_q(\g)}
\newcommand{\uqgt}{\U_{(q,\textbf{c})} (\g^{\theta})}

\newcommand{\oqg}{\mathcal{O}_q(G)}

\newcommand{\tensor}{\otimes}
\newcommand{\tensop}{\otimes\text{-op}}
\newcommand{\boxt}{\boxtimes}
\newcommand{\hcomp}{\; * \;}

\newcommand{\ep}{\varepsilon}
 	
\newcommand{\Hom}{\mathrm{Hom}}

\newcommand{\Map}{\mathrm{Map}}
\newcommand{\Fun}{\mathrm{Fun}}

\newcommand{\id}{\mathrm{id}}

\newcommand{\fdmod}{\text{-}\mathrm{mod}_{\mathrm{f.d.}}}
\newcommand{\Mod}{\text{-}\mathrm{mod}}

\newcommand{\tr}{\mathrm{tr}}
\newcommand{\nerve}{\mathrm{N}}
\newcommand{\Fin}{\mathrm{Fin}}
\newcommand{\pt}{\textrm{pt}}

\begin{document}

\begin{abstract}
It is well known that braided monoidal categories are the categorical algebras of the little two-dimensional disks operad.
We introduce involutive little disks operads, which are $\Z_2$-orbifold versions of the little disks operads. 
We classify their categorical algebras and describe these explicitly in terms of a finite list of functors, natural isomorphisms and coherence equations.
In dimension two, the categorical algebras are braided monoidal categories with an anti-involution together with a pointed module category carrying a universal solution to the (twisted) reflection equation. 
Main examples are obtained from the categories of representations of a ribbon Hopf algebra with an involution and a quasi-triangular coideal subalgebra, such as a quantum group and a quantum symmetric pair coideal subalgebra.
\end{abstract}

\maketitle
\tableofcontents

\section{Introduction}

In this paper we introduce and study the \emph{operad of involutive little disks}. We show that its operations naturally encode the symmetries of quantum groups and quantum symmetric pairs. 
Quantum groups are of course well-studied and known to have myriad  connections to other parts of mathematics. 
Only in recent years, however, have quantum symmetric pairs come to the forefront of research on quantum groups. 

A quantum symmetric pair consists of a quantum group together with a subalgebra that quantizes the subgroup of fixed-points of some given involution.
In the 90s, examples of quantum symmetric pairs were constructed on a case by case basis by M. Noumi, T. Sugitani and M. Dijkhuizen using solutions to the \emph{reflection equation} \cite{noumi,nousu95,nds}. 
This equation, sometimes called the `boundary Yang-Baxter equation', was introduced by I. Cherednik in the context of particle scattering on a half line \cite{cherednik84} and E. Sklyanin's study of quantum integrable systems with non-periodic boundary conditions \cite{sklyanin88,sklyanin92}.
The motivation of the work of Noumi-Sugitani-Dijkhuizen was to study $q$-orthogonal polynomials (e.g. Askey-Wilson polynomials \cite{askeywilson}, Macdonald polynomials \cite{macdonald01} and Koornwinder polynomials \cite{koornwinderpoly}) as zonal spherical functions of quantum symmetric spaces; this field is sometimes called `quantum harmonic analysis' \cite{koornwinder89,koornwinder91,dijkhuizen96}.
A complete classification of quantum symmetric pairs was later achieved by G. Letzter \cite{letzter99,letzter02,letzter03}. 
S. Kolb then extended the theory of quantum symmetric pairs to include Kac-Moody algebras, with examples such as twisted $q$-Yangians and $q$-Onsager algebras \cite{kolb14}.

In recent years the reflection equations appeared in a more methodical way in the theory of quantum symmetric pairs.
S. Kolb and J. Stokman showed that an invertible solution to a type $A$ reflection equation is exactly a character of the reflection equation algebra $\mathcal{O}_q (\mathrm{SL}_n)$. From such a character they construct a quantum symmetric pair, streamlining the approach of Noumi-Sugitani-Dijkhuizen \cite{kolbstokman}. 
Furthermore, M. Balagovic and S. Kolb showed that any pair in Letzter's classification carries a canonical solution to a reflection equation, called a \emph{universal $K$-matrix} \cite{balakolb16,kolb17}. 
Such solutions can be used for applications to low-dimensional topology in the spirit of the celebrated Reshetikhin-Turaev knot invariants \cite{td98,tdho98}. 

A categorical approach to the reflection equations, as well as Letzter's classification of quantum symmetric pairs, allowed for applications to geometric representation theory.
D. Jordan and X. Ma used quantum symmetric pairs to construct representations of the double affine braid group and of the double affine Hecke algebra of type $C^{\vee}C_n$ \cite{jordanma}.
M. Ehrig and C. Stroppel showed that translation functors in parabolic category $\mathcal{O}$ in type $D$ categorify the action of a quantum symmetric pair \cite{es13}. 
Work by H. Bao and W. Wang on type $AIII/AIV$ quantum symmetric pairs, see Table \ref{table:araki}, shows such quantum symmetric pairs admit a canonical basis \cite{bw16}. Furthermore, they have a geometric interpretation through the geometry of partial flag varieties of type B/C \cite{bsww16}, and are Schur-Weyl dual to Hecke algebras of type B \cite{bww16}.

One of the fundamental properties of the Drinfeld-Jimbo quantum group is that it is a quasi-triangular Hopf algebra: it has a universal $R$-matrix providing solutions to the Yang-Baxter equation.
We can rephrase this by saying that the category of finite dimensional modules of a quantum group is a braided monoidal category.
In this paper we introduce a categorical framework, called a $\Z_2$-braided pair, which exactly encodes the structure of the universal $K$-matrix of a quantum symmetric pair; analogous to how a braided monoidal category encodes the $R$-matrix.

It is well known that braided monoidal categories are exactly the categorical algebras over the so-called \emph{$E_2$-operad} \cite{fiedorowicz,dunn97,wahl01}.\footnote{See \cite[Example 5.1.2.4]{lha} for a modern version of this result phrased in the terminology of higher categories and operads.}
Recall that the $E_2$-operad, or operad of little two-dimensional disks, is a topological operad whose operations are parametrized by the different embeddings of a disjoint union of disks into a larger disk.
In the identification of categorical $E_2$-algebras with braided monoidal categories the braiding (i.e. the universal $R$-matrix) is interpreted as an operation rotating two disks in the plane.
In this paper we propose an interpretation for the universal $K$-matrix of Balagovic and Kolb in terms of $\Z_2$-orbifold disks rotating around each other, see Figure \ref{fig:tre}.
This naturally leads us to consider an operadic perspective on quantum symmetric pairs in which the $E_2$-operad is replaced by an $\Z_2$-orbifold analogue.
We therefore introduce the $\ztd_2$-operad of involutive little two-dimensional disks, and study its categorical algebras.
Our main result is that $\Z_2$-braided pairs are exactly the categorical $\ztd_2$-algebras.

Recall that an $E_2$-algebra $A$ can be used to construct invariants of surfaces by computing the so-called \emph{factorization homology of the surface with coefficients in $A$} \cite{lha,af15}.
In \cite{bzbj,bzbj2}, the authors compute the surface invariants of factorization homology with coefficients in the braided monoidal category of quantum group representations. 
These invariants produce quantizations of the \emph{character variety} of the surfaces  (the moduli space of local systems on that surface). 
In \S \ref{intro:outlook} we explain forthcoming work in which we compute the factorization homology of $\Z_2$-orbifold surfaces with coefficients in the $\Z_2$-braided pair of representations of a quantum symmetric pair.
Thereby, we will construct invariants of surfaces $\Sigma$ with an involution and actions of the associated orbifold braid groups $B_n[\Sigma/\Z_2]$.
We expect that these invariants give quantizations of the character varieties of orbifold surfaces with isolated $\Z_2$-singularities.

\subsection{Summary of results}
Central to our results is the notion of a $\Z_2$-braided pair. A $\Z_2$-braided pairs consists of a braided monoidal category, a module category and some extra structure which we now recount.
See \S\ref{sec:ztd2} for the definitions of braided monoidal category and module category.
\begin{defn} (Definition \ref{z2pair}) \label{intro:pdef}
	\begin{enumerate}
		\item A \emph{$\Z_2$-monoidal pair} consists of a monoidal category $\catA$ endowed with an anti-monoidal involution $\Phi: \catA \rightarrow \catA^{\tensop}$, t: $\Phi^2 \simeq \id_\catA$, together with a module category $\catM$ over $\catA$ and a distinguished object $\1_\catM \in \catM$. 
		\item A \emph{$\Z_2$-braided pair} consists of a $\Z_2$-monoidal pair  $(\catA,\catM)$ together with a braiding $\sigma: \tensor \simeq \tensor^{op}$ on $\catA$ such that $\Phi$ is braided, as well as a family of natural isomorphisms $\kappa_{M,X}: M\tensor X \rightarrow M \tensor \Phi(X)$ for all $M \in \catM, X \in \catA$ satisfying the coherence diagrams of Figures \ref{fig:bp1} and \ref{fig:bp2}.
		\item A \emph{$\Z_2$-symmetric pair} is a $\Z_2$-braided pair such that $\sigma^2 = \id$ and $\kappa^2 = \id$.
	\end{enumerate}
\end{defn}
We will call the natural isomorphism $\kappa$ the \emph{$\Z_2$-cylinder braiding} as its topological interpretation is the braid around the singular point in the orbifold cylinder $[D^2/\Z_2]$ (see Figure \ref{fig:z2monodromy}).  
Crucially, Definition \ref{intro:pdef} implies that the $\Z_2$-cylinder braiding $\kappa$ and the braiding $\sigma$ satisfy the $\Phi$-twisted reflection equation \eqref{eq:tre}. 
\begin{rmk}
	The notion of a $\Z_2$-braided pair is very close to the notion of a $\Phi$-braided module category of \cite[Remark 3.15]{kolb17} and braided module category of \cite{brochier13}. 
	However, in the definition of $\Z_2$-braided pair we had to impose additional requirements to match the topology of the $\ztd_2$-operad. For a precise comparison see Remarks \ref{rmk:comparison1} and \ref{rmk:comparison2}. 
\end{rmk}
Our main result interprets the axioms in Definition \ref{intro:pdef} via topological operads $\ztd_n$.
These operads have two colours $\DD$ and $\DD_*$, that represent the free quotient $(D^n \amalg D^n) /\Z_2$ where $\Z_2$ swaps the disks, and the orbifold quotient $D^n/\Z_2$ where $\Z_2$ acts by rotations (see  Figure \ref{fig:danddstar}).
The operations in $\ztd_n$ are parametrized by equivariant embeddings of disjoint unions of the orbifold disks $\DD$ and $\DD_*$ into each other.
Consider a categorical $\ztd_n$-algebra $(\catA,\catM)$, where $\catA$ corresponds to $\DD$ and $\catM$ corresponds to $\DD_*$. 
Disjoint unions of the disk $\DD$ embed into $\DD$, which determines a multiplication structure on $\catA$. Moreover, $\DD$ embeds into $\DD_*$, but not conversely. This determines an $\catA$-module structure on $\catM$. Our main result is the following classification, relating the categorical $\ztd_n$-algebras to Definition \ref{intro:pdef}.
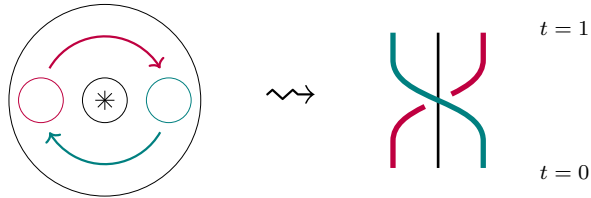
\begin{figure}[h]
	$\vcenter{\hbox{\adjustbox{scale=.85}{\begin{tikzpicture}
				\draw[black] (-.15,0) -- (0.15,0);
				\draw[black] (0,-.15) -- (0,.15);
				\draw[black] (-.11,-.11) -- (0.11,0.11);
				\draw[black] (0.11,-.11) -- (-0.11,0.11);
				\draw[teal] (1,0) circle [radius=.35];
				\draw[purple] (-1,0) circle [radius=.35];
				\draw[black] (0,0) circle [radius = 1.5];
				\draw[black] (0,0) circle [radius =.35];
				\draw [->,line width=1pt,teal] ([shift=(-30:1)]0,0) arc (-30:-150:1);
				\draw [->,line width=1pt,purple] ([shift=(150:1)]0,0) arc (150:30:1);
				\end{tikzpicture}}}}\quad  \quad  \mathlarger{\mathlarger{\mathlarger{\mathlarger{\leadsto}}}} \quad \quad
	\vcenter{\hbox{
			\begin{tikzpicture}[scale=1.2]
			\draw[black,line width=1pt] (1.5,-1.5) -- (1.5,0);
			\braid[line width = 2pt, 
			style strands={1}{teal}, 
			style strands={2}{purple}]
			s_1;
			\node at (2.9,-1.55) {\tiny $t=0$};
			\node at (2.9,0.05) {\tiny $t=1$};
			\end{tikzpicture}}}$
	\caption{(Left) The isotopy of operations in $\ztd_2$ corresponding to the $\Z_2$-cylinder braiding $\kappa$. The operations are embeddings of the form $\DD \sqcup \DD_* \rightarrow \DD_*$. 
		(Right) A coloured braid encoding the isotopy up to homotopy.}
	\label{fig:z2monodromy}
\end{figure}
\begin{thm} (Theorem \ref{thm:equivalences}) \label{intro:maintheorem} Let $\mathbf{Cat}$ denote the 2-category of categories. 
	\begin{enumerate} 
		\item A $\ztd_1$-algebra in $\mathbf{Cat}$ is exactly a $\Z_2$-monoidal pair.
		\item A $\ztd_2$-algebra in $\mathbf{Cat}$ is exactly a $\Z_2$-braided pair. 
		\item Let $n\geq3$. A $\ztd_n$-algebra in $\mathbf{Cat}$ is exactly a $\Z_2$-symmetric pair. 
	\end{enumerate}
\end{thm}

\begin{rmk}
	We stated the main theorem in terms of $\mathbf{Cat}$ for simplicity.
	We prove the main theorem for $\kk$-linear categorical algebras since $\kk$-linear categories provide a more natural setting for representation theory, see e.g. \cite{egno15}.
	However, our proofs are readily adaptable to \textbf{Cat}, where they yield the above result.
\end{rmk}

\begin{conv}
	We refer to plain categories with a tensor product as \emph{monoidal categories}, whereas we refer to $\kk$-linear categories with a tensor product as \emph{tensor categories}.\footnote{To be precise we moreover assume a tensor category is finitely cocomplete, and that the tensor product functor is $\kk$-bilinear and preserves finite colimits in each variable, see Section \ref{subsec:tensorcat}.}
\end{conv}

\begin{rmk}
	We use $\infty$-categories to define $\ztd_n$-algebras in general. 
	Since \textbf{Cat} is a only 2-category we will delay discussing the $\infty$-categorical foundations to \S \ref{sec:genalg}. 
	The reader only interested in applications to representation theory can safely avoid reading that section. 
\end{rmk}

\begin{rmk} Recall that for $n\geq 3$ a categorical $E_n$-algebra is the same thing as a categorical $E_\infty$-algebra (i.e. a symmetric monoidal category). 
	This is an instance of the Baez-Dolan stabilisation hypothesis \cite{bd95}.
	As one would expect from the stabilisation hypothesis, categorical $\ztd_n$-algebras also stabilise when $n\geq 3$. 
\end{rmk}

All our examples of $\Z_2$-braided pairs are obtained from the theory of Hopf algebras. 
Namely, the categories of representations of  quasi-triangular Hopf algebras provide important examples of braided monoidal categories. 
A standard source of module categories are then the coideal subalgebras of the Hopf algebra. Recall that a \emph{coideal subalgebra} $B$ of a Hopf algebra $H$ is a subalgebra $B \subset H$ for which $\Delta(B) \subset B \tensor H$. 
Following \cite{kolb17}, we call a coideal subalgebra \emph{quasi-triangular} if it has a universal $K$-matrix, see Definition \ref{qtqsp}.

\begin{prop} \label{intro:examples} (Proposition \ref{prop:ztd2ex} see also \cite[Remark 4.13]{balakolb16} and \cite[\S 2.2]{kolb17}.)
	Let $H$ be a ribbon Hopf algebra with involution $\phi: H \rightarrow H$ and let $B$ be a quasi-triangular coideal subalgebra of $H$. 
	Then $(H\fdmod,B\fdmod)$ naturally defines a $\Z_2$-braided pair. 
\end{prop}
The main results \cite[Corollary 9.6]{balakolb16} and \cite[Corollary 3.14]{kolb17} state that the quantum symmetric pair coideal subalgebra is quasi-triangular. As a corollary we then have:
\begin{cor}
	Any quantum symmetric pair together with the standard ribbon element of the quantum group satisfies the hypothesis of Proposition \ref{intro:examples}.
\end{cor} 

In order to prove Theorem \ref{intro:maintheorem} we establish various coherence theorems which may be of independent interest. 
A priori, to specify a categorical $\ztd_n$-algebra one needs to provide an infinite amount of data, corresponding to the infinitely many operations in the operad $\ztd_n$, the isotopies between these operations, the isotopies between isotopies, and so on. 
Theorem \ref{intro:maintheorem} states all this data can be reduced to the finite list of functors, isomorphisms and equations in Definition \ref{intro:pdef}.
Concretely, one needs to show that all the functors and natural isomorphisms one can construct using these generators make the correct diagrams commute, as specified by the operad $\ztd_n$.
This is exactly the content of a coherence theorem. We can state our coherence results informally as follows:
\begin{thm} \label{intro:coherence} 
	Let $(\catA,\catM)$ be a $\Z_2$-monoidal/braided/symmetric pair. By a coherence diagram we mean a diagram in $\catA$ or $\catM$ constructed using the functors and natural isomorphisms that are part of Definition \ref{intro:pdef}.
	\begin{enumerate}
		\item Any coherence diagram in a $\Z_2$-monoidal pair commutes. (Theorem \ref{moncoh})
		\item A coherence diagram in a $\Z_2$-braided pair commutes if the underlying braids agree.\footnote{The underlying braid of a natural isomorphism is obtained by interpreting instances of the braiding $\sigma$ and $\kappa$ as generators of the cylinder braid group $B_n^{\mathrm{cyl}}$ (Definition \ref{prop:bcyl}).} (Theorem \ref{braidcoh})
		\item Any coherence diagram in a $\Z_2$-symmetric pair commutes. (Theorem \ref{symcoh})
	\end{enumerate} 
\end{thm}

\subsection{Quantum symmetric pairs and the reflection equations} \label{intro:qsp}
We now provide some further background  on quantum symmetric pairs and the reflection equations.

An infinitesimal symmetric pair $(\g,\g^{\theta})$ consists of a complex semisimple Lie algebra $\g$ and a sub Lie algebra $\g^{\theta} \subset \g$ fixed by some involutive automorphism $\theta: \g \rightarrow \g$. Irreducible symmetric pairs\footnote{A symmetric pair is called irreducible if it cannot be obtained as the symmetric pair associated to a non-trivial product of two symmetric pairs.} were classified by S. Araki \cite{araki62}. The classification of type $A$ symmetric pairs is recorded in Table \ref{table:araki}. 
\begin{table}
	\caption{The classification of irreducible symmetric pairs in type A. Here $w_0$ denotes the involution induced by the longest word in the Weyl group.}
	\begin{tabular}{| c | c | c |}
		\hline
		Type & $(\g,\g^{\theta})$ & Twist $\phi$ \\ \hline
		\textbf{AI} & $(\mathfrak{sl}_n,\mathfrak{so}_n)$ & $w_0$\\ \hline
		\textbf{AII} & $(\mathfrak{sl}_{2n},\mathfrak{sp}_{2n})$& $w_0$\\ \hline 
		\textbf{AIII/AIV} & $(\mathfrak{sl}_{n+m}, \mathfrak{sl}_{n+m} \cap (\mathfrak{gl}_n \oplus \mathfrak{gl}_{m}) )$ & $\id$ \\ \hline
	\end{tabular}
	\label{table:araki}
\end{table}

A quantum symmetric pair $(\uqg,\uqgt)$ should then be a quantization of the pair $(\U(\g), \U(\g^{\theta}))$.
\begin{nota}
	Notation for quantum symmetric pairs varies in the literature.
	Some authors denote the coideal subalgebra by $\B_{\mathbf{c}}$. 
	The $c$ denotes a multi-parameter as there can be multi-parameter families of quantizations of $U(\g)$. 
\end{nota} 
It turns out that requiring $\uqgt$ to be a sub-Hopf algebra of $\uqg$ is too strong a requirement, rather this quantization will be a (left) coideal subalgebra i.e. a subalgebra satisfying $\Delta(\uqgt) \subset \uqgt \tensor \uqg$.

\subsubsection{The reflection equation and cylinder braids}
The first constructions of quantum symmetric pairs in \cite{nousu95,nds,dijkhuizen96} depended crucially on solutions to various versions of the reflection equation:
\begin{align} \label{ure}K_{1} \;R_{21} \; K_{2} \; R_{12} \quad = \quad R_{21} \; K_{2} \; R_{12} \; K_{1},
\end{align}
where $R$ denotes a given solution of the Yang-Baxter equation.
Like the Yang-Baxter equation, the reflection equation has an interpretation in low-dimensional topology, see Figure \ref{fig:ure}. Similar to the Reshetikhin-Turaev invariants, one can use solutions of the reflection equation to construct knot invariants \cite{td98,tdho98}. 

\begin{figure}[h!]
	\begin{center}
		$\vcenter{\hbox{\adjustbox{scale=.5}{
					\begin{tikzpicture}
					\braid[style strands={1}{white}, style strands={2}{black},style strands={3}{black}, line width = 1.3pt]
					s_1 s_1 s_2 s_1 s_1 s_2;
					\draw (1.5,0) -- (1.5,-6.5);
					\draw [draw=black,dotted] (0.5,0) rectangle (2.5,-2);
					\node at (-.5,-1) {\huge $K$};
					\draw [draw=black,dotted] (1.75,-2.25) rectangle (3.25,-3.25);
					\node at (3.75,-2.75) {\huge $R$};
					\draw [draw=black,dotted] (0.5,-3.5) rectangle (2.5,-5);
					\node at (-.5,-4.25) {\huge $K$};
					\draw [draw=black,dotted] (1.75,-5.25) rectangle (3.25,-6.25);
					\node at (3.75,-5.75) {\huge $R$};
					\end{tikzpicture}}}}
		\quad \quad \simeq \quad \quad
		\vcenter{\hbox{\adjustbox{scale=.5}{ 
					\begin{tikzpicture}
					\braid[style strands={1}{white}, style strands={2}{black},style strands={3}{black}, line width = 1.3pt]
					s_2 s_1 s_1 s_2 s_1 s_1;
					\draw (1.5,0) -- (1.5,-6.5);
					\draw [draw=black,dotted] (0.5,-1.5) rectangle (2.5,-3);
					\node at (-.5,-2.25) {\huge $K$};
					\draw [draw=black,dotted] (1.75,-0.25) rectangle (3.25,-1.25);
					\node at (3.75,-.75) {\huge $R$};
					\draw [draw=black,dotted] (0.5,-4.5) rectangle (2.5,-6);
					\node at (-.5,-5.25) {\huge $K$};
					\draw [draw=black,dotted] (1.75,-3.25) rectangle (3.25,-4.25);
					\node at (3.75,-3.75) {\huge $R$};
					\end{tikzpicture}}}}$
		\vspace{5pt}
		\caption{The reflection equation interpreted via cylinder braids.}
		\label{fig:ure}
	\end{center}
\end{figure}
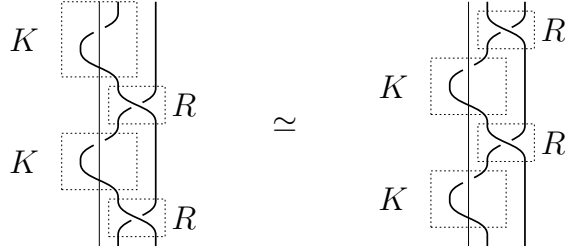
Solutions to the $\gl_n$-reflection equation were studied by J. Donin, A. Mudrov and P.P. Kulish \cite{dkm03} and completely classified by A. Mudrov \cite{mudrov02}. 
It turns out such solutions can be viewed as characters of the \emph{reflection equation algebra} $\oqg$.
In \cite{kolbstokman}, characters of the reflection equation algebra are used to reconstruct the type AIII/AIV quantum symmetric pairs of Noumi-Sugitani-Dijkhuizen. 
\begin{rmk} \label{oqg} The reflection equation algebra $\oqg$ is an (equivariant) quantization of the Semenov-Tian-Shansky Poisson bracket on $\catO (G)$ \cite{mudrov06}. This algebra goes by many names: Majid's \emph{braided dual of} $\uqg$ \cite{majid}, the \emph{quantum-loop algebra} \cite{alekseev96} and is isomorphic to the locally finite part of $\uqg$ via the Rosso form \cite[Proposition 2.8]{kolbstokman}. 
\end{rmk} 

\subsubsection{The reflection equations revisited}

Recall that various versions of the reflection equation appeared in the works \cite{noumi,nousu95,nds,dijkhuizen96}.
It was only later realised by M. Balagovic and S. Kolb, through their algebraic construction of universal $K$-matrices, that there is a general framework of \emph{twisted reflection equations} which unifies all different reflection equations \cite[Remark 9.7]{balakolb16}. 
For a quasi-triangular Hopf algebra $H$ with coideal subalgebra $B$ one fixes an additional Hopf algebra involution $\phi$ of $H$, such that $(\phi\tensor \phi)(R) = R$, called \emph{the twist}. A ($\phi$-twisted) universal $K$-matrix is a universal solution to the $\phi$-twisted reflection equation:
\begin{align} \label{eq:tre}
	K_{1} \; R_{21}^{\phi} \; K_{2} \; R_{12} \quad = \quad R_{21}^{\phi,\phi} \; K_{2} \; R_{12}^{\phi} \; K_{1},
\end{align}
where $R^\phi = (\phi \tensor 1)(R)$ and $R^{\phi,\phi} = (\phi\tensor \phi)(R) = R$.
\begin{rmk} The twist $\phi$ for a given quantum symmetric pair is determined by the Dynkin data that characterises it in Araki's classification \cite{araki62}. See Table \ref{table:araki} for twists in type A. 
\end{rmk} 
\begin{rmk} Quantum symmetric pairs for which the twist is trivial, $\phi = \id$, have universal $K$-matrices that provide solutions to the untwisted reflection Equation \eqref{ure}. 
	This explains why in \cite{kolbstokman} characters of $\oqg$ could only be used to construct quantum symmetric pairs of type $AIII/AIV$: such characters solve the untwisted reflection equation.
\end{rmk}
The twist is naturally built into the categorical framework of $\Z_2$-braided pairs via the involution $\Phi$. Recall from Figure \ref{fig:z2monodromy} that we can interpret the $\Z_2$-monodromy in the operad $\ztd_2$ as a coloured braid.
The twisted reflection equation is then naturally interpreted in the operad $\ztd_2$ as an isotopy of coloured braids, see Figure \ref{fig:tre}.
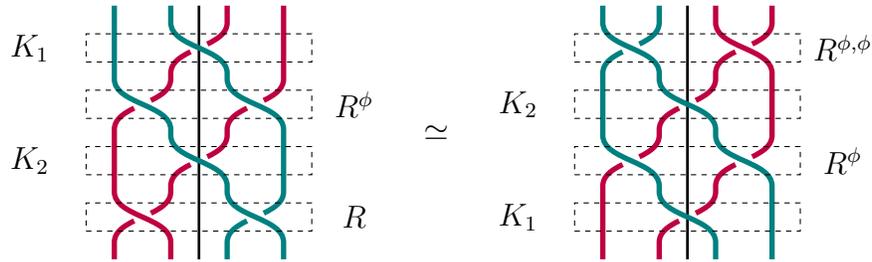
\begin{figure}[h] 
	\[ \vcenter{\hbox{\begin{tikzpicture}[scale=.75]
			\draw[black,line width=1pt] (2.5,-4.5) -- (2.5,0);
			\braid[line width = 2pt, 
			style strands={1}{teal}, 
			style strands={2}{teal},
			style strands={3}{purple},
			style strands={4}{purple}]
			s_2 s_1-s_3 s_2 s_1-s_3;
			\draw [draw=black,dotted] (0.5,-.5) rectangle (4.5,-1);
			\node at (-.5,-.75) {$K_1$};
			\draw [draw=black,dotted] (0.5,-1.5) rectangle (4.5,-2);
			\node at (5.25,-1.75) { $R^\phi$};
			\draw [draw=black,dotted] (0.5,-2.5) rectangle (4.5,-3);
			\node at (-.5,-2.75) {$K_2$};
			\draw [draw=black,dotted] (0.5,-3.5) rectangle (4.5,-4);
			\node at (5.25,-3.75) { $R$};
			\end{tikzpicture}}} \quad \simeq \quad \vcenter{\hbox{\begin{tikzpicture}[scale=.75]
			\draw[black, line width = 1pt] (2.5,-4.5) -- (2.5,0);
			\braid[line width = 2pt, 
			style strands={1}{teal}, 
			style strands={2}{teal},
			style strands={3}{purple},
			style strands={4}{purple}]
			s_1-s_3 s_2 s_1-s_3 s_2;
			\draw [draw=black,dotted] (0.5,-1.5) rectangle (4.5,-2);
			\node at (-.5,-1.75) {$K_2$};
			\draw [draw=black,dotted] (0.5,-.5) rectangle (4.5,-1);
			\node at (5.25,-.75) { $R^{\phi,\phi}$};
			\draw [draw=black,dotted] (0.5,-3.5) rectangle (4.5,-4);
			\node at (-.5,-3.75) {$K_1$};
			\draw [draw=black,dotted] (0.5,-2.5) rectangle (4.5,-3);
			\node at (5.25,-2.75) { $R^{\phi}$};
			
			\end{tikzpicture}}}  
	\]
	\caption{The $\phi$-twisted reflection equation interpreted via coloured braids.}
	\label{fig:tre}
\end{figure}
Therefore, we can now recognise that the twists $\phi$, which had an algebraic origin but lacked intrinsic topological meaning, are an integral part of the topological interpretation of quantum symmetric pairs. 

\subsection{Outlook: factorization homology and quantum symmetric pairs} \label{intro:outlook} The constructions in this paper are the first step in a program where quantum symmetric pairs are used for applications to low-dimensional topology. 
The next steps will involve (developing and) applying a general framework called \emph{factorization homology}.

Factorization homology was first introduced by J. Lurie \cite{lurietft,lha} as a topological variant of the chiral homology of Beilinson-Drinfeld \cite{bd04}, and was further developed by J. Francis, D. Ayala and H. Tanaka \cite{af15, aft17}.
Nowadays there is a large body of literature on factorization homology and related ideas, see e.g. \cite{andrade10,gwilliam12, gtz12,mw12,cg17,horel17}. See \cite{ginot15} for a survey.
The formalism of factorization homology, in its simplest form, takes as input an $E_n$-algebra $\catA$ in a symmetric monoidal higher category $\catC$ and associates to every framed $n$-manifold $M$ an invariant denoted $\int_M \catA$. 
We will informally explain the assignment. 
Following \cite{af15}, we view the manifold as glued from all the framed disks embedding into $M$, and view the $E_n$-algebra $\catA$ as a functor $\catA: \catD isk^{\mathrm{fr}}_n \rightarrow \catC$ from some higher category of framed disks into $\catC$.
The invariant is then defined as follows: 
\begin{align} \label{intro:fhcolim}
	\int_M \catA :=  \underset{D^n \subset M}{\mathrm{colim}} \; \catA(D^n) 
\end{align}
We record some key properties of factorization homology:
\begin{enumerate}
	\item Factorization homology is uniquely characterised by an excision property. After decomposing a manifold along a collared boundary, one can compute the global invariant of the manifold  in terms of the invariants of the pieces \cite[\S 3.5]{af15}.\footnote{The colimit in Equation \eqref{intro:fhcolim} is difficult to compute and typically computed via excision in practice.} 
	\item The invariant is functorial with respect to embeddings. In particular $\mathrm{Diff^{fr}}(M)$ acts on $\int_M \catA$. Moreover, for a categorical $E_n$-algebra $\catA$ and $X_1, \dots, X_n \in \catA$ the braid group $B_n(M)$ acts on an associated object $\int_M (X_1 \tensor \dots \tensor X_n) \in \int_M \catA$. 
	\item There are many versions of factorization homology e.g. for oriented manifolds and manifolds with singularities  \cite{af15,aft17}.
	\item For a given $E_n$-algebra $\catA$, the factorization homology of $n$-manifolds with coefficients in $\catA$ defines a fully extended $n$-dimensional TQFT \cite{scheimbauer}. 
\end{enumerate}
Most relevant to our work are the applications of factorization homology to representation theory via the factorization homology of braided tensor categories as developed in \cite{bzbj,bzbj2}. 
In \cite{bzbj,bzbj2}, the authors compute invariants of genus $g$ framed and oriented surfaces using factorization homology with coefficients in the braided tensor category of quantum group representations. 
For example, the invariant assigned to the annulus is the category of representations of the reflection equation algebra $\oqg$ and the invariant assigned to the torus is the category of strongly equivariant quantum $\mathcal{D}$-modules on $G$.
Via Key Property (2) the authors obtain braid group actions of oriented surfaces.
Such surface braid group actions were previously constructed by `generators and relations' methods in \cite{jordan09} and used to construct representations of the type A \emph{double affine Hecke algebra} (abbreviated DAHA). 
The type A DAHA, due to  I. Cherednik \cite{cherednik05}, is a quotient of the group algebra of the torus braid group $B_n(T)$ by additional Hecke relations.
The braid group actions constructed in \cite{bzbj,bzbj2} recover the surface braid group actions of \cite{jordan09} and provide an intrinsic topological explanation for their existence.

To connect these developments to our work on quantum symmetric pairs and the $\ztd_2$-operad we make the following observation. The colimit formula \eqref{intro:fhcolim} defining factorization homology is motivated by the fact that a $n$-manifold $M$ is covered by disks $D^n$. 
Correspondingly, a $\Z_2$-orbifold surface $[\Sigma/\Z_2]$ with isolated singularities is covered by the orbifold disks $\DD$ and $\DD_*$, which appear in the definition of the operad $\ztd_2$.
Hence it is natural to associate invariants to such orbifold surfaces constructed from a $\ztd_2$-algebra via a formula like Equation \eqref{intro:fhcolim}.
In the follow-up paper \cite{EFH} we make this idea precise by introducing $\Gamma$-equivariant factorization homology, where $\Gamma$ can be an arbitrary finite group. 

Recall that in \cite{jordanma} representations of the  type $C^\vee C_n$ DAHA were constructed using quantum symmetric pairs of type AIII/AIV. These constructions were made using a generators and relations approach alike \cite{jordan09}, and similarly lacked a topological interpretation. 
The DAHA of type $C^\vee C_n$ arises as a quotient of the orbifold surface braid group $B_n[T^2/\Z_2]$, where the $\Z_2$ action on the torus $T^2$ is induced by rotating the fundamental domain. 
In future work we hope to recover the DAHA representations in \cite{jordanma} from the braid group actions arising in the equivariant factorization homology of the orbifold torus $[T^2/\Z_2]$ with coefficients in a quantum symmetric pair. 	
Moreover, the general construction of equivariant factorization homology will allow us to immediately generalise the constructions in \cite{jordanma} to other quantum symmetric pairs, not necessarily of type AIII/AIV. 
This provides important motivation for our work on quantum symmetric pairs and factorization homology.
\subsection{Organisation}
The contents of this article are laid out as follows. 

In Section \ref{sec:operad} we introduce and study the operads $\ztd_n$.
We first recall the definition of orbifold configuration spaces and then define the operads $\ztd_n$.  We compute the homotopy type of the spaces of operations of the operads $\ztd_n$ in terms of orbifold configuration spaces.

In Section \ref{sec:ztd2} we introduce $\Z_2$-monoidal pairs, $\Z_2$-braided pairs and $\Z_2$-symmetric pairs. Moreover, we prove Proposition \ref{intro:examples}.

Section \ref{sec:coherence} is a technical section where we prove Theorem \ref{intro:coherence}. The reader uninterested in methods for proving a coherence result can safely skip the section. To prove the theorem we follow the `strictification implies coherence' approach of Joyal and Street \cite{js93}. 

Section \ref{sec:rexalg} is devoted to proving Theorem \ref{intro:maintheorem}, though two further results from Section \ref{sec:genalg} are needed to complete the proof. We begin by defining the 2-category $\Rex$ and the Deligne-Kelly tensor product.
We then use the coherence results of Section \ref{sec:coherence} to assign categorical $\ztd_n$-algebras for $n=1$, $n=2$, $n\geq 3$ to respectively $\Z_2$-monoidal, $\Z_2$-braided and $\Z_2$-symmetric pairs. We also construct assignments in the opposite direction.

Finally, in Section \ref{sec:genalg} we recall the definitions in \cite{lha} of algebras over $\infty$-operads. 
We show that in the special case of $\ztd_n$-algebras in $\Rex$ we recover the definition of categorical $\ztd_n$-algebras we gave in \S \ref{sec:rexalg}.
We conclude by finishing the proof of Theorem \ref{intro:maintheorem} by showing that the assignments constructed in \S \ref{sec:rexalg} are inverse equivalences.
\addtocontents{toc}{\protect\setcounter{tocdepth}{1}}
\subsection*{A note to the reader}
It is well known that categorical $E_2$-algebras are braided monoidal categories. For a proof in the strict setting see \cite{wahl01}, and for a proof using $\infty$-categories and Dunn Additivity see \cite{lha}. As there is no Dunn Additivity for $\ztd_2$-algebras we had to directly prove that categorical $\ztd_2$-algebras are $\Z_2$-braided pairs. A subset of our arguments can be used to give a direct proof that categorical $E_2$-algebras are braided monoidal categories. As we are unaware of such a proof in the literature, this may be of independent interest. 
\subsection*{Acknowledgements}
First and foremost, I wish to thank David Jordan for his invaluable guidance. 
I would also like to thank Stefan Kolb for patiently answering many questions about quantum symmetric pairs.
This work was supported by the Engineering and Physical Sciences Research Council [grant number 1633460].
\addtocontents{toc}{\protect\setcounter{tocdepth}{2}}

\section{The involutive little disks operad} \label{sec:operad}
\subsection{Orbifold configuration spaces} \label{colbraid}
Orbifold configuration spaces will be fundamental to our understanding of the $\ztd_n$-operads.
We now recall the definitions of configuration spaces for manifolds and global quotient orbifolds and their associated braid groups.
\begin{defn} Let $\Sigma$ be some surface. 
	\begin{enumerate}
		\item The \emph{configuration space of $k$ ordered points} in $\Sigma$ is denoted $F_n(\Sigma)$ and defined as
		\[F_k (\Sigma) := \big\{ (z_1,\dots,z_k) \in \Sigma^{\times k} \mid z_i \neq z_j \text{ if } i \neq j \big\}.\]
		The symmetric group $S_k$ acts freely on $F_k(\Sigma)$ by swapping points. 
		\item The \emph{configuration space of $k$ unordered points} in $\Sigma$ is denoted $C_k(\Sigma)$ and defined to be the quotient space $C_k(\Sigma) := F_k(\Sigma)/S_k$.
		\item The \emph{braid group on $n$ strands} of $\Sigma$ is denoted $B_k(\Sigma)$ and defined to be the fundamental group $ \pi_1 C_k(\Sigma)$.
	\end{enumerate}
\end{defn}
Recall that for a given group $\pi$ an \emph{Eilenberg-MacLane space of type $K(\pi,1)$} is a topological space $S$ that has trivial homotopy groups except $\pi_1(S) \cong \pi$.
\begin{ex} \label{confex} We recall the following classical examples of Eilenberg-MacLane spaces:
	\begin{enumerate}
		\item The space $C_k(\R^2)$ is a $K(B_k,1)$ where $B_k$ is Artin's braid group. 
		\item The space $C_k(\R^2\setminus \pt)$ is a $K(\bcyl_k,1)$ where $\bcyl_k$ is the cylinder braid group.
	\end{enumerate}
\end{ex}

\begin{prop} \label{prop:bcyl}\cite{brieskorn} The group $\bcyl_k$ of cylinder braids admits a presentation with generators $\sigma_1, \dots, \sigma_k$, $\kappa$ and relations
	\begin{align}
		&\sigma_i \sigma_{i+1} \sigma_i = \sigma_{i+1} \sigma_{i} \sigma_{i+1}, & &\sigma_i \sigma_j = \sigma_j \sigma_i \text{ if } |i-j|>1,\label{eq:bcyl1}\\
		&\sigma_1 \kappa \sigma_1 \kappa =  \kappa \sigma_1 \kappa \sigma_1, & &\sigma_i \kappa = \kappa \sigma_i \text{ if } i > 1. \label{eq:bcyl2}
	\end{align}
\end{prop}

We now give a definition of orbifold configuration spaces. The definition is a slight adaptation of the `orbit configuration spaces' of M. Xicot\'encatl \cite{xicot97}.\footnote{The difference between our definition and \cite{xicot97} is that we restrict to smooth points. }
\begin{defn} 
	Let $\Gamma$ be a finite group acting smoothly and faithfully on a surface $\Sigma$. Denote $\Sigma_{\mathrm{free}} \subset \Sigma$ the subset  of \emph{smooth points} i.e. the subset of points where $\Gamma$ acts freely. 
	\begin{enumerate}
		\item \emph{The configuration space of $k$ ordered smooth points} in the orbifold $[\Sigma/G]$ is denoted $F_k[\Sigma/G]$ and defined as
		\[F_k [\Sigma/\Gamma] := \big\{ (z_1,\dots,z_n)\in (\Sigma_{\mathrm{free}})^{\times k} \mid \Gamma \cdot z_i \cap \Gamma \cdot z_j = \emptyset \text{ if } i \neq j  \big\}.\]
		There is a natural free $\Gamma^{\times k} \rtimes S_k$ action on $F_k[\Sigma/\Gamma]$.
		\item \emph{The configuration space of $k$ unordered smooth points} in $[\Sigma/\Gamma]$ is denoted $C_k[\Sigma/\Gamma]$ and defined to be the quotient space
		$F_k [\Sigma/G] /\Gamma^{\times k} \rtimes S_k$.
		\item The \emph{braid group on $k$ strands} in the orbifold $[\Sigma/\Gamma]$ is denoted $B_k[\Sigma/\Gamma]$ and defined to be the fundamental group$\pi_1 C_n[\Sigma/\Gamma]$.
	\end{enumerate}
\end{defn}

The following two $\Z_2$-global quotients will be our central examples throughout:

\begin{ex} \label{danddstardef} Fix a dimension $n\geq 1$. Let $D^n$ denote the open unit disk in $\R^n$. Let $\Z_2$ act on $\R^n$ via the sign representation. We have an induced action on $D^n$.
	\begin{enumerate}
		\item Let $\DD_*$ denote the topological space $D^n$ equipped with the $\Z_2$ action given by the sign action.
		\item Let $\DD$ denote the topological space $D^n \amalg D^n$ equipped with the $\Z_2$ action given by combining the sign action with swapping the two disks.
	\end{enumerate}
	The spaces for $n=2$ are illustrated in Figure \ref{fig:danddstar}.
\end{ex}

\begin{figure}
	\begin{tikzpicture}
	\draw[teal] (0,0) circle [radius=1];
	\draw[purple] (2.5,0) circle [radius=1];
	\node at (0,0) {$D^2_b$};
	\node at (2.5,0) {$D^2_r$};
	\node at (1.25,-1.5) {$\DD$};
	\draw[black] (7,0) circle [radius=1];
	\node at (7,0) {$\ast$};
	\node at (7,-1.5) {$\DD_*$};
	\end{tikzpicture}
	\caption{The two-dimensional $\Z_2$-global quotients $\DD$ and $\DD_*$.}
	\label{fig:danddstar}
\end{figure}
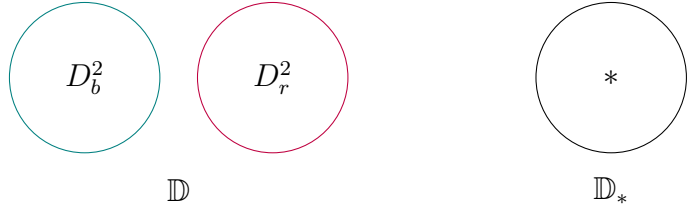

\begin{lem} \label{kpi1} Let $\DD$ and $\DD_*$ be the orbifolds defined above for dimension $n=2$.
	\begin{enumerate}
		\item The space $C_k[\DD^2/\Z_2]$ is a $K(B_k,1)$.
		\item  The space $C_k[\DD^2_*/\Z_2]$  is a $K(\bcyl_k,1)$.
	\end{enumerate}
\end{lem}

\begin{proof} 
	As the proofs are identical we will do both at once. Let $\Sigma$ be either $\DD$ or $\DD_*$. 
	The projection $F_k[\Sigma/\Gamma] \rightarrow F_n(\Sigma_{free}/\Gamma)$ descends to a $S_k$-equivariant homeomorphism $F_k[\Sigma/\Gamma]/\Gamma^{\times k} \cong F_k(\Sigma_{free}/\Gamma)$. Consequently $C_k[\Sigma/\Gamma] \cong C_k(\Sigma_{free}/\Gamma)$ are naturally homeomorphic. Thus we reduce to the classical cases of Example \ref{confex}.
\end{proof}

\subsection{Defining the operad} Let $n \geq 1$ denote some dimension.
\begin{nota}We will use the following notation throughout without further comment:
	\begin{enumerate} 
		\item To differentiate between the two copies of $D^n$ in $\DD$ we will write $\DD = D^n_b \amalg D^n_r$. For a map $f: \DD \rightarrow X$ we will write $f_b = f|_{D^n_b} : D^n_b \rightarrow X$ and $f_r = f|_{D^n_r} : D^n_r \rightarrow X$ for the restrictions. We will write $\textbf{0}_b$ for the center of $D^n_b \subset \DD$. 
		\item We will view maps whose source is a disjoint union of spaces as collections of \emph{component maps}, e.g. viewing $f: \DD^{\amalg k} \rightarrow \DD$  as a collection $(f^i: \DD \rightarrow \DD)_{i=1}^k$.
	\end{enumerate}
\end{nota}
An embedding $f: D^n \rightarrow D^n$ is called \emph{rectilinear} if it is of the form $f(x,y)  = (\lambda x + t_x, \lambda y + t_y)$ for some $\lambda, t_x, t_y \in \R$. We will call an embedding $f: \DD \rightarrow \DD$, and $f: \DD \rightarrow \DD_*$ rectilinear if the restrictions $f_r$ and $f_b$ are rectilinear. We will say an embedding $f:\DD^{\amalg k} \amalg \DD_*^{\amalg m} \rightarrow \DD_*$, or $f: \DD^{\amalg k} \sqcup \DD_*^{\amalg m} \rightarrow \DD$, is \emph{equivariant rectilinear} if the component maps $f_i$ are rectilinear and $\Z_2$-equivariant. 

We will now define a topological coloured operad $\ztd_n$ with two colours $\DD$ and $\DD_*$ and which has spaces of operations
\begin{align}
	&\ztd_n( \DD^{\amalg k} \amalg \DD_*^{\amalg m}, \DD) = \{\text{equivariant rectilinear embeddings } \DD^{\amalg k} \sqcup \DD_*^{\amalg m} \rightarrow \DD\}, \label{operations1} \\
	&\ztd_n( \DD^{\amalg k} \sqcup \DD_*^{\amalg m}, \DD_*) = \{\text{equivariant rectilinear embeddings } \DD^{\amalg k} \sqcup \DD_*^{\amalg m} \rightarrow \DD_* \}. \label{operations2}
\end{align}
Operadic composition is defined in the obvious way: compose the embeddings. Concretely, one inserts a configurations of disks into disks; an example in dimension two is illustrated in Figure \ref{fig:opcomp}.
\begin{rmk}
	Equivariance requires an embedding to map a $\Z_2$-fixed point to a $\Z_2$-fixed point. Therefore, $\ztd_n( \DD^{\amalg k} \amalg \DD_*^{\amalg m}, \DD) = \emptyset$ if $m>0$ and $\ztd_n( \DD^{\amalg k} \amalg \DD_*^{\amalg m}, \DD_*) = \emptyset$ if $m>1$. 
\end{rmk}
It remains to address the topology on the sets of operations \ref{operations1} and \ref{operations2}. For a rectilinear embedding $f: D^n \rightarrow D^n$ we write $r(f)$ for the radius of the disk $f(D^n)$. We define maps
\begin{align*}
	&\ztd_n (\DD^{\amalg k} \sqcup \DD_*, \DD_*) \rightarrow \R^{k+1} \times F_k[\DD_*/\Z_2], \\
	& ((f^i: \DD \rightarrow \DD_*)_{i=1}^k,f_{k+1}: \DD_* \rightarrow \DD_*) \mapsto \big(r(f^1_b),\dots,r(f^k_b),r(f_{k+1}), f^1(\textbf{0}_b),\dots, f^k(\textbf{0}_b) \big),
\end{align*}
\begin{align*}
	&\ztd_n (\DD^{\amalg k}, \DD_*) \rightarrow \R^k \times F_k[\DD/\Z_2],\\
	& ((f^i: \DD \rightarrow \DD_*)_{i=1}^k) \mapsto \big(r(f^1_b),\dots,r(f^k_b),  f^1(\textbf{0}_b),\dots, f^k(\textbf{0}_b) \big),\\
\end{align*}
\begin{align*}
	&\ztd_n(\DD^{\amalg k}, \DD) \rightarrow \R^k \times F_k[\DD/\Z_2],\\
	& ((f^i: \DD \rightarrow \DD)_{i=1}^k) \mapsto \big(r(f^1_b),\dots,r(f^k_b),  f^1(\textbf{0}_b),\dots, f^k(\textbf{0}_b) \big)
\end{align*} 
that record the positions of the centers and the radii. These maps are clearly injections. We induce the subspace topology on the sets of operations, where $\R^k$ has the Euclidean topology. This endows $\ztd_n$ with the structure of a topological operad. 
\begin{defn} The topological coloured operad $\ztd_n$, described above, with two colours $\DD$ and $\DD_*$ and operation spaces \ref{operations1} and \ref{operations2} is called \emph{the involutive little $n$-disks operad}.
\end{defn}
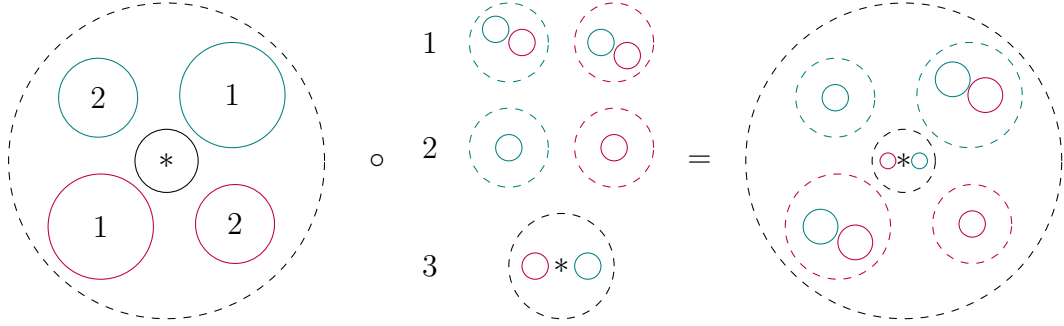
\begin{figure}[h]
	\begin{tikzpicture}[scale=.7]
	\draw [teal,dashed] (6.5,2.25) circle [radius = .75];
	\draw [teal] (6.25,2.5) circle [radius =.25];
	\draw [purple] (6.75,2.25) circle [radius =.25];
	\draw [purple,dashed] (8.5,2.25) circle [radius = .75];
	\draw [purple] (8.75,2) circle [radius =.25];
	\draw [teal] (8.25,2.25) circle [radius =.25];
	\draw [teal,dashed] (6.5,.25) circle [radius = .75];
	\draw [teal] (6.5,.25) circle [radius =.25];
	\draw [purple,dashed] (8.5,.25) circle [radius = .75];
	\draw [purple] (8.5,.25) circle [radius = .25];
	\draw [black, dashed] (7.5,-2) circle [radius = 1];
	\node at (7.5,-2) {$\ast$};
	\draw [teal] (8,-2) circle [radius =.25];
	\draw [purple] (7,-2) circle [radius = .25];
	\draw [black,dashed] (0,0) circle [radius=3];
	\draw [black] (0,0) circle [radius=.6];
	\node at (0,0) {$\ast$};
	\draw [teal] (1.25,1.25) circle [radius = 1];
	\node at (1.25,1.25) {1};
	\draw [purple] (-1.25,-1.25) circle [radius = 1];
	\node at (-1.25,-1.25) {1};
	\draw [teal] (-1.3,1.2) circle [radius =.75];
	\node at (-1.3,1.2) {2};
	\draw [purple] (1.3,-1.2) circle [radius =.75];
	\node at (1.3,-1.2) {2};
	\node at (4,0) {$\circ$};
	\node at (10.125,0) {$=$};
	\node at (5,2.25) {1};
	\node at (5,.25) {2};
	\node at (5,-2) {3};
	\draw [black,dashed] (14,0) circle [radius = 3];
	\node at (14,0) {$\ast$};
	\draw [black, dashed] (14,0) circle [radius =.6];
	\draw [teal,dashed] (15.25,1.25) circle [radius = 1];
	\draw [teal] (14.93,1.55) circle [radius =.33];
	\draw [purple] (15.55,1.25) circle [radius =.33];
	\draw [purple,dashed] (12.75,-1.25) circle [radius = 1];
	\draw [purple] (13.08,-1.55) circle [radius =.33];
	\draw [teal] (12.42,-1.25) circle [radius =.33];
	\draw [teal, dashed] (12.7,1.2) circle [radius =.75];
	\draw [purple, dashed] (15.3,-1.2) circle [radius =.75];
	\draw [teal] (12.7,1.2) circle [radius =.25];
	\draw [purple] (15.3,-1.2) circle [radius =.25];
	\draw [teal] (14.3,0) circle [radius = .15];
	\draw [purple] (13.7,0) circle [radius=.15];
	\end{tikzpicture}  
	\caption{An example of operadic composition in the operad $\ztd_2$.} 
	\label{fig:opcomp}
\end{figure}
\subsection{The homotopy type of the spaces of operations}
The connection between configuration spaces and the $\ztd_n$-operad is established in the following easy but nevertheless important result.
\begin{prop} \label{htype}
	The maps 
	\begin{align*}
		&\ztd_n( \DD^{\amalg k} \sqcup \DD_*, \DD_*) \hookrightarrow \R^{k+1} \times F_k[\DD_*/\Z_2] \xrightarrow{project} F_k[\DD_*/\Z_2],\\
		&\ztd_n( \DD^{\amalg k}, \DD_*) \hookrightarrow \R^k \times F_k[\DD_*/\Z_2] \xrightarrow{project} F_k[\DD_*/\Z_2],\\
		&\ztd_n( \DD^{\amalg k}, \DD) \hookrightarrow \R^k \times F_k[\DD/\Z_2] \xrightarrow{project} F_k[\DD/\Z_2].
	\end{align*}
	that forget the radii define homotopy equivalences.
\end{prop}
\begin{proof}
	Easy adaptation of the usual proof for the $E_n$-operad and $F_k(\R^2)$.\footnote{For the usual proof for the $E_n$-operad and $F_k(\R^2)$ see e.g. \cite[Lemma 5.1.1.3]{lha}.}
\end{proof}
\begin{rmk} \label{rmk:colouredbraids}
	By Proposition \ref{htype} we can encode isotopies in $\ztd_n(\DD^{\amalg k}, \DD_*)$ up to homotopy as paths in $F_k[\DD_/\Z_2]$. 
	Paths in $F_k[\DD_*/\Z_2]$ are naturally drawn as coloured braids, see e.g. Figures \ref{fig:z2monodromy} and \ref{fig:tre}.
\end{rmk}
In dimension one the configuration spaces $F_k[\DD/\Z_2]$ and $F_k[\DD_*/\Z_2]$ are particularly easy to describe. 
To a point $p\in F_1[\DD/\Z_2]$ we assign a value 
\begin{align}
	\ep (p) = \begin{cases} 0 \text{ if } p \in D^1_b,\\ 1 \text{ if } p\in D^1_r. \end{cases} \label{tupledef1}
\end{align}
and to a configuration $x = (x_1,\dots,x_k) \in F_k[\DD_*/\Z_2]$ we assign a permutation $\sigma_x \in S_k$ recording the order of the disks in $D^1_b$, see Figure \ref{fig:perm}. More precisely, $\sigma_x$ is the unique permutation in $S_k$ for which the inequalities
\begin{align} 
	(-1)^{\ep (x_{\sigma_x(1)})} x_{\sigma_x(1)} < (-1)^{\ep(x_{\sigma_x (2)})} x_{\sigma_x(2)} < \dots < (-1)^{\ep (x_{\sigma_x (k)})} x_{\sigma_x(k)} \label{sigmadef}
\end{align}
hold. 
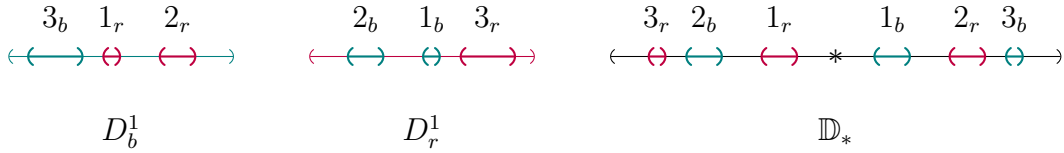
\begin{figure}[h]
	\begin{tikzpicture}[xscale=.5]
	\draw [(-),teal] (0,0) -- (6,0);
	\draw [(-),purple] (8,0) -- (14,0);
	\draw [(-),teal,thick] (.5,0) -- (2,0);
	\draw [(-),purple,thick] (2.5,0) -- (3,0);
	\draw [(-),purple,thick] (4,0) -- (5,0);
	\draw [(-),purple,thick] (12,0) -- (13.5,0);
	\draw [(-),teal,thick] (11,0) -- (11.5,0);
	\draw [(-),teal,thick] (9,0) -- (10,0);
	\node at (1.25,.5) {$3_b$};
	\node at (2.75,.5) {$1_r$};
	\node at (4.5,.5) {$2_r$};
	\node at (9.5,.5) {$2_b$};
	\node at (11.25,.5) {$1_b$};
	\node at (12.75,.5) {$3_r$};
	\node at (3,-1) {$D^1_b$};
	\node at (11,-1) {$D^1_r$};
	\node at (7,-3) {$ \left(1,1,0,(312)\right)$};
	\draw [(-),black] (16,0) -- (28,0);
	\node at (22,0) {$\ast$};
	\draw [(-),purple, thick] (17,0) -- (17.5,0);
	\draw [(-),teal, thick] (18,0) -- (19,0);
	\draw [(-),purple, thick] (20,0) -- (21,0);
	\draw [(-),teal, thick] (23,0) -- (24,0);
	\draw [(-),purple, thick] (25,0) -- (26,0);
	\draw [(-),teal, thick] (26.5,0) -- (27,0);
	\node at (17.25,.5) {$3_r$};
	\node at (18.5,.5) {$2_b$};
	\node at (20.5,.5) {$1_r$};
	\node at (23.5,.5) {$1_b$};
	\node at (25.5,.5) {$2_r$};
	\node at (26.75,.5) {$3_b$};
	\node at (22,-1) {$\DD_*$};
	\node at (22,-3) {$\left(0,1,0,(123)\right)$};
	\node at (15,-2) {with assignments $(\ep(x_1),\ep(x_2),\ep(x_3),\sigma_x)$ being};
	\end{tikzpicture}
	\caption{Operations in $\ztd_1$ and their assignments in $\{0,1\}^{\times 3} \times S_3$.}
	\label{fig:perm}
\end{figure} 
\begin{lem} \label{order1}
	The map 
	\begin{align*}
		&F_k [ \DD/\Z_2] \rightarrow \{0,1\}^{\times k } \times S_k, \\
		&x = (x_1,\dots,x_k) \mapsto (\ep (x_1),\dots, \ep (x_k),\sigma_x )
	\end{align*}
	is a homotopy equivalence.
\end{lem}
\begin{proof}
	Easy.
\end{proof}
Similarly, to a point $x \in F_1[\DD_*/\Z_2]$ we can assign a value  
\begin{align}
	\ep(x) = \begin{cases} 0 \text{ if } x >0,\\ 1 \text{ if } x < 0. \end{cases} \label{tupledef2}
\end{align}
so that for every $x \in F_k[\DD_*/\Z_2]$ there is a unique $\sigma_x \in S_k$ satisfying Equation \eqref{sigmadef}. 
\begin{lem} \label{order2}
	The map
	\begin{align*}
		&F_k [ \DD_*/\Z_2] \rightarrow \{0,1\}^{\times k } \times S_k, \\
		&x = (x_1,\dots,x_k) \mapsto (\ep (x_1),\dots,\ep (x_k),\sigma_x ) 
	\end{align*}
	defines a homotopy equivalence. 
\end{lem}
\begin{proof}
	Easy.
\end{proof}

\section{$\Z_2$-braided pairs} \label{sec:ztd2}

Before we state our definitions we fix notation and remind the reader of some standard definitions in categorical representation theory, for more details and background see \cite{egno15}. A \emph{monoidal category} is a category $\catA$ together with a tensor product functor $\tensor: \catA \times \catA \rightarrow \catA$, a unit object $\1 \in \catA$ and natural isomorphisms
\begin{align*}
\alpha_{X,Y,Z} : (X \tensor Y) \tensor Z \cong X \tensor (Y \tensor Z), \quad
\lambda_X: \1 \tensor X \cong X, \quad
\rho_X: X \tensor \1 \cong X,
\end{align*}
for $X,Y,Z\in \catA$ satisfying the Mac Lane triangle and pentagon axioms. A \emph{braided monoidal category} is a monoidal category $\catA$ together with a natural isomorphism $\sigma$, called \emph{the braiding},
\[ \sigma_{X,Y} : X \tensor Y \cong Y \tensor X\]
for $X,Y\in \catA$ satisfying the Joyal-Street hexagon axioms. A \emph{monoidal functor} between monoidal categories $\catA$ and $\catB$ consists of a functor $\Psi: \catA \rightarrow \catB$ together with natural isomorphisms
\[ (\Psi_2)_{X,Y}: \Psi (X) \tensor \Psi(Y) \cong \Psi( X \tensor Y), \quad \Psi_0: \Psi(\1_\catA) \cong \1_\catB\]
for $X,Y \in \catA$ satisfying unit and associativity axioms. A monoidal functor $\Psi$ between braided monoidal categories is called \emph{braided} if it preserves the braiding i.e. $\Psi(\sigma_\catA) = \sigma_\catB$. 
A (right) \emph{$\catA$-module category} over a monoidal category $\catA$ is a category $\catM$ together with an action functor\footnote{Typically the action functor is denoted $\tensor$, we write $\act$ to avoid confusion with the tensor product on $\catA$.} $\act: \catM \times \catA \rightarrow \catM, (M,X) \mapsto M \tensor X$ and natural isomorphisms
\[ a_{M,X,Y}: (M \tensor X) \tensor Y \cong M \tensor (X\tensor Y), \quad r_M: M \tensor \1 \cong M\]
for $M \in \catM$ and $X,Y \in \catA$ satisfying a unit and associativity axiom. We call a module category $\catM$ \emph{pointed} if it has a distinguished object $\1_\catM \in \catM$.
\begin{nota} For a monoidal category $\catA$ we have the \emph{opposite monoidal category} $\catA^{\tensop}$, which is the category $\catA$ with the reversed order tensor product $\tensor^{op}$ i.e. $X \tensor^{op} Y := Y \tensor X$. 
\end{nota}
\begin{defn} \label{z2pair}
	\begin{enumerate}
		\item An \emph{anti-involution} $\Phi$ of a monoidal category $\catA$ consists of a monoidal functor $(\Phi,\Phi_2,\Phi_0): \catA \rightarrow \catA^{\tensop}$ and a monoidal isomorphism $t: \Phi^2 \Rightarrow \id_\catA$ such that $\Phi(t_X) = t_{\Phi(X)}$.
		\item A \emph{$\Z_2$-monoidal pair} consists of a monoidal category $\catA$  together with an anti-involution $\Phi$, and a pointed $\catA$-module category $\catM$. 
		\item A \emph{$\Z_2$-braided pair} consists of a $\Z_2$-monoidal pair  $(\catA,\catM)$ with the following additional structures. The category $\catA$ is braided monoidal so that $\Phi$ is braided, together with natural isomorphisms $\kappa_{M,X}: M\tensor X \rightarrow M \tensor \Phi(X)$ such that the diagrams in Figures \ref{fig:bp1} and \ref{fig:bp2} commute for all $M\in \catM$ and $X,Y \in \catA$. The natural isomorphism $\kappa$ is called the \emph{$\Z_2$-cylinder braiding}.
	\end{enumerate}
\end{defn}
\begin{figure}[h]
	\adjustbox{scale=.9}{
		\begin{tikzcd}[column sep=huge]
			(M \tensor X) \tensor Y \arrow[d,"a^{-1}"] \arrow[r, "\kappa_{M\tensor X,Y}"] & (M \tensor X) \tensor \Phi (Y) \\
			M \tensor (X \tensor Y) \arrow[d, "\id \tensor \sigma"] & M \tensor (X \tensor \Phi(Y)) \arrow[u,"a"] \\
			M \tensor (Y \tensor X) \arrow[d,"a"] & (M \tensor \Phi(Y)) \tensor X \arrow[u,"\id \tensor \sigma"] \\
			(M \tensor Y) \tensor X \arrow[r,"\kappa_{M,Y}\tensor \id"] & (M\tensor \Phi(Y)) \tensor X \arrow[u, "a^{-1}"]
		\end{tikzcd}}
		\quad \quad \quad $\vcenter{\hbox{\adjustbox{scale=.7}{\begin{tikzpicture}
					\draw[purple,line width=1.2pt] (1.4,0) -- (1.4,3.5);
					\draw[teal,line width=1.2pt] (1.6,0) -- (1.6,3.5);
					\draw[black,line width=1.2pt] (1.5,0) -- (1.5,3.5);
					\begin{knot}[clip width = 3,flip crossing=1,]
					\strand[line width=2pt,purple] (0,0) .. controls +(0,3) and +(0,-3) .. (3,3.5);
					\strand[line width=2pt,teal] (3,0) .. controls +(0,3) and +(0,-3) .. (0,3.5);
					\end{knot}
					\node at (3.2,3.8) {$\Phi(Y)$};
					\node at (3,-.3) {$Y$};
					\node at (1.5,-.3) {$M \otimes X$};
					\end{tikzpicture}}}} \simeq \quad
		\vcenter{\hbox{\adjustbox{scale=.7}{\begin{tikzpicture}
					\draw[black,line width = 1pt] (2.5,-3.5) -- (2.5,0);
					\braid[line width = 2pt, 
					style strands={1}{teal}, 
					style strands={2}{purple},
					style strands={3}{teal},
					style strands={4}{purple}]
					s_1-s_3 s_2 s_1-s_3;
					\node at (2.5,-3.8) {$M$};
					\node at (3,-3.8) {$X$};
					\node at (3,.3) {$X$};
					\node at (4,-3.8) {$Y$};
					\node at (4.2,.3) {$\Phi(Y)$};
					\end{tikzpicture}}}}$
		\caption{Axiom BP1 and a graphical interpretation in terms of coloured braids.}
		\label{fig:bp1}
	\end{figure}
	\begin{figure}[h] 
		\adjustbox{scale=.8}{\begin{tikzcd}
				M \tensor (X \tensor Y) \arrow[r,"\kappa_{M,X\tensor Y}"] \arrow[d,"a"] & M \tensor \Phi(X \tensor Y) & M \tensor (\Phi(Y) \tensor \Phi(X) ) \arrow[l,swap,"\id \tensor \Phi_2"] \\
				(M \tensor X) \tensor Y \arrow[d,"\kappa_{M,X}\tensor \id"] & & (M \tensor \Phi(Y)) \tensor \Phi(X) \arrow[u,"a^{-1}"] \\
				(M \tensor \Phi(X) ) \tensor Y \arrow[d,"a^{-1}"] & & (M \tensor Y) \tensor \Phi(X) \arrow[u,"\kappa_{M,Y}\tensor \id"] \\
				M \tensor( \Phi(X) \tensor Y) \arrow[rr, "\id \tensor \sigma"] & & M \tensor( Y \tensor \Phi(X) ) \arrow[u, "a"] 
			\end{tikzcd}}  \quad
			$\vcenter{\hbox{\adjustbox{scale=.6}{\begin{tikzpicture}
						\draw[black,line width=2pt] (1.5,0) -- (1.5,3.5);
						\begin{knot}[clip width=2,flip crossing=1,]
						\strand[line width=1.3pt,purple,double,double distance between line centers=.3em,] (0,0) .. controls +(0,3) and +(0,-3) .. (3,3.5);
						\strand[line width=1.3pt,teal,double,double distance between line centers=.3em,] (3,0) .. controls +(0,3) and +(0,-3) .. (0,3.5);
						\end{knot}
						\node at (3.2,3.8) {$\Phi(X\otimes Y)$};
						\node at (3,-.3) {$X\otimes Y$};
						\node at (1.5,-.3) {$M$};
						\end{tikzpicture}}}} \simeq \quad
			\vcenter{\hbox{\adjustbox{scale=.55}{\begin{tikzpicture}
						\draw[black,line width = 1pt] (2.5,-3.5) -- (2.5,0);
						\braid[line width = 2pt, 
						style strands={1}{teal}, 
						style strands={2}{teal},
						style strands={3}{purple},
						style strands={4}{purple}]
						s_2 s_1-s_3 s_2;
						\node at (2.5,-3.8) {$M$};
						\node at (3,-3.8) {$X$};
						\node at (3.2,.3) {$\Phi(Y)$};
						\node at (4.2,-3.8) {$Y$};
						\node at (4.4,.3) {$\Phi(X)$};
						\end{tikzpicture}}}}$
\caption{Axiom BP2 and a graphical interpretation in terms of coloured braids.}
	\label{fig:bp2}
\end{figure}

\begin{rmk} \label{rmk:comparison1}
	The definition of a $\Z_2$-braided pair is a $\Z_2$-equivariant version of the \emph{braided module categories} of A. Brochier \cite{brochier13}. A braided module category over a braided monoidal category $\catA$ consists of an $\catA$-module category $\catM$ together with a natural isomorphism $\gamma: M \tensor X \rightarrow M \tensor X$ for $X\in \catA, M\in \catM$ satisfying axioms BP1 and BP2, but where the functor $\Phi = \id_\catA$ is trivial and replacing $\Phi_2$ by the braiding $\sigma$ in axiom BP2. 
\end{rmk}

Essentially all examples of $\Z_2$-braided pairs come from the representation theory of quasi-triangular Hopf algebras, and their coideal subalgebras. 
\begin{nota}
	For an algebra $A$ we denote with $A\Mod$ the category of $A$-modules, and $A\fdmod$ the category of finite dimensional $A$-modules. 
\end{nota}
Recall that for a quasi-triangular Hopf algebra $H$ the category $H\Mod$ is naturally a braided monoidal category. A \emph{coideal subalgebra} $B\subset H$ is a subalgebra so that $\Delta(B) \subset B \tensor H$. The category $B\Mod$ is then naturally a module category over $H\Mod$. 
To obtain a $\Z_2$-braided pair $(H\Mod,B\Mod)$ we need some further structure on $H$ and $B$.

\begin{defn} \label{qtqsp} \cite[Definition 2.7]{kolb17}
	Let $H$ be a quasi-triangular Hopf algebra with universal $R$-matrix $R\in H$, a Hopf algebra involution $\phi: H \rightarrow H$ such that $(\phi \tensor \phi)(R) = R$ and $B\subset H$ a coideal subalgebra. 
	The coideal subalgebra $B$ is called \emph{quasi-triangular} if there exists an invertible element $K \in B \tensor U$ satisfying 
	\begin{align*}
	&K \Delta(b) = (\id \tensor \phi) \Delta(b) K \\
	&(\Delta \tensor \id) (K) = R_{32}^\phi K_{13} R_{23}\\
	&(\id \tensor \Delta) (K) = R_{32} K_{13} R_{23}^\phi K_{12}
	\end{align*}
	Here $R^\phi := \phi \tensor \id (R)$. Note $R^\phi = \id\tensor \phi (R)$ also. We call $K$ the \emph{universal $K$-matrix}.
\end{defn}
The \emph{raison d'etre} of universal $K$-matrices is that they provide solutions to the twisted reflection equation
\[R_{32} K_{13} R_{23}^\phi K_{12} = K_{12} R_{32}^\phi K_{13} R_{23}.\]
Recall that a \emph{balancing} on a braided monoidal category $\catA$ is a natural isomorphism $\theta: \id_\catA \Rightarrow \id_\catA$ satisfying
\begin{align}
\theta_{X\tensor Y} = \sigma_{Y,X} (\theta_Y \tensor \theta_X) \sigma_{X,Y} \label{ribbonaxiom}
\end{align}
for all $M,N \in \catA$. A \emph{ribbon Hopf algebra} is a quasi-triangular Hopf algebra $H$ together with a \emph{ribbon element} i.e. a central invertible element $\nu \in H$ such that the natural isomorphism $\theta: \id_{H\text{-mod}} \Rightarrow \id_{H\Mod}$ defined by acting with $\nu$ gives a balancing on $H\Mod$. 

\begin{rmk} \label{ribbonrmk}
	For a braided monoidal category $\catA$  both $(\id_\catA,\sigma)$ and $(\id_\catA,\sigma^{-1})$ define monoidal functors $\catA \rightarrow \catA^{\tensop}$.
	The balancing axiom \ref{ribbonaxiom} expresses that $\theta: (\id_\catA,\sigma^{-1}) \cong (\id_\catA,\sigma)$ is a monoidal isomorphism. Consequently also $(\id_\catA,\sigma^2) \cong (\id_\catA,\id)$ monoidally.
\end{rmk}

\begin{prop} \label{prop:ztd2ex}
	Let $H$ be a ribbon Hopf algebra, together with a be a Hopf algebra involution $\phi$ that preserves $R$ and $v$, and a quasi-triangular coideal subalgebra $B$. Then $(H\Mod,B\Mod)$ is canonically equipped with the structure of a $\Z_2$-braided pair.
\end{prop}
\begin{proof}
	The category $H\Mod$ is braided monoidal since $H$ is quasi-triangular. The involution $\phi: H \rightarrow H$ defines a functor $\Phi: M \mapsto M^\phi$, where $M^{\phi}$ is the $H$-module where the $H$-action is twisted by $\phi$.
	Since $\phi$ is an Hopf algebra involution that preserves $R$, the functor $(\Phi,\id,\id_\1)$ is a braided monoidal functor. 
	We equip $\Phi$  with the structure of an anti-involution via
	\begin{align*}
	&\Phi_2: \Phi(X) \tensor \Phi(Y) \xrightarrow{\sigma_{\Phi(X),\Phi(Y)}} \Phi(Y) \tensor \Phi(X) = Y^\phi \tensor X^\phi = (Y \tensor X)^\phi = \Phi(Y\tensor X), \\
	&\Phi_0: \Phi(\1) = \1^{\phi} \xrightarrow{\id_{\1}} \1.
	\end{align*} 
	The monoidal functor $(\Phi,\sigma,\id_\1): \catA \rightarrow \catA^{\tensop}$ squares to $(\Phi,\sigma,\id_1)\circ (\Phi,\sigma,\id_\1) = (\id_\catA,\sigma^{2},\id_\1)$. The balancing $\theta$ then gives a monoidal isomorphism $t:\Phi^2 \cong \id_\catA$ by Remark \ref{ribbonrmk}. Since $\phi(\nu) = \nu$ we also have $\Phi(\theta) = \theta_{\Phi}$.
	As $B$ is a right coideal subalgebra the category $B\Mod$ inherits the structure of a right module category over $H\Mod$ with  pointing $\1 \in B\Mod$.\\ 
	Action by the universal $K$-matrix defines the $\Z_2$-cylinder braiding.
\end{proof}

As explained in the introduction, examples are provided by quantum symmetric pairs. 

\begin{rmk} \label{rmk:comparison2}
	Given a braided monoidal category $\catA$, the notion of a $\Z_2$-braided pair $(\catA,\catM)$ is very close to the notion of a $\Phi$-braided module category of \cite{kolb17}. 
	Recall that a \emph{$\Phi$-braided module category} over $\catA$ consists of an $\catA$-module category $\catM$, a braided monoidal equivalence $(\Phi,\id,\id_\1): \catA \rightarrow \catA$ and a family of natural isomorphisms $\kappa_{M,X}: M\tensor X\rightarrow M \tensor \Phi(X)$ for $X\in \catA$ and $M \in \catM$ satisfying the axioms BP1 and BP2 in Figures \ref{fig:bp1} and \ref{fig:bp2}.\footnote{Where $\Phi_2$ is replaced by $\sigma$ in BP2.} \\
	The difference is that in a $\Z_2$-braided pair $\Phi$ is required to be anti-monoidal and involutive. 
	Given a $\Phi$-braided module category $\catM$ over $\catA$ such that the underlying functor $\Phi$ is strictly involutive, then the pair $(\catA,\catM)$ defines a $\Z_2$-braided pair exactly if $\catA$ admits a balancing.\footnote{In that case we can define the anti-involution via $(\Phi,\sigma,\id_\1)$ with $t: \Phi^2 \cong \id$ being the balancing.}
\end{rmk}

\begin{defn}
	A \emph{$\Z_2$-symmetric pair} is a $\Z_2$-braided pair $(\catA,\catM)$ such that $\sigma^2 = \id$, $\kappa^2= \id$.
\end{defn}
Examples of $\Z_2$-symmetric pairs are provided by infinitesimal symmetric pairs.
\begin{ex}
	Let $\g$ be a complex semisimple Lie algebra and let $\theta: \g \rightarrow \g$ be an involution with fixed-points $\g^{\theta}$. We can decompose $\theta = J \phi(-) J^{-1}$ for an outer automorphisms $\phi$ and $J \in G$.\footnote{Here $G$ is the Lie group integrating $\g$. Note that decompositions of involutions of semisimple Lie algebras are not unique in general.}
	Then $(U(\g)\fdmod,U(\g^{\theta})\fdmod)$ is a $\Z_2$-symmetric pair where $\Phi(M) = M^\phi$ and acting with $J$ defines $\kappa$.
\end{ex}

\section{Coherence results} \label{sec:coherence}

A coherence theorem asserts all diagrams of a certain class commute.
Most famously, Mac Lane's coherence theorem for monoidal categories states any diagram\footnote{More precisely, any diagram constructed out of the associator $\alpha$ and unitors $\lambda$ and $\rho$ and $\tensor$.} in a monoidal category $\catA$ commutes. 
An important consequence is that the a-priori different (parenthesized) tensor products $X_1 \tensor \dots \tensor X_n$ of objects $X_1,\dots, X_n \in \catA$ are canonically isomorphic.
In a braided monoidal category $\catA$ coherence is more subtle. It is not true that all diagrams commute, for example $\sigma^2 \neq \id$ in general. 
Thinking of $\sigma$ as the simple braid on two strands provides intuition: the double braid and the trivial braid are different braids.
This can be made precise: one associates a diagram of braids to a diagram in $\catA$ by interpreting $\sigma$ as the generators of a braid group $B_n$. 
The Joyal-Street Coherence Theorem then states that a diagram in a braided monoidal category commutes if the associated diagram of braids commutes \cite{js93}. 
We will now prove a coherence theorem (\ref{intro:coherence}) for $\Z_2$-braided pairs.

\subsection{Coherence in the strict setting}
We will now precisely state the coherence theorem, and prove it in a strict setting. 
Recall that a monoidal category is called \emph{strict} if $\alpha = \id$, $\rho = \id$ and $\lambda=\id$. 
\begin{defn} Let $(\catA,\catM)$ be a $\Z_2$-monoidal pair.
	\begin{enumerate}
		\item We call the $\Z_2$-monoidal pair \emph{strict} if $\catA$ is a strict monoidal category with a strict anti-involution $\Phi$, and $\catM$ is a strict module category.\footnote{Thus $\alpha$, $\lambda$, $\rho$, $a$, $r$, $\Phi_@$,$\Phi_0$ and $t$ are required to be identities.}
		\item We call a $\Z_2$-braided pair (or $\Z_2$-symmetric pair) \emph{strict} if the underlying $\Z_2$-monoidal pair is strict.
	\end{enumerate}
\end{defn}
We have the following simple, but important result:
\begin{lem} \label{relem}
	Let $(\catA,\catM)$ be a strict $\Z_2$-braided pair.
	\begin{enumerate}
		\item The braiding $\sigma$ satisfies the Yang-Baxter equation 
		\[ (\sigma_{Y,Z} \tensor \id_X) \circ (\id_Y \tensor \sigma_{X,Z}) \circ (\sigma_{X,Y} \tensor \id_Z) = (\id_Z \tensor \sigma_{X,Y}) \circ (\sigma_{X,Z} \tensor \id_Y) \circ (\id_X \tensor \sigma_{Y,Z}) \]
		for all $X,Y,Z\in \catA$. 
		\item The $\Z_2$-monodromy $\kappa$ satisfies the twisted reflection equation
		\begin{align*}
		(\id_M \tensor \sigma_{\Phi(Y),\Phi(X)}) \circ (\kappa_{M,Y}\tensor \id_X) &\circ (\id_M \tensor \sigma_{\Phi(X),Y}) \circ (\kappa_{M,X} \tensor \id_Y) \\
		&= \\
		(\kappa_{M,X} \tensor \id_Y) \circ (\id_M \tensor \sigma_{\Phi(Y),X}) &\circ (\kappa_{M,Y} \tensor \id_X)\circ (\id_M \tensor \sigma_{X,Y})
		\end{align*}
		for any $M\in \catM$ and $X,Y \in \catA$.
	\end{enumerate}
\end{lem}
\begin{proof}
	The proof that the braiding satisfies the Yang-Baxter equation is standard, see e.g. \cite[8.1.10]{egno15}. 
	To prove part (2) we observe that in the following diagram
	\begin{center}
		\adjustbox{scale=.75}{
			\begin{tikzcd}
				M \tensor \Phi(Y)  \tensor X \arrow[rr, "\id \tensor \sigma"] & & M \tensor X \tensor \Phi(Y) \arrow[d,"\kappa_{M,X}\tensor \id"] \\
				M \tensor Y \tensor X \arrow[r,"\kappa_{M,Y\tensor X}"] \arrow[u,"\kappa_{M,Y}\tensor \id"] & M \tensor \Phi(Y \tensor X) & M \tensor \Phi(X) \tensor \Phi(Y)  \arrow[l,swap,"\id \tensor \Phi_2"] \\
				M \tensor X \tensor Y \arrow[r,"\kappa_{M,X\tensor Y}"] \arrow[u,"\id \tensor \sigma"]\arrow[d,"\kappa_{M,X}\tensor \id"] & M \tensor \Phi(X \tensor Y) \arrow[u,"\id \tensor \Phi(\sigma)"] & M \tensor \Phi(Y) \tensor \Phi(X)  \arrow[l,swap,"\id \tensor \Phi_2"] \arrow[u,"\id \tensor \sigma"] \\
				M \tensor \Phi(X)  \tensor Y \arrow[rr, "\id \tensor \sigma"] & & M \tensor Y \tensor \Phi(X) \arrow[u,"\kappa_{M,Y}\tensor \id"] 
			\end{tikzcd}
		}
	\end{center}
	the subdiagrams commute by naturality of $\kappa$, the BP2 axiom and the fact that $\Phi$ is a braided monoidal functor. Commutativity of the outer sides follows and expresses that $\kappa$ satisfies the twisted reflection equation.
\end{proof}
Before stating the coherence theorem we introduce some notation and definitions.
\begin{nota}
	For natural isomorphisms $F_1 \xrightarrow{\beta} F_2$,  $G_1\xrightarrow {\gamma} G_2$ with composite functors $G_i \circ F_i$ we write $\gamma * \beta: G_1 \circ F_1 \rightarrow G_2 \circ F_2$ for the \emph{horizontal composition}. For natural isomorphisms $F \xrightarrow{\beta} G \xrightarrow{\gamma} H$ we write $\gamma \circ \beta$ or $\gamma \beta$ for the (vertical) composition $F \rightarrow H$.  
\end{nota}
\begin{defn} Let $(\catA,\catM)$ be a $\Z_2$-braided pair.
	A \emph{structural isomorphism} in $\catA$ (respectively $\catM$) is a natural isomorphism in $\catA$ (respectively $\catM$) constructed as a vertical and horizontal composition of the natural isomorphisms: $\id_{\id_\catA}, \alpha,\lambda,\rho,\id_\Phi, \id_{\tensor}, \Phi_2,\Phi_0,t$ (and $\id_{\id_{\catM}}, \id_{\act}, a, r, \kappa$) and their inverses.
\end{defn}

Any structural isomorphism $f$ in $\catA$ is a natural isomorphism between functors with domain $\catA^{\times n}$ and codomain  $\catA$ for some $n\geq 0$. 
To a structural isomorphism $f$ with a given presentation as a vertical/horizontal composition we can associate its \emph{underlying braid on $n$ strands}, denoted $\beta_f \in B_n$, by interpreting instances of the braiding $\sigma$ as simple braids, and doubling strings for instances of $\id_{\tensor}$. See Figure \ref{fig:underbraid} for two illustrated examples.
\begin{rmk}
	Note that the assignment of the underlying braid depends on the presentation of the structural isomorphism.
\end{rmk}
\begin{figure}[h]
	\begin{tikzpicture}[xscale=.5,yscale=.25]
	\begin{knot}[clip width=2,flip crossing=1,]
	\strand[line width=1.3pt,black,double,double distance between line centers=.3em,] (0,3.5) .. controls +(0,3) and +(0,-3) .. (3,7);
	\end{knot}
	\draw[white,line width=5pt] (1.5,0) -- (1.5,7);
	\draw[black,line width=2pt] (1.5,0) -- (1.5,7);
	\begin{knot}[clip width=2,flip crossing=1,]
	\strand[line width=1.3pt,black,double,double distance between line centers=.3em,] (3,0) .. controls +(0,3) and +(0,-3) .. (0,3.5);
	\end{knot}
	\node at (1.5,-1) {$\kappa * (\id_{\catM},\id_{\tensor}) = \kappa_{- , - \tensor -}$};
	\node at (5,3.5) {$\in \bcyl_2$};
	\end{tikzpicture} \quad \quad \quad \quad \quad
	\begin{tikzpicture}
	\braid[line width = 2pt]
	s_1-s_3;
	\node at (2.5,-2) {$\id_{\tensor} * (\sigma,\sigma) = \sigma \tensor \sigma$};
	\node at (5,-.75) {$\in B_4$};
	\end{tikzpicture}
	\caption{The underlying braids of two structural isomorphisms.}
	\label{fig:underbraid}
\end{figure}
Similarly a structural isomorphism $f$ in $\catM$ is a natural isomorphism between functors with domain $\catM \times \catA^{\times n}$  and codomain  $\catM$ for some $n\geq 0$. 
To such $f$ together with a given presentation we can associate its \emph{underlying cylinder braid on $n$ strands}, denoted $\beta_f \in \bcyl_n$, by interpreting instances of the braiding $\sigma$ and $\Z_2$-cylinder braiding $\kappa$ as the generators $\sigma$ and $\kappa$ in $\bcyl_n$, and doubling strings for instances of $\id_{\tensor}$.
\begin{rmk} \label{braidobservation}
	By axioms BP1 and BP2 of Figures \ref{fig:bp1} and \ref{fig:bp2} one can change the presentation of a structural isomorphism in $\catM$ so that only one string is involved in an instance of the $\Z_2$-cylinder braiding at a time but without changing the underlying braid.
\end{rmk}
\begin{prop} \label{strz2coh} Let $(\catA,\catM)$ be a strict $\Z_2$-braided pair and consider two parallel structural isomorphisms $f,g: T_1 \Rightarrow T_2$ in $\catA$ (or $\catM$). Then $f=g$ if $\beta_f = \beta_g$. 
\end{prop}
\begin{proof} 
	As both cases are analogous we will only give the proof for $f$ and $g$ in $\catM$,
	The non-trivial components of the presentations of $f$ and $g$ consist of instances of $\sigma$, $\kappa$ and $\id_{\tensor}$, $\id_{\Phi}$. 
	By Remark \ref{braidobservation} we can change the presentation of $f$ and $g$ so that the underlying braids only braid two strings at a time without changing the underlying braid.\footnote{Besides Remark \ref{braidobservation} we are using the analogous observation about the hexagon axioms for the braiding.} 
	As $\beta_f = \beta_g$ one can relate the presentations of the braids $\beta_f$ and $\beta_g$ via repeated applications of the cylinder braid relations \ref{eq:bcyl1} and \ref{eq:bcyl2}.
	Using that $\Phi$ is braided and strict we can then also rewrite the presentations of $f$ and $g$ into each other by applying the cylinder braid relations. 
	By Lemma \ref{relem} it follows that $f=g$.
\end{proof}

We will now prove coherence for general $\Z_2$-braided pairs by showing that one can replace a $\Z_2$-braided pair (up to appropriate equivalence) by a strict $\Z_2$-braided pair. 

\subsection{Strictification of the braided monoidal category}

The strictification of $\catA$ will move in two steps: first strictifying the tensor product on $\catA$ and then strictifying the involution $\Phi$. 
\begin{nota} Let $\catA$ be a monoidal category, recall that $\catA$ can be seen as a right $\catA$-module category in a canonical way. We denote this module category by $\catA_\catA$. For two right $\catA$-module categories $\catM_1$ and $\catM_2$ we denote by $\Fun_\catA(\catM_1,\catM_2)$ the category of right $\catA$-module functors, and morphisms of right $\catA$-module functors. For definitions see \cite[ch. 7]{egno15}.
\end{nota} 
\begin{defn} Let $\catA$ be a monoidal category. We denote by $\catA_{st}$ the strict monoidal category $\Fun_\catA(\catA_\catA,\catA_\catA)$. The tensor product on objects is given by composition of functors (denoted by $\circ$), with unit object the identity functor. The tensor product on morphisms is given by horizontal composition of natural transformations (denoted by $\hcomp$).  
\end{defn}
\begin{prop} \cite[Proposition 1.3]{js93} \label{strictifymon}
	Let $\catA$ be a monoidal category. The monoidal functor $(L,L_2,L_0): \catA \rightarrow \catA_{st}$ defined by $L(X) = (X\tensor -,\alpha_{X,-,-})$, $(L_2)_{X,Y} = \alpha_{X,Y,-}$, and $L_0 = l_-$ is a monoidal equivalence.
\end{prop} 
\begin{conv}
	We will use the following fact implicitly throughout: If $F$ is a monoidal functor so that the functor underlying $F$ is an equivalence, then any quasi-inverse $F^{-1}$ is naturally monoidal, and the natural isomorphisms $1 \Rightarrow F^{-1}F$ and $FF^{-1} \Rightarrow 1$ are monoidal isomorphisms \cite[2.4.10]{egno15}. Moreover, the natural isomorphisms can be chosen to satisfy the triangle identities of an adjunction \cite[\S4.4]{maclane97}.
\end{conv}
We now transport the anti-involution along the equivalence to the strict category.

\begin{lem} \label{strinv} 1. Let $\alpha: F_1 \rightarrow F_2$ be a monoidal isomorphism, and let $H,G$ be monoidal functors. The natural isomorphism $\id_H* \alpha *\id_G: HF_1G \Rightarrow HF_2G$ is monoidal.\\
2. Let $F: \catA \rightarrow \catB$ be an equivalence of monoidal categories, $\Phi: \catA \rightarrow \catA^{\tensop}$ a monoidal functor and $t: \Phi^2 \Rightarrow \id_\catA$ an isomorphism of monoidal functors. Then there exists a anti-involution $\Psi: \catB \rightarrow \catB^{\tensop}$, $\pi: \Psi^2 \Rightarrow \id_\catB$ such that $\Psi \circ F \cong F \circ \Phi$ are monoidally isomorphic and the diagram
\adjustbox{scale=.8}{
	\begin{tikzcd}
		F \Phi^2 \arrow[d] \arrow[r, "F(t)"]  & F \arrow[d,"id"] \\
		\Psi^2F \arrow[r,"\pi_F"] & F
	\end{tikzcd}}
	commutes.
\end{lem}

\begin{proof}
	The first statement is an easy verification. 
	For the second claim we define $\Psi := F \Phi F^{-1}$. As a composite of monoidal functors it is monoidal. Similarly, as a composite of adjoint equivalences it is an adjoint equivalence. 
	One easily verifies the natural isomorphism $\pi$ can defined as the composite
	\[ 
	\Psi^2 = F \Phi F^{-1} F \Phi F^{-1} \xrightarrow{\id_{F\Phi}* \eta^{-1}} F \Phi \Phi F^{-1} \xrightarrow{\id_F *t*\id_{F^{-1}}} F F^{-1} \xrightarrow{\ep} \id_\catA.
	\]
\end{proof}

Next we will strictify the anti-involution using a method adapted from \cite{calindo16}.
\begin{defn}
	Let $\catA$ be a strict monoidal category, with an anti-involution $(\Psi,\psi,\psi_0)$, $\pi: \Psi^2 \Rightarrow \id_\catA$. The category $\catA^{\Z_2}$ has as objects quadruples 
	\[(X_0,X_1, \eta_0: X_1 \rightarrow \Psi(X_0), \eta_1: X_0 \rightarrow \Psi(X_1))\] 
	where $X_i \in \catA$ and $\eta_i$ are isomorphisms in $\catA$ such that the left diagram in Figure \ref{eq:calindo} commutes for $i=0,1$.
	A morphism 
	\[\vec{f} = \begin{bmatrix} f_0 \\ f_1 \end{bmatrix}: (X_0,X_1,\eta_0,\eta_1) \rightarrow (Y_0,Y_1,\nu_0,\nu_1)\] consists of a pair $f_i: X_i \rightarrow Y_i$ of morphisms in $\catA$ such that the right diagram of Figure \ref{eq:calindo} commutes for $i=0,1$. 
	\begin{figure}[h]
		$\begin{tikzcd}
		X_{i+1} \arrow[d, "\eta_{i}"]\arrow[r,"\pi^{-1}"] & \Psi^2 (X_{i+1}) \\
		\Psi(X_{i})  \arrow[ru, swap, "\Psi(\eta_{i+1})"] &
		\end{tikzcd}$ \quad \quad \quad 
		$\begin{tikzcd}
		X_{i+1}  \arrow[d,"\eta_i"]\arrow[r,"f_{i+1}"] & Y_{i+1} \arrow[d,"\nu_i"] \\
		\Phi(X_i) \arrow[r, "\Phi(f_i)"] & \Phi(Y_i)
		\end{tikzcd}$
		\caption{Object and morphism diagrams in $\catA^{\Z_2}$.}
		\label{eq:calindo}
	\end{figure}
\end{defn}
\begin{nota}
	We will write $\vec{X}$ for a quadruple $(X_0,X_1,\eta_0,\eta_1)$, or use vector notation $(\begin{bmatrix}
	X_0 \\ X_1
	\end{bmatrix}, \begin{bmatrix}
	\eta_0 \\ \eta_1
	\end{bmatrix}$). We will sometimes refer to $X_i$ (resp. $\eta_i$) the object (resp. morphism) components of $\vec{X}$.
\end{nota}
\begin{lem}
	The category $\catA^{\Z_2}$ is strictly monoidal with tensor product \[\vec{X} \tensor \vec{Y} := \left(\begin{bmatrix} X_0 \tensor Y_0 \\ Y_1 \tensor X_0 \end{bmatrix}, \begin{bmatrix}  Y_1 \tensor X_1 \xrightarrow{\nu_0 \tensor \eta_0} \Phi(Y_0) \tensor \Phi(X_0) \xrightarrow{\phi}   \Phi(X_0 \tensor Y_0)\\ 
	X_0 \tensor Y_0 \xrightarrow{\eta_1 \tensor \nu_1} \Phi(X_1) \tensor \Phi(Y_1)   \xrightarrow{\phi}\Phi(Y_1 \tensor X_1) \end{bmatrix}\right),\]
	\[\vec{f} \tensor \vec{g} := \begin{bmatrix} f_0 \tensor g_0 \\ g_1 \tensor f_1 \end{bmatrix}\]
	and unit $\vec{\1} := (\1,\1,\phi_0,\phi_0)$. Moreover, the functor $\Psi_{st}: \catA^{\Z_2} \rightarrow \catA^{\Z_2}$ defined by 
	\begin{align*}
	\Psi_{st}\left(\begin{bmatrix}
	X_0 \\ X_1
	\end{bmatrix}, \begin{bmatrix}
	\eta_0 \\ \eta_1
	\end{bmatrix}\right) 
	= 
	\left(\begin{bmatrix}
	X_1 \\ X_0
	\end{bmatrix}, \begin{bmatrix}
	\eta_1 \\ \eta_0
	\end{bmatrix}\right)  \quad \text{ and } \quad 
	\Psi_{st}\begin{bmatrix} f_0 \\ f_1\end{bmatrix} = \begin{bmatrix} f_1 \\ f_0\end{bmatrix}
	\end{align*}
	is a strict anti-monoidal functor such that $\Psi_{st}^2 = \id$ is a monoidal isomorphism.
\end{lem} 
\begin{proof}
	A straightforward calculation.
\end{proof}
We can now prove a strictification theorem for monoidal categories with an anti-involution.
\begin{prop} \label{strmoninv} 1. The functor 
	\begin{align*}
	\Omega: &\catA \rightarrow \catA^{\Z_2}\\
	&X \mapsto (X,\Psi(X),\id_{\Psi(X)},\pi^{-1}_X), \quad \quad f \mapsto \begin{bmatrix} f \\ \Psi(f) \end{bmatrix}
	\end{align*}
	defines an equivalence of categories.\\ 
	2. The maps 
	\begin{align*}
	&(\Omega_2)_{X,Y} := \begin{bmatrix} \id \\ \Psi_2 \end{bmatrix} : \Omega(X)\tensor \Omega(Y) \rightarrow \Omega(X\tensor Y) \\
	&\Omega_0:= \begin{bmatrix} id \\ \Psi_0 \end{bmatrix}: \vec{\1} \rightarrow \Omega(\1)
	\end{align*}
	equip $\Omega$ with the structure of a monoidal functor. \\
	3. The natural transformation $\begin{bmatrix} \id \\ \pi \end{bmatrix}: \Psi_{st} \circ \Omega \Rightarrow \Omega \circ \Psi$ defines a monoidal isomorphism such that the diagram
	\adjustbox{scale=.85}{
		\begin{tikzcd}
			\Omega \Psi^2 \arrow[d] \arrow[r, "\Omega \pi"]  & \Omega \arrow[d,"id"] \\
			\Psi_{st}^2 \Omega \arrow[r,"\id"] & \Omega
		\end{tikzcd}}
commutes.
\end{prop}
\begin{proof}
Let us first prove $\Omega$ defines an equivalence. To see $\Omega(X)$ defines an object, we observe that indeed $\Psi(\id) \circ \pi^{-1}_X = \pi^{-1}_X$ and $\Psi(\pi_X^{-1}) \circ \id = \pi^{-1}_{\Psi(X)}$. That $\Omega(f)$ defines a morphism is clear. To show $\Omega$ is an equivalence we will check it's fully faithful and essentially surjective. Faithfulness is clear. For fullness, let $\vec{f}: \Omega(X) \rightarrow \Omega(Y)$ be any morphism. We have 
\adjustbox{scale=.8}{\begin{tikzcd}
	\Psi(X) \arrow[r,"\Psi(f_0)"] \arrow[d,"\id"] & \Psi(Y) \arrow[d,"\id"]\\
	\Psi(X) \arrow[r,"f_1"] &  \Psi(Y) 
\end{tikzcd}} by definition of $\vec{f}$ being a morphism. Then $\vec{f} = \Omega(f_0)$. Finally, $\Omega$ is essentially surjective since for any $\vec{X}$ we have the isomorphism $\begin{bmatrix} \id \\ \eta_0 \end{bmatrix}:\vec{X} \rightarrow \Omega(X_0)$.\\
The assertions of (2) are all easy checks that follow from $\pi$, $\psi$, $\psi_0$ being natural isomorphisms and monoidality of $\pi$ and $\Psi$.\\
Finally, to prove part (3) we note that the maps $\begin{bmatrix} \id \\ \pi \end{bmatrix}: \Omega \circ \Psi \Rightarrow \ep \circ \Omega$ define the components of a natural isomorphism. Monoidality can be checked componentwise, and is inherited from $\id$ and $\pi$ e.g. the diagram
\adjustbox{scale=.8}{
	\begin{tikzcd} \Psi^2Y \tensor \Psi^2X \arrow[r, "\pi \tensor \pi"] \arrow[d, "\Psi(\psi) \tensor \psi"] & Y \tensor X \arrow[d,"\id"] \\
		\Psi^2(Y \tensor X) \arrow[r,"\pi"] & Y \tensor X 
	\end{tikzcd} }
	commutes since $\pi$ is monoidal. The diagram in the theorem commutes since $\Omega(\pi) = (\pi, \Psi(\pi)) = (\pi, \pi_{\Psi}) = \ep( (id,\pi) ) \circ (\id,\pi)_{\Psi}$, where the second equality holds by assumption of $\Psi$ being an anti-involution.
\end{proof}

We can also transport the braiding $\sigma$ on $\catA$ to the strictification.  
\begin{lem} \label{braidtransport} \cite[Example 2.4]{js93}. 
	\begin{enumerate}
		\item Let $\catA$ and $\catA$ be monoidal categories and let $(F,\mu): \catA \rightarrow \catA$ be a monoidal equivalence. If $\catA$ is braided, there is a unique braiding on $\catA$ making $F$ braided monoidal.
		\item A monoidal equivalence $F$ is braided monoidal iff its quasi-inverse $F^{-1}$ is braided monoidal.
	\end{enumerate}
\end{lem}

\subsection{Strictification of the module category}

We will strictify $\catM$ by replacing it by the category $\catM_{st} := \Fun_\catA(\catA,\catM)$. This category is a strict module category over $\catA_{st}$ where the action is defined by precomposition i.e. $\act(F,S) := F \circ S$ for $F \in \catM_{st}$, $S \in \catA_{st}$.

\begin{lem} \label{moduleequiv} Let $\catM$ be a right $\catA$-module category. The functor $\bL: \catM \rightarrow \catM_{st}$ defined by $\bL(M) = (M\tensor -, (a)_{M,-,-})$ is an equivalence of categories.
\end{lem}

\begin{proof} Standard, see \cite[Remark 7.2.4]{egno15}.
\end{proof}

The equivalences $\bL$ and $L$ intertwine the actions of $\catA$ and $\catA^{str}$ on $\catM$ and $\catM^{str}$ respectively.
\begin{lem} \label{lambdab}
	Let $L: \catA \rightarrow \catA_{st}$, and $\bL: \catM \rightarrow \catM_{st}$ be as above. There exist natural ismorphisms $(\lambdab)_{M,X}: \bL(M \tensor X) \cong  \bL (M) \circ L(X)$ for $M\in \catM$ and $X\in \catA$ such that the diagrams
	\begin{center}
		\adjustbox{scale=.9}{
			\begin{tikzcd}
				\bL (M \tensor (X \tensor Y)) \arrow[d, "(\lambdab)_{M,X\tensor Y}"] \arrow[r, "\bL(a)"] & \bL ((M \tensor X) \tensor Y) \arrow[d, "(\lambdab)_{M\tensor X, Y}"] \\
				\bL(M) \circ L(X \tensor Y)  \arrow[d, "\id \hcomp \lambda^{-1}"] & \bL(M \tensor X) \circ L(Y) \arrow[d, "(\lambdab)_{M,X} \hcomp \id"] \\
				\bL(M)\circ (L(X) \circ L(Y) )  \arrow[r, "\id"] & (\bL(M)\circ L(X) ) \circ L(Y) 
			\end{tikzcd} \quad \quad
			\begin{tikzcd}
				\bL(M \tensor \1) \arrow[r,"\bL(r)"] \arrow[d, "(\lambdab)_{M,\1}"] & \bL(M) \arrow[dd, "\id"] \\
				\bL(M)\circ L(\1)  \arrow[d, "\id \hcomp \lambda_0^{-1}"] & \\
				\bL(M) \circ \id_\catA \arrow[r, "\id"] & \bL(M)
			\end{tikzcd}
		}
	\end{center}
	commute for any $M\in \catM$, $X,Y\in \catA$. 
\end{lem}
\begin{proof}
	The natural isomorphism $\lambdab$ is defined as follows
	\[(\lambdab)_{M,X}: \bL(M \tensor X) = (M \tensor X) \tensor - \xrightarrow{(a)_{M,X,-}^{-1}} M \tensor (X \tensor -) = \bL(M) \circ L(X).\]
	To check if the diagrams of natural transformations commutes we need to check if the components commute. If we evaluate the functors on some $D \in \catA$ the first diagram becomes the pentagon axiom of the module category $\catM$, whereas the second diagram becomes the triangle axiom of the module category. Thus the components define commuting diagrams for any $D$, and hence the diagrams of natural transformations commute.
\end{proof}

\begin{nota}
	We denote the quasi-inverse of $(L,L_2,L_0)$ by $(R,R_2,R_0)$ and let $\eta: 1 \Rightarrow RL$ and $\ep: LR \rightarrow 1$ be the monoidal isomorphisms of the equivalence $L$. Similarly, we have $\overline{\eta}: 1 \Rightarrow \bR \bL$, $\overline{\ep}: \bL \bR \Rightarrow 1$.
\end{nota}

\begin{lem} \label{rhob}
	Let $R: \catA_{st} \rightarrow \catA$, and $\bR: \catM_{st} \rightarrow \catM$ be as above. There exist natural ismorphisms $(\rhob)_{F,S}: \bR(F \circ S) \cong  \bR (F) \tensor R(S)$ for $F\in \catM_{st}$ and $S\in \catA_{st}$ such that the diagram
	\begin{center}
		\adjustbox{scale=.9}{
			\begin{tikzcd}
				\bR (F \circ (S \circ T)) \arrow[d, "(\rhob)_{F,S\circ T}"] \arrow[r, "\bR(\id)"] & \bR ((F \circ S) \circ T) \arrow[d, "(\rhob)_{F\circ S, T}"] \\
				\bR(F) \tensor R(S \circ T)  \arrow[d, "\id \tensor \rho^{-1}"] 
				& \bR(F \circ S) \tensor R(T) \arrow[d, "(\rhob)_{F,S} \tensor \id"] \\
				\bR(F)\tensor (R(S) \tensor R(T))  \arrow[r, "\alpha"] & (\bR(F)\tensor R(S) ) \tensor R(T) 
			\end{tikzcd} 
		}
	\end{center}
	commutes for any $F\in \catM_{st}$, $S,T\in \catA_{st}$. 
\end{lem}

\begin{proof}
	Since $\bL$ is full and faithful we can define isomorphisms $(\rhob)_{M,X}: \bR(M\circ X) \Rightarrow \bR(M) \tensor R(X)$ via the diagram 
	\begin{tikzcd}
		\bL \bR(M \circ X) \arrow[d,"\overline{\ep}"] \arrow[r,"\bL(\rhob)"]  & \bL ( \bR(M) \tensor \bR(X) ) \arrow[d, "\lambdab"] \\
		M \circ X & \bL \bR(M) \circ LR(X) \arrow[l,"\overline{\ep} \hcomp \ep"] 
	\end{tikzcd}. Then by naturality of $\lambdab$ and $\ep$, $\overline{\ep}$ we have $\bL(\rhob)\circ LR( \alpha * \beta) = L(R(\alpha)* R(\beta)) \circ \bL(\rhob)$. Faithfulness of $\bL$ then implies naturality of $\rhob$. Consider the diagram
	\begin{center}
		\adjustbox{scale=.66}{
			\begin{tikzcd}[column sep=small]
				FST & \bL\bR(FST) \arrow[l,swap,"\bep"] \arrow[r,"\id"] \arrow[d, "\bL(\rhob)"] & \bL\bR(FST) \arrow[d," \bL(\rhob)"] \arrow[r,"\bep"] & FST & \\
				\bL\bR F \circ LR(ST) \arrow[u, "\bep \hcomp \ep"] \arrow[d, "\id \hcomp L(R_2^{-1})"] & \bL(\bR F \tensor R(ST)) \arrow[l, "\lambdab"] \arrow[d,"\bL(\id \hcomp \R_2^{-1})"] & \bL(\bR(FS) \tensor R(T) ) \arrow[r,"\lambdab"] \arrow[d, "\bL(\rhob \tensor \id)"]  & \bL \bR (FS) \circ LR(T) \arrow[dd,"\bL(\rhob)\hcomp \id"] \arrow[u, "\bep \hcomp \ep"] \arrow[r,"\bep \hcomp \id"] & FS \circ LR(T) \arrow[lu, swap,"\id\hcomp \ep"] \\
				\bL\bR F \circ L(RS \tensor RT) \arrow[dd,"\id \hcomp L_2^{-1}"] & \bL(\bR F \tensor (RS \tensor RT) ) \arrow[l,"\lambdab"] \arrow[r, "\bL a"] \arrow[d, "\lambdab"] & \bL( (\bR F \tensor RS) \tensor RT) \arrow[rd, "\lambdab"] & & \\
				& \bL \bR F \circ L(RS \tensor RT) \arrow[ld, "\id \hcomp L_2^{-1}"] & & \bL(\bR F \tensor RS) \circ LRT \arrow[rd, "\lambdab \hcomp \id"] & \\
				\bL\bR F \circ LRS \circ LRT \arrow[rrrr,"\id"]& & & &  \bL \bR F \circ LRS \circ LRT \arrow[uuu,"\bep \hcomp \ep \hcomp \id"]
			\end{tikzcd}
		}
	\end{center}
	We wish to show the upper inner hexagon commutes. The other subdiagrams commute by definition of $\rhob$, naturality of $\lambdab$ and Lemma \ref{lambdab}. The outside of the diagram commutes since $\ep: LR \Rightarrow \id$ is monoidal. Therefore, the inner hexagon commutes. Faithfulness of $\bL$ implies $\rhob$ has the required property. 
\end{proof}

Note that the category $\catM_{st}$ is also a strict module category over $\catA_{st}^{\Z_2}$ with action $F \cdot \vec{X} := F \circ X_0$, for $\vec{X} = (X_0,X_1,\eta_0,\eta_1) \in \catA_{st}^{\Z_2}$, and  $F \in \catM_{st}$. 
Lemma \ref{lambdab} holds mutatis mutandis for the composite $\Omega L$ and $\bL$. To strictify the $\Z_2$-braided pair it remains to transport the $\Z_2$-monodromy along the equivalences to $\catA_{st}$ and $\catA_{st}^{\Z_2}$.

\begin{prop} \label{ktransport}
	Let $(\catA, \catM)$ be a $\Z_2$-braided pair. There is a $\Z_2$-monodromy $K$ on $(\catA_{st},\catM_{st})$ such that the diagram
	\begin{center}
		\begin{tikzcd}[column sep = huge]
			\bL (M \tensor X) \arrow[r, "\bL (\kappa_{M,X})"] \arrow[dd,"\lambdab"] & \bL (M \tensor \Phi(X)) \arrow[d,"\lambdab"]\\
			& \bL (M) \circ L(\Phi(X)) \arrow[d, "\id \hcomp L\Phi(\eta)"] \\
			\bL(M) \circ L(X) \arrow[r, "K_{\bL(M),L(X)}"] & \bL(M) \circ \Psi(L(X)) 
		\end{tikzcd}
	\end{center}
	commutes for any $M \in \catM$ and $X \in \catA$. 
\end{prop}

\begin{proof}  Let $U\in \catM_{st}$ and $S, T \in \catA_{st}$. Since $\bR$ is full and faithful the diagram
	\begin{center}
		\adjustbox{scale=.75}{
			\begin{tikzcd}[column sep = large]
				\bR (U)\tensor R (S)  \arrow[r, "\kappa_{RU, RS}"] & \bR(U) \tensor \Phi(R(S))  \\
				& \bR (U) \tensor R(\Psi(S)) \arrow[u,"\id \tensor \eta^{-1}" ] \\
				\bR(U \circ S) \arrow[uu, "\rhob"] \arrow[r,dotted,"\bR(K_{U,S})"] & \bR(U \circ \Psi(S))  \arrow[u, "\rhob"]
			\end{tikzcd}
		}
	\end{center}
	defines the components $K_{U,S}$ uniquely. Naturality of $\eta, \rhob$ and $\kappa$ implies that $\bR(K) \circ \bR(\alpha \hcomp \beta) = \bR(\alpha \hcomp \Psi(\beta)) \circ \bR(K)$. Faithfulness of $\bR$ then implies naturality of $K$. To show $K$ defines a $\Z_2$-monodromy we need to check it satisfies axioms BP1 and BP2. We consider the diagram in  Figure \ref{fig:z2diagram}.
	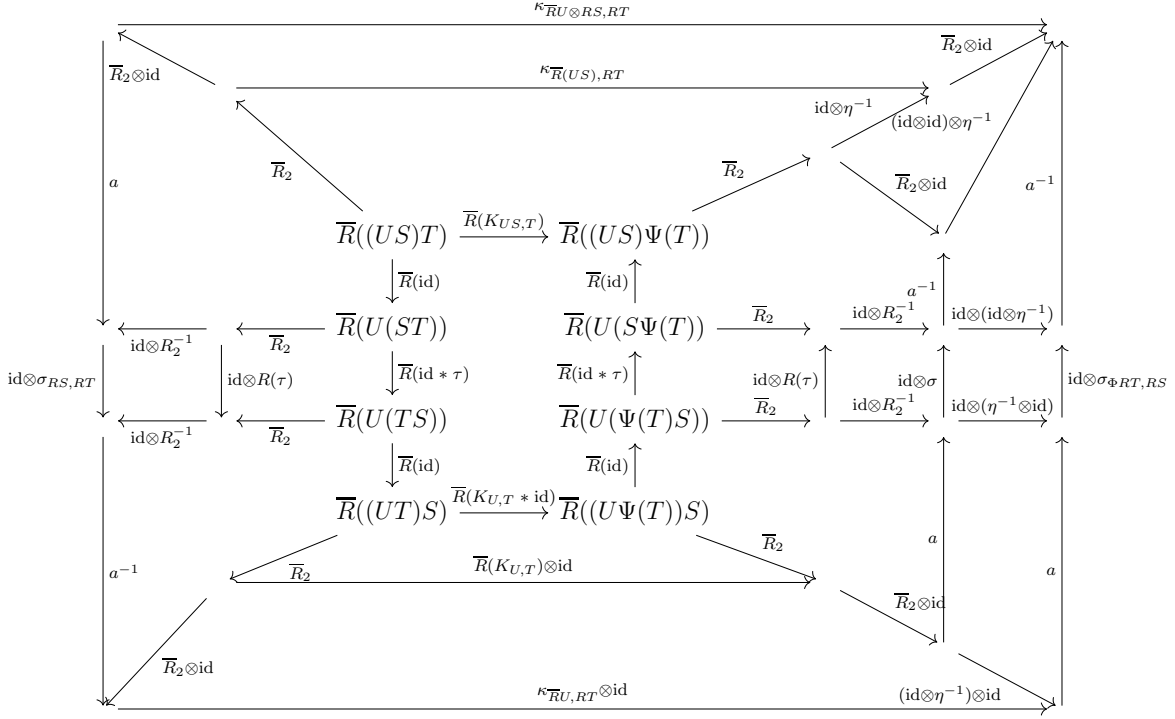
\begin{figure}
		\begin{center}
			\adjustbox{scale=.8}{
				\begin{tikzcd}[column sep=large]
					\; \arrow[dddd,"a"] \arrow[rrrrrr,"\kappa_{\bR U \tensor RS,RT}"]& & & & & & \; \\
					&\; \arrow[lu,"\rhob\tensor \id"] \arrow[rrrr,"\kappa_{\bR(US),RT}"] & & & & \; \arrow[ru,"\rhob \tensor \id"]& \\
					& & & & \; \arrow[ru,"\id\tensor \eta^{-1}"] \arrow[rd,"\rhob \tensor \id"] & & \; \\
					& & \bR((US)T) \arrow[luu,"\rhob"] \arrow[r,"\bR(K_{US,T})"] \arrow[d,"\bR(\id)"] & \bR((US)\Psi(T)) \arrow[ru,"\rhob"]  &\;  &\; \arrow[ruuu,"(\id\tensor \id)\tensor \eta^{-1}"] & \\
					\; \arrow[d,swap, "\id \tensor \sigma_{RS,RT}"] & \; \arrow[l, "\id \tensor R_2^{-1}"] \arrow[d,"\id \tensor R(\tau)"] & \bR(U(ST)) \arrow[l,"\rhob"] \arrow[d,"\bR(\id \hcomp \tau)"] & \bR(U(S\Psi(T)) \arrow[u,"\bR(\id)"] \arrow[r,"\rhob"] &\; \arrow[r, "\id \tensor R_2^{-1}"] &\; \arrow[u,"a^{-1}"] \arrow[r,"\id \tensor (\id \tensor \eta^{-1})"] & \; \arrow[uuuu,"a^{-1}"] \\
					\; \arrow[dddd, "a^{-1}"]& \; \arrow[l, "\id \tensor R_2^{-1}"] &\bR(U(TS)) \arrow[l,"\rhob"] \arrow[d,"\bR(\id)"] &  \bR(U(\Psi(T)S)) \arrow[u,"\bR(\id \hcomp \tau)"] \arrow[r,"\rhob"] &\; \arrow[u,"\id \tensor R(\tau)"] \arrow[r, "\id \tensor R_2^{-1}"] & \; \arrow[u,"\id \tensor \sigma"] \arrow[r,"\id \tensor (\eta^{-1}\tensor \id)"] & \; \arrow[u,swap,"\id \tensor \sigma_{\Phi RT,RS}"] \\
					& &\bR((UT)S) \arrow[r,"\bR(K_{U,T}\hcomp \id)"] \arrow[ld,"\rhob"] & \bR((U\Psi(T))S) \arrow[u,"\bR(\id)"] \arrow[rd,"\rhob"] & & & \\
					&\; \arrow[ldd,"\rhob \tensor \id"] \arrow[rrr,"\bR(K_{U,T})\tensor \id"] & & & \; \arrow[rd,"\rhob \tensor \id"]& & \\
					& & & & & \; \arrow[rd,swap, "(\id \tensor \eta^{-1})\tensor \id"] \arrow[uuu,"a"] & \\
					\; \arrow[rrrrrr,"\kappa_{\bR U,RT}\tensor \id"]& & & & & &\; \arrow[uuuu,"a"] \\
				\end{tikzcd}
			}
		\end{center}
		\caption{A diagram chase to prove $K$ satisfies axiom BP1.}
		\label{fig:z2diagram}
	\end{figure}
	We wish to show the inner octagon commutes. The outer diagram commutes since $\kappa$ is a $\Z_2$-monodromy. The other subdiagrams commute by naturality, Lemma \ref{rhob}, definition of $K$, and $R$ being braided monoidal. We conclude the inner octagon commutes. By faithfulness of $\bR$ we find that $K$ satisfies axiom BP1. 
	Similarly, one can write down a diagram to show $K$ also satisfies BP2. We leave this to the reader, or see \cite{thesis}.
	It remains to show $K$ satisfies the diagram in the Proposition. Let $M\in \catM$, $X \in \catA$ and consider the diagram
	\begin{center}
		\adjustbox{scale=.65}{
			\begin{tikzcd}[column sep = huge]
				\bL M \circ LX \arrow[rd,"\bL\etab \hcomp L\eta"] \arrow[dddd,"\id"] && \bL(M\tensor X) \arrow[ll,"\lambdab"] \arrow[r,"\bL (\kappa_{M,X})"] \arrow[d,"\bL (\etab \tensor \eta)"]& \bL (M \tensor \Phi(X)) \arrow[rr,"\lambdab"] \arrow[d,"\bL (\etab \tensor \Phi \eta)"] & & \bL M \tensor L\Phi (X) \arrow[ld,"\bL \etab \hcomp L\Phi \eta"]  \arrow[dddd,"\id \hcomp L\Phi (\eta)"]\\
				&\; \arrow[lddd,"\bep\hcomp \ep"]&\bL (\bR\bL M \tensor RLX) \arrow[l,"\lambdab"] \arrow[r,"\bL (\kappa_{\bR\bL M,RLX})"] & \bL (\bR\bL M \tensor \Phi RLX) \arrow[r,"\lambdab"] &\;&\\
				&&&\bL (\bR\bL M \tensor R\Psi LX) \arrow[u,"\bL (\id \tensor \eta^{-1})"] \arrow[r,"\lambdab"] & \; \arrow[u,"\id \hcomp L \eta^{-1}"] \arrow[rdd,"\bep \hcomp \ep"]&\\
				&&\bL\bR (\bL M \circ LX) \arrow[r,"\bL\bR (K_{LM,LX})"] \arrow[lld,"\bep"] \arrow[uu,"\bL (\rhob)"] & \bL\bR (\bL M \circ \Psi LX) \arrow[u,"\bL (\rhob)"] \arrow[rrd,"\bep"]&&\\
				\bL M \circ LX \arrow[rrrrr,"K_{\bL M,LX}"] &&&&& \bL M \circ \Psi LX
			\end{tikzcd}
		}
	\end{center}
	The subdiagrams commute by naturality, the triangle identity, the definition of $\rhob$, and the definition of $K$. Therefore the outer diagram commutes, which is what we wished to show.
\end{proof}

Finally, we transport the $\Z_2$-cylinder braiding from $\catA_{st}$ to $\catA_{st}^{\Z_2}$. Namely, let $\vec{S} = (S_0,S_1,\eta_0^S, \eta_1^S)\in \catA_{st}^{\Z_2}$ and $F \in \catM_{st}$ we define $\vec{K}_{F,\vec{S}}$ as the composite 
\[F \cdot \vec{S} = F S_0  \xrightarrow{K_{F,S_0}} F \Psi(S_0) \xrightarrow{\id \hcomp \eta_0^S} F S_1 = F \cdot \Psi_{st}(\vec{S}).\] 
One easily verifies that $\vec{K}$ is a $\Z_2$-cylinder braiding. Thus $(\catA_{st}^{\Z_2},\catM_{st})$ is a strict $\Z_2$-braided pair.

\subsection{Coherence as a corollory to strictification}

Our strictification results now yield the coherence theorem as an easy corollary.

\begin{thm} \label{braidcoh}
	Let $(\catA,\catM)$ be a $\Z_2$-braided pair and consider two parallel structural isomorphisms $T_1 \xrightarrow{f,g} T_2$ in $\catA$ (or $\catM$). Then $f=g$ if the underlying (cylinder) braids are equal $\beta_f = \beta_g$. 
\end{thm}

\begin{proof}
	We will provide the prove when $f$ and $g$ are structural isomorphisms with a given presentation in $\catM$. The case of structural isomorphisms in $\catA$ is analogous and left to the reader. 
	We consider the diagram with sides $\bL(f)$ and $\bL(g)$ in $\catM_{st}$. 
	The presentation of $f$ and $g$ consists of a vertical composition of horizontal compositions of the $a,\kappa,r$, etcetera. 
	Below each arrow of the presentation we build a rectangle using the diagrams in Proposition \ref{strictifymon}, Lemma \ref{strinv}.2, Proposition \ref{strmoninv}.3, Lemma \ref{lambdab}, Lemma \ref{rhob} and Proposition \ref{ktransport} so that we obtain a prism of natural isomorphisms. 
	The lower face of the prism consists of two structural isomorphisms $f_{st}$ and $g_{st}$ in $\catA_{st}^{\Z_2}$, whereas the upper face of the prism is our original diagram containing $f$ and $g$. 
	It is clear from these rectangles that if $\beta_f= \beta_g$, then also $\beta_{f_{st}}=\beta_{g_{st}}$. 
	By assumption $\beta_f=\beta_g$ so that Proposition \ref{strz2coh} implies that $f_{st} = g_{st}$. 
	Then $\bL(f)= \bL(g)$ follows since the prism has vertical faces consisting of commutative diagrams of isomorphisms, and we just proved that the lower face commutes. By faithfulness of $\bL$ we conclude $f=g$.
\end{proof}

We also have the following two coherence results:
\begin{thm} \label{moncoh}
	Let $(\catA,\catM)$ be a strict $\Z_2$-monoidal pair, any two parallel structural isomorphisms are equal
\end{thm}
\begin{proof} 
	This is clear in a strict $\Z_2$-monoidal pair since all horizontal composition of identities are identities. Now we deduce coherence from the strictification results, following the technique of the proof of Theorem \ref{braidcoh}
\end{proof}

\begin{thm} \label{symcoh}
	Let $(\catA,\catM)$ be a strict $\Z_2$-symmetric pair, any two parallel structural isomorphisms are equal
\end{thm}
\begin{proof}
	Note that the strictification $(\catA_{st}^{\Z_2},\catM_{st})$ is a strict $\Z_2$-symmetric pair. We leave the proof, an easy adaptation of the proofs of Proposition \ref{strz2coh} and Theorem \ref{braidcoh}, to the reader.
\end{proof}

\section{Classifying categorical algebras} \label{sec:rexalg}
\begin{nota} 
	Let $\kk$ be a field. We denote with $\vect$ the category of $\kk$-vector spaces. 
\end{nota}
We now analyse categorical algebras in $\Rex$, a setting of $\kk$-linear categories well suited to higher algebra. By a \emph{$\kk$-linear category} we mean a category enriched over $\vect$. 
We introduce $\Rex$ following the exposition given in \cite{bzbj}.
\begin{nota} 
	We denote with $\fdvect$ the $\kk$-linear category of finite dimensional $\kk$-vector spaces.
\end{nota}
\begin{nota}
	For a (higher) category $\catC$ and $X,Y\in \catC$ we sometimes write $\catC(X,Y)$ for $\mathrm{Hom}_{\catC}(X,Y)$.
\end{nota}
Recall that a functor is called \emph{right exact} if it preserves finite colimits. A category is \emph{essentially small} if it is equivalent to a small category.\footnote{A category is called \emph{small} if it has a set of objects and a set of morphisms rather than proper classes.}
\begin{defn}
	$\Rex$ is the $2$-category of $\kk$-linear essentially small categories that admit finite colimits with morphisms right exact functors and 2-morphisms natural isomorphisms.
\end{defn}

\begin{ex}
	For any $\kk$-algebra $A$ the category $A\fdmod$ is in $\Rex$.	
\end{ex}
For $\catC, \catD, \mathcal{E}$ $k$-linear we let $\text{Bilin}(\catC \times \catD, \mathcal{E})$ denote the category of $\kk$-bilinear functors from $\catC \times \catD$ to $\mathcal{E}$ that preserve finite colimits in each variable separately.
\begin{defn} 
	The \emph{Deligne-Kelly tensor product} $\catC \boxt \catD$ for $\catC,\catD \in \Rex$ is uniquely characterised by the natural equivalence
	\[ \Rex( \catC \boxt \catD, \mathcal{E}) \simeq \text{Bilin}(\catC \times \catD, \mathcal{E}).\]
\end{defn} 

\begin{ex} 
	If $A$ and $B$ are $\kk$-algebras then $A\fdmod \boxt B \fdmod \simeq (A \tensor_{\kk} B)\fdmod$. 
\end{ex}

\begin{conv}
	From now on any functor between categories in $\Rex$ will be assumed to be $\kk$-linear and right exact without further comment. 
\end{conv}
Two operations in $\ztd_n$ are called \emph{isotopic} if there is a path in the space of operations connecting them; we will refer to this path as an \emph{isotopy}.  
If $\beta: g\Rightarrow g'$ is an isotopy of operations and $(\gamma_i: f_i \Rightarrow h_i)_{i=1}^k$ a collection of isotopies then we call the associated isotopy 
\[ \beta * (\gamma_1\amalg \dots \amalg \gamma_k): g \circ (f_1 \amalg \dots \amalg f_k) \Rightarrow g' \circ (h_1 \amalg \dots \amalg h_k) \] 
the \emph{horizontal composition}.  A \emph{2-isotopy} between two isotopies is a homotopy between the two paths in the space of operations. 
\begin{nota}
	For a permutation $\sigma \in S_k$ we abuse notation by also writing $\sigma$ for the induced permutation maps $\sigma: \DD^{\amalg k} \rightarrow \DD^{\amalg k}$ and $\sigma: \catA^{\boxt k} \rightarrow \catA^{\boxt k}$ where $\catA\in \Rex$.
	More generally, for a collection of operations $(f_i)_{i=1}^{k}$ (resp. functors $(F_i)_{i=1}^{k}$) in a composite $\sigma \circ (\amalg_i f_i)$ and $\sigma \circ (\boxt_i F_i)$ the map (resp. functor) $\sigma$ permutes the codomains of the $f_i$ (resp. $F_i$).
\end{nota}
\begin{defprop} \label{rexalgdef}
	An $\ztd_n$-algebra $F$ in $\Rex$ is a choice of categories $\catA, \catM \in \Rex$ together with assignments 
	\begin{enumerate}
		\item For every operation $f$ assigned functors $F(f)$ as follows
		\begin{align*}
		&f \in \ztd_n( \DD^{\amalg k}, \DD) \mapsto F(f): \catA^{\boxt k} \rightarrow \catA,\\
		&f \in \ztd_n( \DD_* \amalg \DD^{\amalg k}, \DD_*) \mapsto F(f): \catM \boxt \catA^{\boxt k} \rightarrow \catM, \\
		&f \in \ztd_n( \DD^{\amalg k}, \DD_*) \mapsto F(f): \catA^{\boxt k} \rightarrow \catM.
		\end{align*}
		where $\catA^{\boxt 0 }:= \fdvect$.
		\item For every isotopy between compositions of operations $(g,(f_i)_{i=1}^k,h)$ and permutation $\sigma \in S_k$ an assigned natural isomorphism as follows
		\[
		\begin{tikzcd}[row sep=small,column sep=large]
		& \bullet \arrow[dd, "g"]\\
		\bullet \arrow[ur,"\sigma \circ (f_1\amalg \dots \amalg f_k)"] \arrow[dr, "h"{name=U, below}]{} & \\
		& \bullet  \arrow[Rightarrow, "\gamma" near start, from=U,uu, start anchor={[xshift=1ex,yshift=1ex]},
		end anchor={[xshift=-1ex,yshift=-1ex]}]
		\end{tikzcd} 
		\quad \quad \mapsto \quad \quad 
		\begin{tikzcd}[row sep=small,column sep=large]
		& \bullet \arrow[dd, "F(g)"]\\
		\bullet \arrow[ur,"\sigma \circ (F(f_1) \boxt \dots \boxt F(f_k))"] \arrow[dr, "F(h)"{name=U, below}]{} & \\
		& \bullet \arrow[Rightarrow, "F(\gamma)" near start, from=U,uu, start anchor={[xshift=1ex,yshift=1ex]},
		end anchor={[xshift=-1ex,yshift=-1ex]}]
		\end{tikzcd}
		\]
	\end{enumerate}
	Subject to the conditions that: 
	\begin{enumerate}[label=(\roman*)]
		\item identity operations get assigned identity functors, 
		\item for every operation $f$ the constant isotopies \begin{tikzcd} & \bullet \\ \bullet \arrow[ru,"f"] \arrow[r,"f"] & \bullet \arrow[u,"\id"] \end{tikzcd} and \begin{tikzcd} & \bullet \\ \bullet \arrow[ru,"f"] \arrow[r,"\id"] & \bullet \arrow[u,"f"] \end{tikzcd} get assigned the identity natural isomorphism $\id_{F(f)}$, 
		\item \label{con3} and for every diagram of operations and isotopies
		\begin{center}
			\begin{tikzcd}[column sep = 70pt,row sep =large]
				k \arrow[r, "\alpha"] \arrow[d, "\gamma"] & m \circ (\sigma_h \circ \amalg_n h_n) \arrow[d, "\id_m * (id_{\sigma_h} *\amalg \beta_n)"]\\
				l \circ (\sigma_f \circ \amalg_j f_j) \arrow[r, "\delta * (\id_{\sigma_f}* \amalg \id_{f_j})"] & m \circ (\sigma_g \circ \amalg_n g_n)\circ (\sigma_f \circ \amalg_j f_j)
			\end{tikzcd}
		\end{center}
		such that there exists a 2-isotopy between the composite isotopies filling the square then the following equation holds between the assigned natural isomorphisms
		\begin{align} \Big( \id_{F(m)} * (\id_{\sigma_h} * \boxt_n F(\beta_n) \Big) \;  \circ \;  F(\alpha) =  \Big( F(\delta) * (\id_{\sigma_f} * \boxt_j \id_{F(f_j)}) \Big) \; \circ \;  F(\gamma). \label{2isoeq} \end{align}
	\end{enumerate}
\end{defprop} 
\begin{rmk}
	We will view $\ztd_n$ as an $\infty$-operad to define $\ztd_n$-algebras in general. Definition-Proposition \ref{rexalgdef} unpacks that definition for $\Rex$.
	The proof is delayed until \S7.
\end{rmk}
Various special cases of condition \ref{con3} determine the `pseudofunctoriality' of $\ztd_n$-algebras:
\begin{lem} \label{pseudofun} Let $F$ be a $\ztd_n$-algebra in $\Rex$.
	\begin{enumerate}
		\item Let $\alpha: f \simeq g$ be an isotopy between operations, then $F(\begin{tikzcd} & \bullet \\ \bullet \arrow[ru,"f"] \arrow[r,"g"] & \bullet \arrow[u,"\id"] \end{tikzcd})$ and $F(\begin{tikzcd} & \bullet \\ \bullet \arrow[ru,"f"] \arrow[r,"\id"] & \bullet \arrow[u,"g"] \end{tikzcd})$ are assigned the same natural isomorphism denoted $F(\alpha)$. 
		\item Let $\alpha^{-1}$ denote the inverse isotopy of $\alpha: f\simeq g$, then $F(\alpha^{-1}) = F(\alpha)^{-1}$.
		\item If $\alpha$ and $\beta$ are isotopies that are 2-isotopic, then $F(\alpha) = F(\beta)$.
		\item For any choice of composite $\beta \circ \alpha$ of two isotopies $\beta$, $\alpha$ we have $F(\beta \circ \alpha) = F(\beta) \circ F(\alpha)$.
		\item The constant isotopy 
		$\begin{tikzcd} 
		& \bullet \\ 
		\bullet \arrow[ru,"m \circ(\sigma_g \circ \amalg_n g_n)"] \arrow[r,swap,"\sigma_g \circ \amalg_n g_n"] & \bullet \arrow[u,"m"] 
		\end{tikzcd}$
		is assigned a natural isomorphism 
		\[F\big(m \circ (\sigma_g \circ \amalg_n g_n) \big) \cong F(m) \circ \big(\sigma_g \circ \boxt_n F(g_n) \big).\] 
		\item The two composites of natural isomorphisms between $F\big(m \circ (\sigma_g \circ \amalg_n g_n) \circ (\sigma_f \circ \amalg _i f_j))$ and $F(m) \circ \big(\sigma_g \circ \boxt_n F(g_n) \big) \circ \big(\sigma_f \circ \boxt_j F(f_j))$ are equal. 
		\item The diagram
		\[
		\begin{tikzcd} F\big(m_1 \circ (\sigma \circ \amalg_n g_n) \big) \arrow[r,"\cong"] \arrow[d,"F\big(\chi * (\id_{\sigma}* \amalg_n \beta_n) \big)"] & F(m_1) \circ \big(\sigma \circ \boxt_n F(g_n) \big) \arrow[d,"F(\chi) * \big(\id_{\sigma} * \boxt_n F(\beta_n)\big)"] \\
		F\big(m_2 \circ (\sigma \circ \amalg_n h_n) \big) \arrow[r,"\cong"] & F(m_2) \circ \big(\sigma \circ \boxt_n F (h_n) \big)
		\end{tikzcd}
		\]
		commutes
	\end{enumerate}
\end{lem}

\begin{conv}
	A right exact functor $F: \fdvect \rightarrow \catC$ in $\Rex$ is uniquely determined by the image of $\kk \in \fdvect$. We will identify $F$ with $F(\kk) \in \catC$ without further comment. 
\end{conv}

In the coming three subsections we will construct assignments of $\kk$-linear categorical $\ztd_n$-algebras for  $n=1$, $n=2$, $n\geq 3$ out of  $\kk$-linear $\Z_2$-monoidal, $\Z_2$-braided, $\Z_2$-symmetric pairs, and vice versa. 
That these assignments are inverse equivalences is shown in Section \ref{sec:genalg}.
\begin{rmk}
	We delay proving the assignments are equivalences because we don't want to write down all the data here that determines an isomorphism of $\ztd_n$-algebras in $\Rex$. 
\end{rmk}
\subsection{$\Z_2$-tensor pairs versus $\ztd_1$-algebras} \label{subsec:tensorcat}

By a \emph{tensor category} we will a monoidal category $\catA$ such that $\catA\in \Rex$ and $\tensor \in \text{Bilin}(\catA \times \catA, \catA)$.  
\begin{rmk}
	By definition of the Deligne-Kelly product $\boxt$ the tensor product $\tensor$ corresponds to a functor $\tensor: \catA \boxt \catA \rightarrow \catA$. 
	A tensor category is thus exactly an $E_1$-algebra in $\Rex$.
\end{rmk}
A right module category over a tensor category $\catA$ is an $\catA$-module category $\catM$ such that $\catM\in \Rex$ and $\act \in \text{Bilin}(\catM \times \catA, \catM)$.

\begin{defn}
	A \emph{$\Z_2$-tensor pair} consists of a tensor category $\catA$ with an anti-involution $\Phi$ and a right $\catA$-module category $\catM$ with a pointing $\1_\catM \in \catM$.
\end{defn}

\begin{rmk} \label{klincoh}
	One obtains a $\Z_2$-monoidal pair from a $\Z_2$-tensor pair by forgetting $\kk$-enrichment.
	As forgetting $\kk$-enrichment is \emph{faithful}\footnote{Two morphisms in a $\kk$-linear category are equal iff they are equal in the underlying plain category.} 
	Theorems \ref{braidcoh}, \ref{moncoh}, \ref{symcoh} automatically yield coherence theorems for their $\kk$-linear tensor category analogues.
\end{rmk}

\begin{prop} \label{prop:1rn}
	To a $\ztd_1$-algebra $F$ in $\Rex$ there is an associated $\Z_2$-tensor pair $(\catA_F,\catM_F)$.
\end{prop}
\begin{proof}
	We set $\catA_F := \catA$ and $\catM_F := \catM$.
	Consider the fixed operations $f_\tensor$, $\id_\DD$, $f_\Phi$, $f_\1$, $f_\act$, $\id_{\DD_*}$, $f_{\1_{\catM}}$ in Table \ref{table:ztd1op}.
	The operations are assigned functors that for which we introduce the following suggestive notation, respectively: $\tensor$, $\id_\catA$, $\Phi$, $\1$, $\act$, $\id_{\catM}$ and $\1_{\catM}$.
	Note that by Definition-Proposition \ref{rexalgdef} $(i)$ the functors $\id_\catA$ and $\id_\catM$ are indeed the identity functors on $\catA$ and $\catM$. 
	
	Next we need to show there exist natural isomorphisms $\alpha, \lambda,\rho,\Phi_2,\Phi_0,t,r,a$ that are part of the definition of a $\Z_2$-tensor pair. 
	One can define all the isomorphisms analogously, therefore we only define $\alpha$ leaving the others to the reader.
	Consider the composite operations $f_{\tensor} \circ (f_{\tensor} \amalg \id_\DD)$ and $f_{\tensor} \circ (\id_{\DD} \amalg f_{\tensor})$. 
	Observe that the two composites define isotopic operations since they lie in the same connected components lying over $(0,0,0,\id) \in \{0,1\}^{\times 3} \times S_3$ in the notation of Lemma \ref{order1}.
	Choose an isotopy $f_{\alpha}$ between them, then by Lemma \ref{pseudofun}.1 we have an associated natural isomorphism $F\big(f_{\alpha}): F(f_{\tensor} \circ (f_{\tensor} \amalg \id_\DD) \big) \cong F\big(f_{\tensor} \circ (\id_{\DD} \amalg f_{\tensor}\big)$.
	Using the natural isomorphisms of Lemma \ref{pseudofun}.5 we can define a natural isomorphism $\alpha$ by requiring that the diagram
	\[
	\begin{tikzcd}
	F(f_{\tensor} \circ (f_{\tensor} \amalg \id_\DD) \big) \arrow[r,"\cong"] \arrow[d,"F(f_\alpha)"] & F(f_{\tensor}) \circ \big( F(f_{\tensor}) \boxt F(\id_{\DD}) \big) = \tensor \circ (\tensor \boxt \id_{\catA}) \arrow[d,"\alpha"]\\
	F\big(f_{\tensor} \circ (\id_{\DD} \amalg f_{\tensor}) \big) \arrow[r,"\cong"] & F(f_{\tensor}) \circ \big( F(\id_{\DD}) \boxt F(f_{\tensor}) \big) = \tensor \circ (\id_\catA \boxt \tensor)
	\end{tikzcd}
	\]
	commutes.
	
	To show that our functors and natural isomorphisms equip $(\catA,\catM)$ with the structure of a $\Z_2$-tensor pair, we need to check that the natural isomorphisms satisfy the defining equations of a $\Z_2$-tensor pair. 
	Since all such equations can be treated similarly, we will verify one case in detail leaving the others to the reader. 
	We will verify that $\alpha$, $\lambda$ and $\rho$ satisfy the triangle axiom of a monoidal category.
	By definition the triangle axiom requires that for every $X,Y \in \catA$  the diagram 
	\[ 
	\begin{tikzcd}
	(X \tensor \1) \tensor Y \arrow[dr, swap,"\rho_X \tensor \id_Y"] \arrow[rr,"\alpha"] & & X \tensor  (\1 \tensor Y) \arrow[dl, "\id_X \tensor \lambda_Y"] \\
	& X \tensor Y &
	\end{tikzcd} 
	\]
	commutes. 
	This is equivalent to asking that the equation 
	\begin{align} \label{trieq}
	\id_{\tensor} * (\id_{\id_{\catA}} \boxt \rho) = \big(\id_{\tensor} *  (\lambda \boxt \id_{\id_{\catA}}) \big) \; \circ \; \big( \alpha * \id_{\id_{\catA} \boxt \1 \boxt \id_{\catA}} \big) 
	\end{align}
	holds.
	Let us consider the isotopy $\id_{f_{\tensor}} * (\id_{\id_{\DD}} \amalg f_\rho)$ and choose a composite isotopy  
	$\big(\id_{f_{\tensor}} *  (f_\lambda \amalg \id_{\id_{\DD}}) \big)  \circ \big( f_\alpha * \id_{\id_{\DD} \amalg f_\1 \amalg \id_{\DD}} \big)$.
	Note that both these isotopies are paths between the same operations, and recall that by Lemmas \ref{htype} and \ref{order1} the space of operations $\ztd_n(\DD^3,\DD)$ contracts onto a discrete space.
	We conclude the two isotopies are 2-isotopic, so that their assigned natural isomorphisms are equal by Lemma \ref{pseudofun}.3.
	Now consider the diagram
	\[
	\begin{tikzcd}[column sep =77pt]
	\tensor \circ (\tensor \boxt \id_{\catA}) \circ (\id_\catA \boxt \1 \boxt \id_{\catA}) \arrow[rrrd, bend left=10,"\id_{\tensor} * (\id_{\id_{\catA}} \boxt \rho)"] \arrow[ddd,swap,"\alpha*\id_{\id_{\catA} \boxt \1 \boxt \id_{\catA}}"]& & & \\
	&  \bullet \arrow[lu,"\cong"] \arrow[r,"F\big(\id_{f_{\tensor}} * (\id_{\id_{\DD}} \amalg f_\rho)\big)"] \arrow[d,swap,"F\big( f_\alpha * \id_{\id_{\DD} \amalg f_\1 \amalg \id_{\DD}} \big)"] &  F(f_{\tensor}) \arrow[r,"="]& \tensor \\
	& \bullet \arrow[ld,"\cong"] \arrow[ru,swap,"F\big(\id_{f_{\tensor}} *  (f_\lambda \amalg \id_{\id_{\DD}}) \big)"] &  & \\
	\tensor \circ (\id_{\catA} \boxt \tensor) \circ (\id_\catA \boxt \1 \boxt \id_{\catA})\arrow[rrruu, bend right=15,swap,"\id_{\tensor} *  (\lambda \boxt \id_{\id_{\catA}})"] & & & 
	\end{tikzcd}
	\]
	and note that the arrows $\cong$ are unambiguous by Lemma \ref{pseudofun} part 6. 
	Then all inner squares commute by Lemma \ref{pseudofun} and the definitions of $\lambda$, $\rho$ and $\alpha$. 
	The inner triangle commutes by Lemma \ref{pseudofun} parts 3 and 4.
	Hence the outer sides of the diagram commute which expresses Equation \eqref{trieq} holds.
\end{proof}

\begin{table}[h]
	\caption{Certain operations in the $\ztd_1$-operad are fixed by their visualisations.}
	\begin{tabular}{| l | l| } \hline
		Operation & Visualisation \\
		\hline
		$f_\tensor: \DD \amalg \DD \rightarrow \DD$ &  
		\begin{tikzpicture} 
		\draw [(-),teal,dotted] (0,0) -- (3.5,0);
		\draw [(-),purple,dotted] (4,0) --(7.5,0);
		\draw [(-),teal,thick] (0.5,0) -- (1.5,0);
		\draw [(-),teal,thick] (2,0) -- (3,0);
		\draw [(-),purple,thick] (4.5,0) -- (5.5,0);
		\draw [(-),purple,thick] (6,0) -- (7,0);
		\node at (1,.5) {$1_b$};
		\node at (2.5,.5) {$2_b$};
		\node at (5,.5) {$2_b$};
		\node at (6.5,.5) {$1_b$};
		\end{tikzpicture} \\ \hline
		$\id_\DD: \DD \rightarrow \DD$ &
		\begin{tikzpicture} 
		\draw [(-),teal] (0,0) -- (3.5,0);
		\draw [(-),purple] (4,0) --(7.5,0);
		\end{tikzpicture}  \\ \hline
		$f_\Phi: \DD \rightarrow \DD$ &
		\begin{tikzpicture} 
		\draw [(-),teal,dotted] (0,0) -- (3.5,0);
		\draw [(-),purple,dotted] (4,0) --(7.5,0);
		\draw [(-),purple,thick] (0.5,0) -- (3,0);
		\draw [(-),teal,thick] (4.5,0) --(7,0);
		\end{tikzpicture}  \\ \hline 
		$f_{\1}: \emptyset \rightarrow \DD$ &
		\begin{tikzpicture} 
		\draw [(-),teal,dotted] (0,0) -- (3.5,0);
		\draw [(-),purple,dotted] (4,0) --(7.5,0);
		\end{tikzpicture}  \\ \hline
		$f_{\act}: \DD_* \amalg \DD \rightarrow \DD_*$ &
		\begin{tikzpicture} 
		\draw [(-),dotted] (0,0) -- (5,0);
		\draw [(-),purple,thick] (0.5,0) -- (1.5,0);
		\draw [(-),black,thick] (2,0) -- (3,0);
		\draw [(-),teal,thick] (3.5,0) -- (4.5,0);
		\node at (2.5,0) {$\ast$};
		\end{tikzpicture} \\ \hline
		$\id_{\DD_*}: \DD_* \rightarrow \DD_*$ &
		\begin{tikzpicture} 
		\draw [(-),thick] (0,0) -- (5,0);
		\node at (2.5,0) {$\ast$};
		\end{tikzpicture}  \\ \hline
		$f_{\1_\catM}: \emptyset \rightarrow \DD_*$ & 
		\begin{tikzpicture} 
		\draw [(-),dotted] (0,0) -- (5,0);
		\node at (2.5,0) {$\ast$};
		\end{tikzpicture} \\ \hline
	\end{tabular}
	\label{table:ztd1op}
\end{table}

\begin{thm} \label{thm:1rs}
	To a $\Z_2$-tensor pair $(\catA,\catM)$ there is an associated $\ztd_1$-algebra $F_{(\catA,\catM)}$ in $\Rex$.
\end{thm}
\begin{proof}
	We will abbreviate $F_{(\catA,\catM)}$ by $F$.
	First we define the assignment of functors to embeddings.
	To do that we introduce some notation.
	Recall we can assign a tuple $(\ep_1,\dots,\ep_k) \in \{0,1\}^{\times k}$ and a permutation $\sigma \in S_k$ to operations via the Lemmas \ref{htype}, \ref{order1} and \ref{order2}.
	We denote $\Phi^0 := \id_\catA$ and $\Phi^1 := \Phi$, and $\tensor_0:= \fdvect \xrightarrow{\1} \catA, \tensor_1 := \id_\catA$ and $\tensor_n := \tensor \circ (\tensor_{n-1} \boxt \id_\catA): \catA^{\boxt n} \rightarrow \catA$. 
	
	We can then define the associated functors to operations in $\ztd_n( \DD^{\amalg k}, \DD)$ by
	\begin{align} \label{rass1}
	&\catA^{\boxt k} \xrightarrow{\Phi^{\ep_1} \boxt \dots \boxt \Phi^{\ep_k}} \catA^{\boxt k} \xrightarrow{\sigma} \catA^{\boxt k} \xrightarrow{\tensor_k} \catA,
	\end{align}
	and define the associated functors to operations in $\ztd_n( \DD^{\amalg k}, \DD_*)$ by
	\begin{align} \label{rass2} 
	&\catA^{\boxt k} \xrightarrow{\Phi^{\ep_1} \boxt \dots \boxt \phi^{\ep_k}} \catA^{\boxt k} \xrightarrow{\sigma} \catA^{\boxt k} \xrightarrow{\tensor_k} \catA \xrightarrow{\1_{\catM} \tensor -} \catM,
	\end{align}
	and define the associated functors to operations in $\ztd_n(\DD_* \amalg \DD^{\amalg k}, \DD_*)$ by
	\begin{align} \label{rass3} 
	&\catM \boxt \catA^{\boxt k} \xrightarrow{\id_\catM \boxt \Phi^{\ep_1} \boxt \dots \boxt \Phi^{\ep_k}} \catM \boxt A^{\boxt k} \xrightarrow{\id_\catM \boxt \sigma} \catM \tensor \catA^{\boxt k} \xrightarrow{\id_\catM \boxt \tensor_k} \catM \boxt \catA \xrightarrow{\act} \catM.
	\end{align}
	Note this assignment satisfies Definition-Proposition \ref{rexalgdef} $(i)$.
	
	Next we will assign the natural isomorphisms.
	To an isotopy between operations $g \circ (\sigma_f \circ (\amalg f_j))$ and $h$ we will assign a structural isomorphisms between the assigned functors $F(g) \circ (\sigma\circ (\boxt F(f_j))$ and $F(h)$.
	Combining Theorem \ref{moncoh} and Remark \ref{klincoh} we know there is at most one structural isomorphism  between two given functors.
	Hence to unambiguously specify the assignment it suffices to show that there exists a structural isomorphism.
	We will only discuss one particular case of an assignment, leaving the others to the reader since they are all analogous.
	Consider operations $g \circ (\sigma_F\circ (\amalg_j f_j))$ and $h$ with codomain $\DD$ that are isotopic.
	Note that then all the $f_j$ must also be operations with codomain $\DD$. 
	We consider their images under the maps of Lemma \ref{htype} and \ref{order1}:
	\begin{align*}
	&\ztd_1(\DD^{\amalg i},\DD) \rightarrow F_i[\DD/\Z_2] \rightarrow \{0,1\}^{\times i} \times S_i & &\\
	&h \mapsto (\epsilon_1,\dots,\epsilon_k, s) &  &\text{ here } i = k,\\
	&g \mapsto (\ep_1,\dots,\ep_k, \sigma) & &\text{ here } i = j,\\
	&f_j \mapsto (\ep_1^j,\dots,\ep_{k_j}^j,\sigma_j) & &\text{ here } i = k_j.
	\end{align*}
	As $g \circ (\sigma_f\circ(\amalg_j f_j)) \simeq h$ and $\ztd_n(\DD^k,\DD)$ contracts onto the discrete space $\{0,1\}^{\times k}\times S_k$ we find
	\begin{align*}
	&s = \sigma \circ \big( \sigma_f * (\sigma_1 \times \dots \times \sigma_j)\big) \text{ as permutations in } S_k,\\
	&\epsilon_{s(k_1 + \dots + k_{j-1} + n)} = \ep_{\sigma(j)} + \ep^{\sigma_f(j)}_{\sigma_j (n)} \; \mathrm{mod} \;2 \quad \quad \text{ for any } j, \; 1 \leq n \leq k_j.
	\end{align*}
	Hence $F(h)$ and $F(g)\circ (\sigma_f \circ (\boxt_j F(f_j)))$ multiply a tuple of objects $(X_1,\dots,X_k)$ in $\catA$ in the same order, and with the same distribution of $\Phi$'s (modulo 2).  
	Then we can define a structural isomorphism between $F(h)$ and $F(g) \circ (\boxt_j F(f_j))$ via repeated applications of $\alpha, \phi$, and $t$ (and possibly $\phi_0$, $\lambda$ and $\rho$ in case one or more of the $k_j=0$). 
	
	Note that our assignment of natural isomorphisms satisfies Definition-Proposition \ref{rexalgdef} $(ii)$. 
	It remains to check condition $(iii)$. Consider a diagram of isotopies and operations as in condition $(iii)$ together with a 2-isotopy. 
	The two sides of Equation \eqref{2isoeq} are  structural isomorphisms by construction of our assignments.
	Thus the two sides are parallel structural isomorphisms.
	Then they must agree by Theorem \ref{moncoh} combined with Remark \ref{klincoh}.
\end{proof}

\subsection{$\Z_2$-braided pairs versus $\ztd_2$-algebras}

\begin{defn}
	A \emph{$\Z_2$-braided tensor pair} is a $\Z_2$-tensor pair $(\catA,\catM)$ together with a braiding $\sigma$ on $\catA$, such that $\Phi$ is braided monoidal, and a natural isomorphism 
	\[\kappa: \tensor_{\catM} \Rightarrow \tensor_{\catM} \circ (\id_{\catM} \boxt \Phi)\] 
	such that the diagrams of Figures \ref{fig:bp1} and \ref{fig:bp2} commute for all $M\in \catM$, $X,Y \in \catA$.
\end{defn}
To study categorical $\ztd_2$-algebras we will make use of an ordering on $\R^2$ that subdivides the spaces of operations.
\begin{defn} \label{assignmentdef} Let $\DD$ and $\DD_*$ be of dimension $n >1$.
	\begin{enumerate}
		\item For $x = (x_1,\dots,x_n)$, $y = (y_1,\dots,y_n) \in \R^n$. We write $x<y$ iff $x_1 < x_2$ or there exists $i>1$ such that $x_1 = y_1,x_2 = y_2,\dots, x_i = y_i$ and $x_{i+1} < y_{i+1}$.
		\item To $x = (x_1,\dots,x_k) \in F_k [\DD/\Z_2]$ or $x = (x_1,\dots,x_k) \in F_k[\DD_*/\Z_2]$ we assign $(\ep_1,\dots,\ep_k)\in \{0,1\}^{\times k}$ as in \ref{tupledef1} respectively \ref{tupledef2}.
		\item For $x \in F_k [\DD/\Z_2]$ or $x\in F_k[\DD_*/\Z_2]$ Equation \eqref{sigmadef} defines a unique $\sigma \in S_k$.
	\end{enumerate}
\end{defn}
\begin{rmk}Please note that the assignments 
	\begin{align}
	F_k[\DD/\Z_2] \rightarrow \{0,1\}^{\times k} \times S_k, \label{assdef1}\\
	F_k[\DD_*/\Z_2] \rightarrow \{0,1\}^{\times k} \times S_k \label{assdef2}
	\end{align}
	of Definition \ref{assignmentdef} do not define continuous maps. 
\end{rmk}
\begin{lem} \label{lem:braidassign} Let $n=2$ and $k\geq 0$ and fix basepoints of the spaces $F_k[\DD/\Z_2]$ and $F_k[\DD_*/\Z_2]$. 
	\begin{enumerate} 
		\item To a path $\gamma$ in $F_k[\DD/\Z_2]$ there is a canonically associated braid $[\gamma] \in B_n$ such that $[\gamma_1 * \gamma_2] = [\gamma_1][\gamma_2]$ and $[\gamma^{-1}] = [\gamma]^{-1}$.
		\item  To a path $\gamma$ in $F_k[\DD_*/\Z_2]$ there is a canonically associated braid $[\gamma] \in \bcyl_n$ such that $[\gamma_1 * \gamma_2] = [\gamma_1][\gamma_2]$ and $[\gamma^{-1}] = [\gamma]^{-1}$.
	\end{enumerate}
\end{lem}
\begin{proof}
	We will prove part two, and leave part one to the reader as the proof is analogous. 
	Consider the non-continuous projection 
	\[F_k[\DD_*/Z_2] \xrightarrow{\ref{assdef2}} \{0,1\}^{\times k} \times S_k\]
	and observe that the set-theoretic fibers of $p$ define contractible subspaces $\{F^i\}_{i=0}^N$ of $F_k[\DD_*/\Z_2]$.
	Let $\{x_1,\dots,x_N: x_i \in F^i\}$ denote the $\Z_2^{\times k} \times S_k$ orbit of the basepoint of $F_k[\DD_*/\Z_2]$.
	Given a path $\gamma$ in $F_k[\DD_*/\Z_2]$ with endpoints $\gamma_0 \in F^{i_0}$ and $\gamma_1 \in F^{i_1}$ we choose a path $e_0$ lying in $F^{i_0}$ from $x_{i_0}$ to $\gamma_0$ and a path $e_1$ lying $F^{i_1}$ from $\gamma_1$ to $x_{i_1}$. 
	Recall that we have the quotient map
	\[ q: F_k[\DD_*/\Z_2] \rightarrow C_k[\DD_*/\Z_2],\] 
	we then define 
	\[ [\gamma] := [q(e_1 * \gamma * e_0)] \in \pi_1 \big( C_k[\DD_*/\Z_2],q(x_0) \big) \overset{\ref{kpi1}}{=}  \bcyl_k. \]
	Since the subspaces $F^i$ are contractible one easily verifies the assignment is independent of the choice of $e_0$ and $e_1$ and respects composition and inversion.
\end{proof}
\begin{rmk} \label{rmk:braidassign}
	Consequently, Lemma \ref{lem:braidassign} together with the homotopy equivalences of Proposition \ref{htype} allows us to associate braids to isotopies of operations in $\ztd_2$. Moreover, note that 2-isotopic isotopies are assigned the same braids.
\end{rmk}

\begin{prop} \label{prop:2rn}
	To a $\ztd_2$-algebra $F$ in $\Rex$ there is an associated $\Z_2$-braided tensor pair $(\catA_F,\catM_F)$.
\end{prop}
\begin{proof}
	We set $\catA_F := \catA$ and $\catM_F := \catM$.
	Consider the fixed operations $f_\tensor$, $f_\Phi$,  $f_\act$, in Figure \ref{table:ztd2op}, the identity operations $\id_\DD$ and $\id_{\DD_*}$ and the unique operations $f_{\1} \in \ztd_2(\emptyset,\DD)$ and $f_{\1_\catM} \in \ztd_2(\emptyset,\DD_*)$.
	These operations are by $F$ assigned functors that we denote respectively $\tensor$, $\Phi$, $\act$, $\id_\catA$,$\id_{\catM}$, $\1$, and $\1_{\catM}$. 
	Note that by Definition-Proposition \ref{rexalgdef} $(i)$ the functors $\id_\catA$ and $\id_\catM$ are indeed the identity functors on $\catA$ and $\catM$. 
	
	Next we need to show there exist natural isomorphisms $\alpha, \lambda,\rho,\sigma,\Phi_2,\Phi_0,t,r,a,\kappa$ that are part of the definition of a $\Z_2$-braided tensor pair.
	Our strategy is as follows. To define a structural isomorphism $\eta$, we choose an isotopy $\gamma$ such that the associated braid $[\gamma]$ is equal to the underlying braid $\beta_\eta$, and then define $\eta$ as the induced isomorphism of $F(\gamma)$.
	Let us treat the case of $\sigma$, leaving the others to the reader.
	Consider the `swap' permutation $\tau = (12) \in S_2$. 
	We choose the isotopy $i_\sigma$  from $f_\tensor$ to the composite operation $f_\tensor \circ (\tau \circ (\id_\DD \amalg \id_\DD))$, given by rotating the disks 180 degrees around each other. 
	By definition of $\ztd_2$-algebra the assignment $F(i_\sigma)$ is a natural isomorphism from $\tensor = F(f_\tensor)$ to $\tensor \circ (\tau \circ ( \id_\catA \boxt \id_\catA) = \tensor^{op}$. We then define $\sigma := F(i_\sigma)$.
	
	It remains to check that the natural isomorphisms satisfy the defining equations of a $\Z_2$-braided tensor pair.
	As in the proof of Theorem \ref{prop:1rn} one easily deduces that the equations hold from the existence of various 2-isotopies. 
	For example, consider axiom $BP1$. The isotopy of coloured braids drawn there can easily be `fattened' to give a 2-isotopy in $\ztd_2(\DD^{\amalg 2} \amalg \DD_*,\DD_*)$. The existence of this 2-isotopy implies $a, \Phi_2,\kappa$ and $\sigma$ satisfy axiom BP1.
	We leave the details to the reader.
\end{proof}

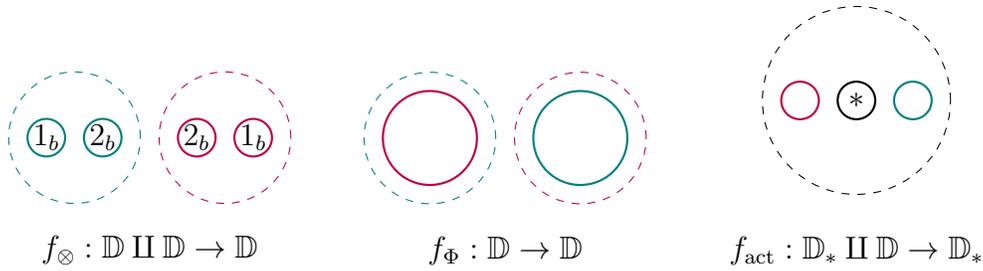
\begin{figure}[h]
	\begin{tikzpicture}[scale=.5]
	\draw [(-),teal,dashed] (1.75,0) circle (1.75);
	\draw [(-),purple,dashed] (5.75,0) circle (1.75);
	\draw [(-),teal,thick] (1,0) circle (.5);
	\draw [(-),teal,thick] (2.5,0) circle (.5);
	\draw [(-),purple,thick] (5,0) circle (.5);
	\draw [(-),purple,thick] (6.5,0) circle (.5);
	\node at (1,0) {$1_b$};
	\node at (2.5,0) {$2_b$};
	\node at (5,0) {$2_b$};
	\node at (6.5,0) {$1_b$};
	\node at (3.75,-3) {$f_\tensor: \DD \amalg \DD \rightarrow \DD$};
	\end{tikzpicture} 
	\quad \quad
	\begin{tikzpicture}[scale=.5]
	\draw [(-),teal,dashed] (1.75,0) circle (1.75);
	\draw [(-),purple,dashed] (5.75,0) circle (1.75);
	\draw [(-),purple,thick] (1.75,0) circle (1.25);
	\draw [(-),teal,thick] (5.75,0) circle (1.25);
	\node at (3.75,-3) {$f_\Phi: \DD \rightarrow \DD$ };
	\end{tikzpicture}  
	\quad \quad
	\begin{tikzpicture}[scale=.5]
	\draw [(-),dashed] (2.5,0) circle (2.5);
	\draw [(-),purple,thick] (1,0) circle (.5);
	\draw [(-),black,thick] (2.5,0) circle (.5);
	\draw [(-),teal,thick] (4,0) circle (.5);
	\node at (2.5,0) {$\ast$};
	\node at (2.5,-4	) {$f_{\act}: \DD_* \amalg \DD \rightarrow \DD_*$};
	\end{tikzpicture} 
	\caption{In this figure operations in the $\ztd_2$-operad are fixed by their visualisations.}
	\label{table:ztd2op}
\end{figure}

\begin{thm} \label{thm:2rs}
	To a $\Z_2$-braided tensor pair $(\catA,\catM)$ there is an associated $\ztd_2$-algebra $F_{(\catA,\catM)}$ in $\Rex$.
\end{thm}

\begin{proof}
	We will abbreviate $F_{(\catA,\catM)}$ by $F$.
	Via Definition \ref{assignmentdef} we can assign a permutation $\sigma \in S_k$ and a tuple $(\ep_1,\dots,\ep_k) \in \{0,1\}^{\times k}$ to an operation $f$.
	We then assign functors $F(f)$ to operations via the formulae in Equations \ref{rass1} - \ref{rass3}. 
	
	Next we assign natural isomorphisms to isotopies. 
	To an isotopy $\gamma$ between operations $g \circ (\sigma_f \circ (\amalg f_j))$ and $h$ we will assign a structural isomorphisms between the assigned functors $F(g) \circ (\sigma \circ (\boxt F(f_j))$ and $F(h)$, as we did in the proof of Theorem \ref{thm:1rs}. 
	Via Lemma \ref{lem:braidassign} and Remark \ref{rmk:braidassign} we have an associated braid $[\gamma]$ to $\gamma$.
	Recall that by Theorem \ref{braidcoh} there is at most one structural isomorphism  between $F(g) \circ (\sigma \circ (\boxt F(f_j))$ and $F(h)$ that has underlying braid $[\gamma]$.
	Hence to specify the assignment it suffices to show that there exists such a structural isomorphism. 
	Existence is an easy case-by-case analysis similar to the one in Theorem \ref{thm:1rs} that we leave to the reader. 
	
	Note that our assignment of natural isomorphisms satisfies Definition-Proposition \ref{rexalgdef} $(ii)$. 
	It remains to check condition $(iii)$. Consider a diagram of isotopies and operations as in condition $(iii)$ together with a 2-isotopy. 
	The two sides of Equation \eqref{2isoeq} are  structural isomorphisms by construction of our assignments and by definition of structural isomorphism.
	Moreover, the assigned braids of both sides are equal by Lemma \ref{lem:braidassign} and Remark \ref{rmk:braidassign}. 
	Hence, Equation \eqref{2isoeq} holds by Theorem \ref{braidcoh} combined with Remark \ref{klincoh}.
\end{proof}

\subsection{$\Z_2$-symmetric pairs versus $\ztd_n$-algebras}

\begin{defn}
	A $\Z_2$-braided tensor pair $(\catA,\catM)$ is called a \emph{$\Z_2$-symmetric tensor pair} if
	\begin{align*}
	&\sigma_{X,Y} \circ \sigma_{Y,X} = \id_{X\tensor Y},\\
	&\kappa_{M,\Phi(X)} \circ \kappa_{M,X} = \id_{M \tensor X},
	\end{align*}
	for all $M\in \catM, X, Y \in \catA$ 
\end{defn}

\begin{prop} \label{prop:nrn} Let $n\geq 3$.
	To a $\ztd_n$-algebra $F$ in $\Rex$ there is an associated $\Z_2$-symmetric tensor pair $(\catA_F,\catM_F)$.
\end{prop}
\begin{proof}
	As the proof is very similar to Proposition \ref{prop:2rn} we will be brief.
	The $n$-dimensional analogues of the operations, isotopies and 2-isotopies in the proof of Proposition  \ref{prop:2rn} define the structure of a $\Z_2$-braided pair on $\catA_F := \catA$ and $\catM_F:= \catM$. 
	Denote the isotopies $i_\sigma$ and $i_\kappa$ that are mapped by $F$ to the natural isomorphisms $\sigma$ and $\kappa$. 
	For $n\geq 3$ there now exist 2-isotopies 
	\[ i_\sigma \circ i_\sigma \simeq \id_{f_\tensor}, \quad \quad i_\kappa \circ i_\kappa \simeq \id_{f_{\act}}\]
	in $\ztd_n$ enforcing the equations $\sigma^2 = \id$ and $\kappa^2 = \id$. 
\end{proof}

\begin{thm} \label{thm:nrs}  Let $n\geq 3$.
	To a $\Z_2$-symmetric tensor pair $(\catA,\catM)$ there is an associated $\ztd_n$-algebra $F_{(\catA,\catM)}$ in $\Rex$.
\end{thm}
\begin{proof}
	As the proof is very similar to Theorems \ref{thm:1rs} and \ref{thm:2rs} we will be brief.
	Via Definition \ref{assignmentdef} we can assign a permutation $\sigma \in S_k$ and a tuple $(\ep_1,\dots,\ep_k) \in \{0,1\}^{\times k}$ to an operation $f$.
	We then assign functors $F(f)$ to operations via the formulae in Equations \eqref{rass1} - \eqref{rass3}. 
	For an isotopy $\gamma$ between operations $g \circ (\sigma_f \circ (\amalg f_j))$ and $h$ it is not hard to show that there exists a structural isomorphism between the assigned functors $F(g) \circ (\sigma \circ (\boxt F(f_j))$ and $F(h)$.
	Theorem \ref{symcoh} combined with Remark \ref{klincoh} then imply this assignment is unique, and satisfies Definition-Proposition \ref{rexalgdef}.
\end{proof}

\section{General algebras} \label{sec:genalg}
To define algebras in the $(2,1)$-category $\Rex$ over the topological operad $\ztd_n$ we employ the language of $(\infty,1)$-categories.
Our model of $(\infty,1)$-categories is given by \emph{quasi-categories}.\footnote{Simplicial sets satisfying the weak Kan condition of Boardman-Vogt.} 
Our references are \cite{lhtt} and \cite{lha}.

\subsection{Higher categories and higher operads}
Before recalling the definition of $\infty$-operads we first adopt a different perspective on topogical operads.

\begin{nota} Let $\Fin_*$ denoted the skeletal category of finite pointed sets.\footnote{$\Fin_*$ has objects $[k]$ for $k\in \Z_{\geq 0}$ and a morphism $f:[k] \rightarrow [m]$ is map $f:\{0,\dots,k\} \rightarrow \{0,\dots,m\}$ that sends 0 tot 0.}
\end{nota}
\begin{defn} \cite[2.1.1.7]{lha}
	Let $\catO$ be a topological coloured operad with operation spaces $\catO(c_1 \tensor \dots \tensor c_k,c)$ for colours $c,c_i$. The topological category $\catO^{\tensor}$ has as objects finite sequences of colours of $\catO$ and a morphism $\{c_1,\dots,c_k\} \rightarrow \{d_1,\dots,d_m\}$ is a collection 
	\[ f : [k] \rightarrow [m], \quad \quad f_i \in \catO \big( \bigotimes_{j: f(j)=i}c_j,d_i \big) \text{ for } 1 \leq i \leq k,\]
	where $f$ is a morphism in $\Fin_*$.\footnote{The composition of morphisms in $\catO^{\tensor}$ is induced from the operadic composition in $\catO$. The morphism spaces in $\catO^{\tensor}$ are canonically unions of products of operation spaces of $\catO$.}
\end{defn}
\begin{rmk}
	The category $\catO^{\tensor}$ has a canonical forgetful functor $\catO^{\tensor} \rightarrow \Fin_*$ which one can use to reconstruct the operad $\catO$ out of $\catO^{\tensor}$. 
\end{rmk}

One obtains the associated $\infty$-operad by taking the \emph{(homotopy) coherent nerve} $\nerve$ of $\catO^{\tensor}$.\footnote{See e.g. \cite[\S1.1.5]{lhtt} for the definition of the coherent nerve (also known as the \emph{simplicial nerve}).} 
To apply the homotopy coherent nerve to a topological category one needs to change enrichment to simplicial sets via the singular chains functor.
\begin{conv}
	We will change enrichment from topological to simplicial without comment.
\end{conv}
It is worthwhile to remind ourselves of some of the low-dimensional simplices of $\nerve \catC$. 
\begin{ex} \label{snerve}
	Let $\catC$  be a simplicially enriched category, and denote $\circ$ (resp. $*$) the horizontal composition in $\catC$ of vertices (resp. edges). The coherent nerve $\nerve \catC$ is a simplicial set with vertices being the objects of $\catC$, 1-simplices being the morphisms of $\catC$, 2-simplices being diagrams of the form 
	\adjustbox{scale=.75}{
		\begin{tikzcd}[row sep=small,column sep=large]
			& \bullet \arrow[dd, "g"]\\
			\bullet \arrow[ur,"f"] \arrow[dr, "h"{name=U, below}]{} & \\
			& \bullet  \arrow[Rightarrow, from=U,uu, start anchor={[xshift=1ex,yshift=1ex]},
			end anchor={[xshift=-1ex,yshift=-1ex]}]
		\end{tikzcd}  
	}
	and a 3-simplex consists of 
	\begin{align*}
	&\text{four objects} & & X_0,\dots,X_3 \in \catC  ,\\
	&\text{six morphisms/vertices} & &\{\; f_{ij}: X_i \rightarrow X_j \in \catC(X_i,X_j)_0 \; \}_{0\leq i < j \leq3}\\
	&\text{four 2-morphisms/edges} & & \{\; \alpha_{ijk}: f_{jk} \circ f_{ij} \simeq f_{il} \in \catC(X_i,X_k)_1 \; \}_{0\leq i < j < k \leq 3}\\
	&\text{an extra 2-morphism/edge} & &\alpha_{0123}: f_{03} \simeq f_{23} \circ f_{12} \circ f_{01} \in \catC(X_0,X_3)_1  \\
	&\text{and two 3-morphisms/2-simplices} &  &\text{in } \catC(X_0,X_3)_2 \text{ with boundaries}
	\end{align*}
	
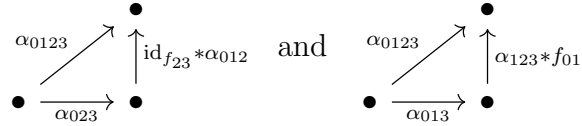
\begin{center}
		\begin{tikzcd} 
			& \bullet \\ 
			\bullet \arrow[ru, "\alpha_{0123}"] \arrow[r,swap,"\alpha_{023}"] & \bullet \arrow[u,swap,"\id_{f_{23}} * \alpha_{012}"]  
		\end{tikzcd} 
		\text{ and } 
		\begin{tikzcd} 
			& \bullet \\ 
			\bullet \arrow[ru, "\alpha_{0123}"] \arrow[r,swap,"\alpha_{013}"] & \bullet \arrow[u,swap,"\alpha_{123}* f_{01}"]  
		\end{tikzcd} 
	\end{center}
\end{ex} 

\begin{nota}
	We denote $\nerve \catO^{\tensor}$ as $N^\tensor(\catO)$ and call it \emph{the operadic nerve} of $\catO$ \cite[2.1.1.23]{lha}.
\end{nota}
\begin{rmk}
	$\nerve^{\tensor} (\catO)$ is an $\infty$-category with a canonical functor $q: \nerve^{\tensor} (\catO) \rightarrow \nerve \Fin_*$.\footnote{$\Fin_*$ is a discrete topological category so that the coherent nerve of $\Fin_*$ agrees with the usual nerve.}
\end{rmk}

To describe the special properties of the map $q$ we recall some further definitions. 
A map $f: [n] \rightarrow [m]$ is called \emph{inert} if for all $0< i \leq m$ there is a unique $j \in [n]$ such that $f(j)=i$. 
\begin{ex}
	For given $k$ and $1\leq i \leq n$ we denote by $\rho^i: [k] \rightarrow [1]$ the inert map in $\Fin_*$ that sends $i$ to $1$ and the rest to $0$.
\end{ex}
A morphism $f: X\rightarrow Y$ in $\catO^{\tensor}$ is \emph{$q$-coCartesian} if and only for every $Z\in \catO^{\tensor}$ the diagram
\begin{center}
	\begin{tikzcd}
		\catO^{\tensor} (Y,Z) \arrow[d,"q"] \arrow[r,"f^*"] & \catO^{\tensor}(X,Z) \arrow[d,"q"]\\
		\Fin_* (q(Y),q(Z)) \arrow[r,"q(f)^*"] & \Fin_* (q(X),q(Z))
	\end{tikzcd}
\end{center} 
is a homotopy pullback.\footnote{This is equivalent to the usual Definition of $q$-coCartesian of \cite{lhtt} by Proposition \cite[2.4.1.10]{lhtt}.}

\begin{ex} \label{ex:cocart}
	A morphism $(f,f_i)$ in $\nerve^{\tensor}(\ztd_n)$ lying over an inert morphism $f$ is $q$-coCartesian if and only if each $f_i$ is isotopic to either $\id_\DD$ or $\id_{\DD_*}$.
\end{ex}

\begin{nota}
	For $f: [m] \rightarrow [n]$ denote $\Map_{\catO^{\tensor}}^{f}(X,Y)$ the union of connected components of $\Map_{\catO^{\tensor}}(X,Y)$ lying over $f$. 
\end{nota}

\begin{nota}
	We denote with $\catO^{\tensor}_{[n]}$ the subcategory of $\catO^{\tensor}$ consisting of lists of length $n$ and morphisms lying over $\id_{[n]}$. 
\end{nota}

\begin{prop} \cite[2.1.1.27]{lha} \label{infoperad}
	For any topological operad $\catO$ the functor $q: \nerve^{\tensor} (\catO) \rightarrow \nerve \Fin_*$ defines an $\infty$-operad i.e.
	\begin{enumerate}
		\item For any $f: [m] \rightarrow [k]$ and $X \in \catO^{\tensor}_{[m]}$ there is an object $Y \in \catO^{\tensor}_{[k]}$ and $q$-coCartesian morphism $f_!: X \rightarrow Y$ lying over $f$.
		\item Let $X \in \catO^{\tensor}_{[m]}$, $Y \in \catO^{\tensor}_{[k]}$ and choose $q$-coCartesian morphism $\rho^i_!:Y \rightarrow Y_i$ lying over the $\rho_i: [k] \rightarrow [1]$. Then the induced map 
		\[\Map_{\catO^{\tensor}}^f(X,Y) \rightarrow \prod_{1\leq i \leq k} \Map_{\catO^{\tensor}}^{\rho^i \circ f}(X,Y_i) \]
		is a homotopy equivalence. 
		\item For any $X_1,\dots,X_k \in \catO^{\tensor}$ there is a $X \in \catO^{\tensor}_{[k]}$ and $q$-coCartesian morphisms $\rho^i_!:X \rightarrow X_i$ lying over $\rho^i$.
	\end{enumerate}
\end{prop}

\begin{defn}  \label{genalgdef} 
	Let $p:\catO^{\tensor} \rightarrow \nerve \Fin_*$ be an $\infty$-operad. 
	A \emph{$\ztd_n$-algebra in $\catO^{\tensor}$} is a map of $\infty$-operads from $\ztd_n$ to $\catO^{\tensor}$ i.e. a functor $F: \nerve^{\tensor}(\ztd_n)  \rightarrow \catO^{\tensor}$ such that $p \circ F = q$ and $F$ sends $q$-coCartesian morphisms lying over inert morphisms to $p$-coCartesian morphisms.
\end{defn}

\begin{rmk}
	A symmetric monoidal $\infty$-category $\catC^{\tensor} \rightarrow \nerve \Fin_*$ is a special kind of $\infty$-operad \cite[2.0.0.7]{lha}, so that Definition \ref{genalgdef} provides a notion of $\ztd_n$-algebra in $\catC^{\tensor}$.
\end{rmk}

\begin{rmk}
	If $\catO^{\tensor} = N^{\tensor}(\catO)$ for some topological operad $\catO$ then a $\ztd_n$-algebra in $\catO^{\tensor}$ corresponds to a homotopy coherent algebra in $\catO$ over $\ztd_n$.  
\end{rmk}

\begin{defn} Let $\catO^{\tensor}$ be an $\infty$-operad.
	The $\infty$-category of $\ztd_n$-algebras in $\catO$, denoted $\mathrm{Alg}_{\ztd_n}(\catO)$, is the full subcategory of
	\[ \mathrm{Fun}(\nerve^{\tensor}(\ztd_n),\catO^{\tensor})\]
	with objects the $\infty$-operad maps.\footnote{Recall that $\mathrm{Fun}(\nerve^{\tensor}(\ztd_n),\catO^{\tensor})$ is just the inner hom of simplicial sets.}
\end{defn}

\subsection{Categorical algebras revisited}
Recall that $\Rex$ has natural isomorphisms as 2-morphisms. 
Hence $\Rex$ is a $(2,1)$-category i.e. a category enriched over groupoids.
We can view $\Rex$ as a simplicially enriched category by base-changing enrichment along the nerve functor $\nerve: \textrm{Grpds} \rightarrow \textrm{sSet}$.
\begin{defn}
	The $\infty$-operad $p: \nerve^{\tensor}(\Rex) \rightarrow \nerve \Fin_*$ is the operadic nerve of the simplicial operad whose colours are the objects of $\Rex$ and which has operations spaces
	\[ \Rex( X_1 \boxtimes \dots \boxtimes X_k,X ) \quad \quad \text{for} \quad X_1, \dots, X_k,X \in \Rex.\footnote{The Deligne-Kelly tensor product can be used to define an operadic structure as it defines a symmetric monoidal structure on the $(2,1)$-category $\Rex$ \cite{kelly89}.}\]	
\end{defn}

Since $\Rex$ is a (2,1)-category the operadic nerve allows a simple description as it is 3-coskeletal. Let us briefly recall this notion.
Denote $\Delta_n \subset \Delta$ the full subcategory generated by $[0],[1],\dots,[n]$. For a simplicial set the restriction to $\Delta_n$ defines a truncation functor 
\[\tr_n: [\Delta^{op},\mathrm{Set}] \rightarrow [\Delta_n^{op},\mathrm{Set}].\]
A simplicial set $X$ is called \emph{n-coskeletal} if for any simplicial set $Y$ the natural map
\[ \Hom_{[\Delta^{op},\mathrm{Set}]} (Y,X) \rightarrow \Hom_{[\Delta^{op}_n,\mathrm{Set}]} (\tr_n Y, \tr_n X)\]
is a bijection.\footnote{In fact, there is an adjunction $\mathrm{cosk}_n: [\Delta_n^{op},\mathrm{Set}] \leftrightharpoons [\Delta^{op},\mathrm{Set}]: \tr_n$ and $n$-coskeletal sets are exactly those for which $X \cong \mathrm{cosk}_n\tr_n X$.}

\begin{defprop} \cite[Proposition 4.1]{bfb04} \label{geometricnerve}
	Let $\catC$ be a 2-category. The coherent nerve $\nerve \catC$ is the 3-coskeletal simplicial set with vertices the objects of $\catC$, 1-simplices the morphisms of $\catC$,  2-simplices the diagrams of the form 
	\adjustbox{scale=.75}{
		\begin{tikzcd}[row sep=small,column sep=large]
			& \bullet \arrow[dd, "g"]\\
			\bullet \arrow[ur,"f"] \arrow[dr, "h"{name=U, below}]{} & \\
			& \bullet  \arrow[Rightarrow, from=U,uu, start anchor={[xshift=1ex,yshift=1ex]},
			end anchor={[xshift=-1ex,yshift=-1ex]}]
		\end{tikzcd}  
	}
	and 3-simplices are `commutative tetrahedra' of the form 
	
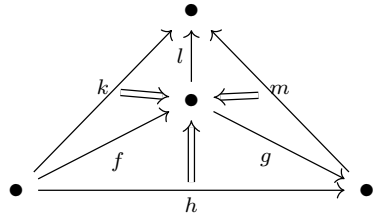
\begin{center}
		\begin{tikzcd}[column sep = 50pt, row sep= 20pt]
			& \bullet & \\
			& \bullet \arrow[u,"l"] \arrow[rd,swap,"g"] & \\
			\bullet \arrow[uur, "k"{name=X,above}] \arrow[ur, swap,"f"] \arrow[rr,swap,"h"{name=Y,below}] \arrow[Rightarrow,ur, from=X, start anchor={[xshift = .3ex, yshift=-.3ex]}] \arrow[Rightarrow,ur, from=Y, start anchor={[yshift = .6ex]}]& & \bullet \arrow[uul,"m"{name=Z, above}] \arrow[Rightarrow,ul, from=Z, start anchor={[xshift = -.3ex, yshift=-.3ex]}]
		\end{tikzcd}
		$\begin{array}{ l l }
		\alpha: k \Rightarrow mh & \text{(bottom face)}\\
		\beta: h \Rightarrow gf  & \text{(lower face)} \\
		\gamma: k \Rightarrow lf & \text{(left face)}\\
		\delta: l \Rightarrow mg & \text{(right face)}
		\end{array}$
	\end{center}
	where commutativity means that the diagram
	\begin{tikzcd}[scale=.7]
		k \arrow[r,"\alpha"] \arrow[d, "\gamma"] & mh \arrow[d, "\id_m * \beta"] \\
		lf \arrow[r, "\delta * \id_f"] & mgf
	\end{tikzcd}
	commutes.\footnote{This simplicial set is sometimes called the \emph{geometric nerve} of the 2-category and is due to R. Street.}
\end{defprop} 

We will now prove Definition-Proposition \ref{rexalgdef} by showing $\ztd_n$-algebras in $N^{\tensor}(\Rex)$ are exactly determined by the list of data specified there.
\vspace{.5em} \\
\textbf{Proof of Definition-Proposition \ref{rexalgdef}.}
Let $F: \nerve^{\tensor} (\ztd_n ) \rightarrow \nerve^{\tensor} (\Rex)$ be a map of $\infty$-operads. We denote $\catA := F(\DD)$ and $\catM := F(\DD_*)$.
Without loss of generality we may assume $F (X_1,\dots X_k) = (F(X_1),\dots, F(X_k))$. 
By Property (2) in Proposition \ref{infoperad} the functor $F$ is completely determined by its value on $n$-simplices in which the edge between vertices $n-1$ and $n$ is a map lying over some $f:[k] \rightarrow [1]$.
Such a map is exactly an operation in $\ztd_n$. 
The value of $F$ on such 1-, 2-, and 3-simplices are on the nose the assignments of Definition-Proposition \ref{rexalgdef}. The further restrictions on identity operations and trivial isotopies in Definition \ref{rexalgdef} correspond to $F$ commuting with the simplicial maps between 1-,2- and 3-simplices. 
Now recall that the nerve of any 2-category is 3-coskeletal i.e.
\begin{align}
\Fun (\nerve^{\tensor} (\ztd_n), \nerve^{\tensor} (\Rex) \cong \Hom_{\mathrm{sSet}_{\leq 3}} (\tr_3\nerve^{\tensor} (\ztd_n), \tr_3 \nerve^{\tensor} (\Rex). \label{cosk}
\end{align}
Hence by Equation \eqref{cosk} the functor $F$ is completely determined by an assignment as in Definition \ref{rexalgdef}. 
Conversely, we observe that starting from such an assignment the associated functor $F$ is a map of $\infty$-operads. Indeed, this is clear from Example \ref{ex:cocart} since these $q$-coCartesian morphisms are assigned either $\id_\catA$ or $\id_\catM$ which are of course $p$-coCartesian. \\
\qed

We will conclude by showing that the assignments in section \ref{sec:rexalg} are inverse equivalences.
\begin{thm}\label{thm:equivalences}
	\begin{enumerate}
		\item The assignments 
		\begin{align*}
		&F \mapsto (\catA_F,\catM_F) \text{ of Proposition \ref{prop:1rn},}\\ 
		&(\catA,\catM) \mapsto F_{(\catA,\catM)} \text{ of Theorem \ref{thm:1rs}}
		\end{align*}
		are inverse equivalences. 
		\item The assignments 
		\begin{align*}
		&F \mapsto (\catA_F,\catM_F) \text{ of Proposition \ref{prop:2rn},}\\ 
		&(\catA,\catM) \mapsto F_{(\catA,\catM)} \text{ of Theorem \ref{thm:2rs}}
		\end{align*}
		are inverse equivalences. 
		\item The assignments 
		\begin{align*}
		&F \mapsto (\catA_F,\catM_F) \text{ of Proposition \ref{prop:nrn},}\\ 
		&(\catA,\catM) \mapsto F_{(\catA,\catM)} \text{ of Theorem \ref{thm:nrs}}
		\end{align*}
		are inverse equivalences. 
	\end{enumerate}
\end{thm}
\begin{proof}
	The three proofs are identical, so we will prove all three statements at once.
	First observe that the composite assignment 
	\[(\catA,\catM) \mapsto F_{(\catA,\catM)} \mapsto (\catA_{F_{(\catA,\catM)}},\catM_{F_{(\catA,\catM)}}) \]
	by construction gives back $(\catA,\catM)$ on the nose. 
	Conversely, let us show that a $\ztd_n$-algebra $F$ is isomorphic to the $\ztd_n$-algebra $F_{(\catA_F,\catM_F)}$ as objects in $\mathrm{Alg}_{\ztd_n}(\Rex)$. We will abbreviate $F_{(\catA_F,\catM_F)}$ by $G$.  
	Since $\nerve^{\tensor}(\Rex)$ is 3-coskeletal, we find that a morphism $\eta: F \rightarrow G$ is completely determined by specifying:
	\begin{itemize}
		\item for each 1-simplex $f: X \rightarrow Y$ in $N^{\tensor}(\ztd_n)$ a 1-simplex $\eta(f): F(X) \rightarrow G(Y)$ in $N^{\tensor}(\Rex)$,
		\item for each 2-simplex $\alpha: h \simeq f\circ g$ in $N^{\tensor}(\ztd_n)$ 2-simplices 
		\[\eta^F(\alpha): \eta(h) \Rightarrow \eta(g) \circ F(f), \quad \eta^G(\alpha): \eta(h) \Rightarrow G(g) \circ \eta(f)\]
		in $N^{\tensor}(\Rex)$,
		\item so that for every 3-simplex in $N^{\tensor}(\ztd_n)$, in the notation of Definition-Proposition \ref{geometricnerve}, Equations \eqref{meq1} to \eqref{meq3} hold.
	\end{itemize}
	\begin{align}
	(F(\delta) * \id) \circ \eta^F (\gamma) = (\id * \eta^F(\beta))\circ \eta^F(\alpha), \label{meq1}\\
	(\eta^F(\delta) * \id) \circ \eta^G(\gamma) = (\id * \eta^G(\beta)) \circ \eta^F(\alpha), \label{meq2}\\
	(G (\delta) * \id) \circ \eta^G(\gamma) = (\id * \eta^G(\beta)) \circ \eta^G(\alpha). \label{meq3}
	\end{align}
	By Property (2) in Proposition \ref{infoperad} it suffices to define $\eta$ on operations and $\eta^F,\eta^G$ on isotopies between operations. 
	We set $\eta(f) := F(f)$ for all operations, and set $\eta^F(\alpha) := F(\alpha)$ for isotopies. 
	To define $\eta^G(\id_f)$ we must construct a natural isomorphism $F(f) \Rightarrow G(f)$.
	The operation $f$ lies in the contractible fiber of one of the projections 
	\begin{align*}
	&\ztd_n(\DD^{\amalg k} \amalg \DD_*^{\amalg m},\DD) \xrightarrow{\ref{htype}} F_k[\DD/\Z_2] \xrightarrow{\ref{assdef1}} \{0,1\}^{\times k} \times S_k,\\
	&\ztd_n(\DD^{\amalg k} \amalg \DD_*^{\amalg m},\DD_*) \xrightarrow{\ref{htype}} F_k[\DD_*/\Z_2] \xrightarrow{\ref{assdef2}} \{0,1\}^{\times k} \times S_k,\end{align*}
	and note that there exists an operation $f_0$ in the fiber that is a horizontal/vertical composite of the fixed operations $f_\tensor$, $f_\act$, $f_\Phi$, $f_{\1}$, $f_{\1_\catM}$ of Proposition \ref{prop:1rn}  (or \ref{prop:2rn} or \ref{prop:nrn}). Let us choose a path in the fiber from $f$ to $f_0$. This defines an isotopy $\gamma_f: f \simeq f_0$ which we view as a 1-simplex in $\nerve^{\tensor}(\ztd_n)$.
	We define $\eta^G (\id_f)$ to be the composite natural isomorphism
	\[ F (f) \xrightarrow{F(\gamma_f)} F(f_0) \xrightarrow{\ref{pseudofun}.5} G(f_0) = G(f).\]
	More generally, we define $\eta^G$ for an isotopy $\alpha: h \simeq g\circ f$ as the horizontal composition 
	\[\eta^G(\alpha) := (\eta^G(\id_g)*\id_{F(f)} )\circ F(\alpha).\]
	It remains to show $\eta$ defines an isomorphism. 
	By writing out the definitions of $G$ and $\eta$ in terms of $F$ one shows Equations \eqref{meq1} to \eqref{meq3} hold. For example, Equation \eqref{meq1} follows immediately from functoriality of $F$. For the details see \cite{thesis}.
	To show that $\eta$ is an isomorphism it suffices to show its components $\eta(\id_X)$ for all $X \in \nerve^{\tensor}(\ztd_n)$ are equivalences \cite[Theorem C]{joyalnotes}. 
	This clearly holds as $\eta (\id_X) = \id_{F(X)}$.
\end{proof}


\begin{thebibliography}{99}
	
	\bibitem[AF15]{af15}
	D.~Ayala and J.~Francis.
	\newblock Factorization homology of topological manifolds.
	\newblock {\em J. Topol.}, 8(4):1045--1084, 2015.
	
	\bibitem[AFT17]{aft17}
	D.~Ayala, J.~Francis, and H.~L. Tanaka.
	\newblock Factorization homology of stratified spaces.
	\newblock {\em Selecta Math. (N.S.)}, 23(1):293--362, 2017.
	
	\bibitem[And10]{andrade10}
	R.~Andrade.
	\newblock {\em From manifolds to invariants of {E}n-algebras}.
	\newblock ProQuest LLC, Ann Arbor, MI, 2010.
	\newblock Thesis (Ph.D.)--Massachusetts Institute of Technology.
	
	\bibitem[Ara62]{araki62}
	S.~Araki.
	\newblock On root systems and an infinitesimal classification of irreducible
	symmetric spaces.
	\newblock {\em J. Math. Osaka City Univ.}, 13(1), 1962.
	
	\bibitem[AS96]{alekseev96}
	A.~Y. Alekseev and V.~Schomerus.
	\newblock Representation theory of {C}hern-{S}imons observables.
	\newblock {\em Duke Math. J.}, 85(2):447--510, 1996.
	
	\bibitem[AW95]{askeywilson}
	R.~Askey and J.~Wilson.
	\newblock {\em Some basic hypergeometric orthogonal polynomials that generalize
		{J}acobi polynomials}, volume~54 of {\em Mem. Amer. Math. Soc.}
	\newblock AMS, 1995.
	
	\bibitem[BD95]{bd95}
	J.~C. Baez and J.~Dolan.
	\newblock Higher-dimensional algebra and topological quantum field theory.
	\newblock {\em J. Math. Phys.}, 36(11):6073--6105, 1995.
	
	\bibitem[BD04]{bd04}
	A.~Beilinson and V.~Drinfeld.
	\newblock {\em Chiral algebras}, volume~51 of {\em {AMS} Coll. Publ.}
	\newblock {AMS}, Providence, RI, 2004.
	
	\bibitem[BFB05]{bfb04}
	M.~Bullejos, E.~Faro, and V.~Blanco.
	\newblock A full and faithful nerve for 2-categories.
	\newblock {\em Appl. Categ. Structures}, 13(3):223--233, 2005.
	
	\bibitem[BK]{balakolb16}
	M.~Balagovic and S.~Kolb.
	\newblock Universal {K}-matrix for quantum symmetric pairs.
	\newblock http://arxiv.org/abs/1507.06276.
	
	\bibitem[Bri71]{brieskorn}
	E.~Brieskorn.
	\newblock Die {F}undamentalgruppe des {R}aumes der regul\"aren {O}rbits einer
	endlichen komplexen {S}piegelungsgruppe.
	\newblock {\em Invent. Math.}, 12:57--61, 1971.
	
	\bibitem[Bro13]{brochier13}
	A.~Brochier.
	\newblock Cyclotomic associators and finite type invariants for tangles in the
	solid torus.
	\newblock {\em Algebr. Geom. Topol.}, 13(6):3365--3409, 2013.
	
	\bibitem[BSWW16]{bsww16}
	H.~{Bao}, P.~{Shan}, W.~{Wang}, and B.~{Webster}.
	\newblock {Categorification of quantum symmetric pairs I}.
	\newblock https://arxiv.org/abs/1605.03780, 2016.
	
	\bibitem[BW16]{bw16}
	H.~{Bao} and W.~{Wang}.
	\newblock {Canonical bases arising from quantum symmetric pairs}.
	\newblock https://arxiv.org/abs/1610.09271, 2016.
	
	\bibitem[BWW16]{bww16}
	H.~{Bao}, W.~{Wang}, and H.~{Watanabe}.
	\newblock {Multiparameter quantum Schur duality of type B}.
	\newblock https://arxiv.org/abs/1609.01766, 2016.
	
	\bibitem[BZBJ1]{bzbj}
	D.~Ben-Zvi, A.~Brochier, and D.~Jordan.
	\newblock Integrating quantum groups over surfaces.
	\newblock https://arxiv.org/abs/1501.04652.
	
	\bibitem[BZBJ2]{bzbj2}
	D.~Ben-Zvi, A.~Brochier, and D.~Jordan.
	\newblock Quantum character varieties and braided module categories.
	\newblock https://arxiv.org/abs/1606.04769.
	
	\bibitem[CG17]{cg17}
	K.~Costello and O.~Gwilliam.
	\newblock {\em Factorization algebras in quantum field theory. {V}ol. 1},
	volume~31 of {\em New Mathematical Monographs}.
	\newblock Cambridge University Press, Cambridge, 2017.
	
	\bibitem[Che84]{cherednik84}
	I.~V. Cherednik.
	\newblock Factorizing particles on a half line, and root systems.
	\newblock {\em Teoret. Mat. Fiz.}, 61(1):35--44, 1984.
	
	\bibitem[Che05]{cherednik05}
	I.~Cherednik.
	\newblock {\em Double affine {H}ecke algebras}, volume 319 of {\em London
		Mathematical Society Lecture Note Series}.
	\newblock Cambridge University Press, Cambridge, 2005.
	
	\bibitem[Dij96]{dijkhuizen96}
	M.~S. Dijkhuizen.
	\newblock Some remarks on the construction of quantum symmetric spaces.
	\newblock {\em Acta Appl. Math.}, 44(1-2):59--80, 1996.
	\newblock Representations of Lie groups, Lie algebras and their quantum
	analogues.
	
	\bibitem[DKM03]{dkm03}
	J.~Donin, P.~P. Kulish, and A.~I. Mudrov.
	\newblock On a universal solution to the reflection equation.
	\newblock {\em Lett. Math. Phys.}, 63(3):179--194, 2003.
	
	\bibitem[Dun97]{dunn97}
	G.~Dunn.
	\newblock Lax operad actions and coherence for monoidal {$N$}-categories,
	{$A_\infty$} rings and modules.
	\newblock {\em Theory Appl. Categ.}, 3:No. 4, 50--84, 1997.
	
	\bibitem[EGNO15]{egno15}
	P.~Etingof, S.~Gelaki, D.~Nikshych, and V.~Ostrik.
	\newblock {\em Tensor categories}, volume 205 of {\em Math. Surv. and Monogr.}
	\newblock A.M.S., Providence, RI, 2015.
	
	\bibitem[ES13]{es13}
	M.~{Ehrig} and C.~{Stroppel}.
	\newblock {Nazarov-Wenzl algebras, coideal subalgebras and categorified skew
		Howe duality}.
	\newblock {\em ArXiv e-prints}, 2013.
	
	\bibitem[Fie]{fiedorowicz}
	Z.~Fiedorowicz.
	\newblock The symmetric bar construction.
	\newblock Unpublished note.
	
	\bibitem[{Gal}]{calindo16}
	C.~{Galindo}.
	\newblock {Coherence for monoidal $G$-categories and braided $G$-crossed
		categories}.
	\newblock https://arxiv.org/abs/1604.01679.
	
	\bibitem[Gin15]{ginot15}
	G.~Ginot.
	\newblock Notes on factorization algebras, factorization homology and
	applications.
	\newblock In {\em Mathematical aspects of quantum field theories}, Math. Phys.
	Stud., pages 429--552. Springer, Cham, 2015.
	
	\bibitem[GTZ14]{gtz12}
	G.~Ginot, T.~Tradler, and M.~Zeinalian.
	\newblock Higher {H}ochschild homology, topological chiral homology and
	factorization algebras.
	\newblock {\em Comm. Math. Phys.}, 326(3):635--686, 2014.
	
	\bibitem[Gwi12]{gwilliam12}
	O.~Gwilliam.
	\newblock {\em Factorization {A}lgebras and {F}ree {F}ield {T}heories}.
	\newblock ProQuest LLC, Ann Arbor, MI, 2012.
	\newblock Thesis (Ph.D.)--Northwestern University.
	
	\bibitem[Hor17]{horel17}
	G.~Horel.
	\newblock Factorization homology and calculus {\it \`a la} {K}ontsevich
	{S}oibelman.
	\newblock {\em J. Noncommut. Geom.}, 11(2):703--740, 2017.
	
	\bibitem[JM11]{jordanma}
	D.~Jordan and X.~Ma.
	\newblock Quantum symmetric pairs and representations of double affine {H}ecke
	algebras of type {$C^\vee C_n$}.
	\newblock {\em Selecta Math. (N.S.)}, 17(1):139--181, 2011.
	
	\bibitem[Jor09]{jordan09}
	D.~Jordan.
	\newblock Quantum {$D$}-modules, elliptic braid groups, and double affine
	{H}ecke algebras.
	\newblock {\em Int. Math. Res. Not. IMRN}, (11):2081--2105, 2009.
	
	\bibitem[Joy]{joyalnotes}
	A.~Joyal.
	\newblock The theory of quasi-categories and its applications.
	\newblock In preparation.
	
	\bibitem[JS93]{js93}
	A.~Joyal and R.~Street.
	\newblock Braided tensor categories.
	\newblock {\em Adv. Math.}, 102(1):20--78, 1993.
	
	\bibitem[Kel89]{kelly89}
	G.~M. Kelly.
	\newblock Elementary observations on {$2$}-categorical limits.
	\newblock {\em Bull. Austral. Math. Soc.}, 39(2):301--317, 1989.
	
	\bibitem[Kol14]{kolb14}
	S.~Kolb.
	\newblock Quantum symmetric {K}ac-{M}oody pairs.
	\newblock {\em Adv. Math.}, 267:395--469, 2014.
	
	\bibitem[{Kol}17]{kolb17}
	S.~{Kolb}.
	\newblock Braided module categories via quantum symmetric pairs.
	\newblock https://arxiv.org/abs/1705.04238, 2017.
	
	\bibitem[Koo90]{koornwinder89}
	T.~H. Koornwinder.
	\newblock Orthogonal polynomials in connection with quantum groups.
	\newblock In {\em Orthogonal polynomials ({C}olumbus, {OH}, 1989)}, volume 294
	of {\em NATO Adv. Sci. Inst. Ser. C Math. Phys. Sci.}, pages 257--292. Kluwer
	Acad. Publ., Dordrecht, 1990.
	
	\bibitem[Koo91]{koornwinder91}
	T.~H. Koornwinder.
	\newblock The addition formula for little {$q$}-{L}egendre polynomials and the
	{${\rm SU}(2)$} quantum group.
	\newblock {\em SIAM J. Math. Anal.}, 22(1):295--301, 1991.
	
	\bibitem[Koo92]{koornwinderpoly}
	T.~H. Koornwinder.
	\newblock {A}skey-{W}ilson polynomials for root systems of type {$BC$}.
	\newblock In {\em Hypergeometric functions on domains of positivity, {J}ack
		polynomials, and applications ({T}ampa, {FL}, 1991)}, volume 138 of {\em
		Contemp. Math.}, pages 189--204. Amer. Math. Soc., Providence, RI, 1992.
	
	\bibitem[KS92]{sklyanin92}
	P.~P. Kulish and E.~K. Sklyanin.
	\newblock Algebraic structures related to reflection equations.
	\newblock {\em J. Phys. A}, 25(22):5963--5975, 1992.
	
	\bibitem[KS09]{kolbstokman}
	S.~Kolb and J.V. Stokman.
	\newblock Reflection equation algebras, coideal subalgebras, and their centres.
	\newblock {\em Selecta Math. (N.S.)}, 15(4):621--664, 2009.
	
	\bibitem[Let99]{letzter99}
	G.~Letzter.
	\newblock Symmetric pairs for quantized enveloping algebras.
	\newblock {\em J. Algebra}, 220(2):729--767, 1999.
	
	\bibitem[Let02]{letzter02}
	G.~Letzter.
	\newblock Coideal subalgebras and quantum symmetric pairs.
	\newblock In {\em New directions in {H}opf algebras}, volume~43 of {\em Math.
		Sci. Res. Inst. Publ.}, pages 117--165. Cambridge Univ. Press, Cambridge,
	2002.
	
	\bibitem[Let03]{letzter03}
	G.~Letzter.
	\newblock Quantum symmetric pairs and their zonal spherical functions.
	\newblock {\em Transform. Groups}, 8(3):261--292, 2003.
	
	\bibitem[Lur09a]{lhtt}
	J.~Lure.
	\newblock {\em Higher Topos Theory}.
	\newblock Princeton University Press, 2009.
	
	\bibitem[Lur09b]{lurietft}
	J.~Lurie.
	\newblock On the classification of topological field theories.
	\newblock In {\em Current developments in mathematics, 2008}, pages 129--280.
	Int. Press, Somerville, MA, 2009.
	
	\bibitem[Lur17]{lha}
	J.~Lurie.
	\newblock Higher algebra.
	\newblock Available at \url{http://www.math.harvard.edu/~lurie/papers/HA.pdf},
	2017.
	
	\bibitem[Mac01]{macdonald01}
	I.~G. Macdonald.
	\newblock Orthogonal polynomials associated with root systems.
	\newblock {\em S\'em. Lothar. Combin.}, 45:Art.\ B45a, 40, 2000/01.
	
	\bibitem[Mac97]{maclane97}
	S.~MacLane.
	\newblock {\em Categories for the Working Mathematician}.
	\newblock Springer-Verlag, 2nd edition edition, 1997.
	\newblock Graduate Texts in Mathematics, Vol. 5.
	
	\bibitem[Maj95]{majid}
	S.~Majid.
	\newblock {\em Foundations of quantum group theory}.
	\newblock Cambridge University Press, Cambridge, 1995.
	
	\bibitem[Mud02]{mudrov02}
	A.~Mudrov.
	\newblock Characters of {$\mathcal{U}_q(gl(n))$}-reflection equation algebra.
	\newblock {\em Lett. Math. Phys.}, 60(3):283--291, 2002.
	
	\bibitem[Mud06]{mudrov06}
	A.~Mudrov.
	\newblock On quantization of the {S}emenov-{T}ian-{S}hansky {P}oisson bracket
	on simple algebraic groups.
	\newblock {\em Algebra i Analiz}, 18(5):156--172, 2006.
	
	\bibitem[MW12]{mw12}
	S.~Morrison and K.~Walker.
	\newblock Blob homology.
	\newblock {\em Geom. Topol.}, 16(3):1481--1607, 2012.
	
	\bibitem[NDS97]{nds}
	M.~Noumi, M.~S. Dijkhuizen, and T.~Sugitani.
	\newblock Multivariable {A}skey-{W}ilson polynomials and quantum complex
	{G}rassmannians.
	\newblock In {\em Special functions, {$q$}-series and related topics
		({T}oronto, {ON}, 1995)}, volume~14 of {\em Fields Inst. Commun.}, pages
	167--177. Amer. Math. Soc., Providence, RI, 1997.
	
	\bibitem[Nou96]{noumi}
	M.~Noumi.
	\newblock Macdonald's symmetric polynomials as zonal spherical functions on
	some quantum homogeneous spaces.
	\newblock {\em Adv. Math.}, 123(1):16--77, 1996.
	
	\bibitem[NS95]{nousu95}
	M.~Noumi and T.~Sugitani.
	\newblock Quantum symmetric spaces and related {$q$}-orthogonal polynomials.
	\newblock In {\em Group theoretical methods in physics ({T}oyonaka, 1994)},
	pages 28--40. World Sci. Publ., River Edge, NJ, 1995.
	
	\bibitem[Sch14]{scheimbauer}
	C.I. Scheimbauer.
	\newblock {\em Factorization Homology as a Fully Extended Topological Field
		Theory}.
	\newblock PhD thesis, ETH Z\"{u}rich, 2014.
	
	\bibitem[Skl88]{sklyanin88}
	E.~K. Sklyanin.
	\newblock Boundary conditions for integrable quantum systems.
	\newblock {\em J. Phys. A}, 21(10):2375--2389, 1988.
	
	\bibitem[tD98]{td98}
	T.~tom Dieck.
	\newblock Categories of rooted cylinder ribbons and their representations.
	\newblock {\em J. Reine Angew. Math.}, 494:35--63, 1998.
	
	\bibitem[tDHO98]{tdho98}
	T.~tom Dieck and R.~H{\"a}ring-Oldenburg.
	\newblock Quantum groups and cylinder braiding.
	\newblock {\em Forum Math.}, 10(5):619--639, 1998.
	
	\bibitem[Wah01]{wahl01}
	N.~Wahl.
	\newblock {\em Ribbon braids and related operads}.
	\newblock PhD thesis, University of {O}xford, 2001.
	\newblock http://www.maths.ox.ac.uk/research/theses/.
	
	\bibitem[Wee1]{thesis}
	T.A.N. Weelinck.
	\newblock PhD thesis, University of {E}dinburgh.
	\newblock To appear.
	
	\bibitem[Wee2]{EFH}
	T.A.N. Weelinck.
	\newblock Equivariant factorization homology of global quotient orbifolds.
	\newblock In preparation.
	
	\bibitem[Xic97]{xicot97}
	M.~Xicot\'encatl.
	\newblock On configuration spaces of orbits of points and their loop space
	homology.
	\newblock {\em Morfismos}, 1(1):13--26, 1997.
	
\end{thebibliography}
\end{document}